\documentclass[a4,10pt]{article}

\usepackage[utf8]{inputenc}
\usepackage[english]{babel}

\usepackage{amsthm}
\usepackage{amsmath}
\usepackage{amsfonts}
\usepackage{amssymb}

\usepackage{mathrsfs}  % needed for \mathscr command

\usepackage{enumerate}

\usepackage{anysize}

\usepackage[left= 2.8cm, right=2.8 cm, top=2.5 cm, bottom= 2.5 cm]{geometry}

\newtheorem{theo}{Theorem}
\numberwithin{theo}{section}
\newtheorem{defi}[theo]{Definition}
\newtheorem{lemm}[theo]{Lemma}
\newtheorem{prop}[theo]{Proposition}
\newtheorem{cor}[theo]{Corollary}
\newtheorem{bem}[theo]{Remark}

\theoremstyle{definition}

\newcommand{\supp} {\text{supp }}
\newcommand{\diag}{{\text{\rm diag}}}
\newcommand{\argmin}{\mbox{argmin}}
\DeclareMathOperator{\sgn}{\rm sgn}

\usepackage{hyperref}
\newcommand{\N}{\mathbb{N}}
\newcommand{\dimension}{\N\backslash\{1\}}
\newcommand{\Z}{\mathbb{Z}}
\newcommand{\R}{\mathbb{R}}
\newcommand{\C}{\mathbb{C}}
\newcommand{\sph}{\mathbb{S}}
\newcommand{\cE}{{\cal E}}

\newcommand{\angles}[1]{\langle #1\rangle}
\newcommand{\sprod}[1]{\langle #1\rangle}
\newcommand{\abs}[1]{\left\vert #1 \right\vert}
\newcommand{\norm}[1]{\left\Vert #1 \right\Vert}
\newcommand{\bimapnorm}[1]{{ \left\vert\kern-0.25ex\left\vert\kern-0.25ex\left\vert #1 \right\vert\kern-0.25ex\right\vert\kern-0.25ex\right\vert }}
\newcommand{\otnorm}[1]{\abs{#1}_{[d-1]}}

\newcommand{\set}[1]{\left\lbrace #1\right\rbrace}
\newcommand{\abspi}[1]{\abs{\set{#1}}}

\newcommand{\afunc}{a_{\lambda}}
\newcommand{\bfunc}{b_{\mu}}
\newcommand{\gfunc}{g_{\lambda}}

\newcommand{\afuncfor}{\hat{a}_{\lambda}}
\newcommand{\bfuncfor}{\hat{b}_{\mu}}
\newcommand{\gfuncfor}{\hat{g}_{\lambda}}

\newcommand{\sigmafuncfor}{\hat{\gamma}_{j,\ell,k}}

\newcommand{\matricol}[1]{A^{-1}_{\alpha,s_{#1}}R_{\theta_{#1}} R_{\varphi_{#1}}}
\newcommand{\rotmatr}[1]{R_{\theta_{#1}} R_{\varphi_{#1}}}
\newcommand{\rotmatrT}[1]{R_{\varphi_{#1}}^T R_{\theta_{#1}}^T}

\newcommand{\wfuncrad}[4]{\frac{\min\left\{1, s_{#1}^{-1}(1+|\xi|)\right\}^{#2}}{\left( 1+ s_{#1}^{-1}|\xi|\right)^{#3}(1+s_{#1}^{-\alpha}\otnorm{R_{\theta_{#1}}R_{\varphi_{#1}}\xi})^{#4}}}
\newcommand{\Sfunc}[4]{\frac{\min\left\{1, s_{#1}^{-1}(1+|\xi|)\right\}^{#2} }{(1 +s_\lambda^{-1}|\xi|)^{#3}(1+s_{#1}^{1-\alpha}\otnorm{R_{\theta_{#1}} R_{\varphi_{#1}}(\xi/|\xi|)})^{#4}}}
\newcommand{\SfuncRad}[3]{\frac{\min\left\{1, s_{#1}^{-1}(1+r)\right\}^{#2} }{(1 +s_{#1}^{-1}r)^{#3}}}

\begin{document}

\title{Multivariate $\alpha$-Molecules}

\author{Axel Flinth \and Martin Sch\"{a}fer \thanks{ Both authors are affiliated to the Institut f\"ur Mathematik, Technische Universit\"at Berlin. \newline 
\emph{E-mail addresses:} flinth@math.tu-berlin.de, schafer@math-tu.berlin.de}}

\maketitle

\begin{abstract}

The suboptimal performance of wavelets with regard to the approximation of multivariate data
gave rise to new representation systems, specifically designed for data with anisotropic features.
Some prominent examples of these are given by ridgelets, curvelets, and shearlets, to name a few.

The great variety of such so-called directional systems motivated the search for a common framework, which unites many under one roof
and enables a simultaneous analysis, for example with respect to approximation properties.
Building on the concept of parabolic molecules, the recently introduced framework of $\alpha$-molecules does in fact include the previous mentioned systems.
Until now however it is confined to the bivariate setting, whereas
nowadays one often deals with higher dimensional data. This
motivates the extension of this unifying theory to dimensions larger than 2, put forward in this work.
In particular, we generalize the central result that the cross-Gramian of any two systems of $\alpha$-molecules will to some extent be localized.

As an exemplary application, we investigate the sparse approximation of video signals, which are instances of 3D data.
The multivariate theory allows us to derive almost optimal approximation rates for a large class of representation systems.


\end{abstract}

%------------------------------------------------------
\section{Introduction}\label{sec:intro}
%------------------------------------------------------

\footnote{The final publication has been published in Journal of Approximation Theory. It can be found via its DOI \href{http://www.dx.doi.org/doi:10.1016/j.jat.2015.10.004}{10.1016/j.jat.2015.10.004}. }One of the most influential modern developments in applied harmonic analysis has undeniably been the introduction of \emph{wavelets} \cite{Dau92}. Their construction is based on dilations and translations of
a (finite) set of generating functions $(g_\lambda)_\lambda\subseteq L^2(\R^d)$. By carefully choosing the generators, the resulting systems can become frames or even orthonormal bases of the space $L^2(\R^d)$.
Furthermore, additional properties can be obtained, such as e.g.\ smoothness or compact support. Some real-world applications of wavelets today are e.g.\ data compression (e.g.\ JPEG2000)
or restoration tasks in imaging sciences~\cite{CDOS12}. In the field of PDE's wavelets nowadays play a central role in solving elliptic equations~\cite{CDD01}.

The great success of wavelet systems -- besides their elegant construction principle and available fast numerical implementations -- rests upon the fact that
they provide efficient multiscale representations for various types of data. In particular, they optimally sparsely approximate
functions $f:\R^d\rightarrow\C$, which are smooth apart from (a finite number of) point singularities, in the sense of fast decay of the $N$-term approximation error. Since such singularities
are the only ones that occur in `reasonable' 1D data, we can safely say that wavelets are optimal for approximating one-dimensional functions.

Moving up a dimension however, the situation changes completely. Two-dimensional data may well have singularities along curves and
is often governed by such anisotropic features -- think of edges in images for instance. A widely used model for such data is the class of cartoon-like functions~\cite{Don01}, i.e.\
functions which are smooth except for a curve-like singularity (see Section \ref{sec:videodata} for details). For this class, wavelets do not perform optimally any more~\cite{Don01}, and hence other approaches have to be considered.

\subsection{Directional Representation Systems}

The reason for the non-optimal performance of wavelets in a multivariate setting is due to their isotropic scaling law, which is
not optimally suited for resolving anisotropic, i.e.\ directional, features.
Therefore, many systems employing some form of anisotropic scaling and thus better suited for this task have been considered in recent years. Subsequently, we briefly recall
a few of these constructions for motivation purposes, but by no means this shall be a complete overview.

\subsubsection{Ridgelets}

Aiming to approximate functions with line singularities, Candès defined so-called \emph{ridgelets}~\cite{Can98} as translated, rotated and dilated versions of a ridge function, which
is constant orthogonal to some specified direction $\eta \in \sph^{d-1}$. Since such functions are not square-integrable, the concept was adjusted by Donoho~\cite{CD99} allowing ridgelets
a slow decay orthogonal to the $\eta$-direction, leading to a modified notion adopted for instance in~\cite{GrohsRidLT, GKKSaSPIE2014, GKKScurve2014}. In the new sense, a system of
ridgelets is constructed performing rotations, translations and directional scaling on a generator $g\in  L^2(\R^d)$ with corresponding scaling matrix
	\begin{align*}
		A_{0,s} = \begin{pmatrix}
		s & 0 \\ 0 & 1
		\end{pmatrix}, \qquad s>0.
\end{align*}
Tight ridgelet frames of this type were constructed e.g.\ in \cite{GrohsRidLT, GKKScurve2014}.

\subsubsection{Curvelets and Shearlets}

A true breakthrough was achieved by Candès and Donoho in 2002 with the introduction of \emph{curvelets}~\cite{CD04},
the first system to provide a provably (almost) optimal approximation rate for a certain class
of cartoon-like functions.
Again, the idea is to apply certain rotation, translation and scaling operations to a generating function. The major novelty was the use of \emph{parabolic scaling},
a compromise between directional scaling used for ridgelets and isotropic scaling used for wavelets, described by a matrix of the form
\begin{align}\label{eq:paramat}
		A_{\frac{1}{2},s} = \begin{pmatrix}
		s & 0 \\ 0 & s^{1/2}
		\end{pmatrix}, \qquad s>0.
\end{align}
This type of scaling is specifically adapted to data with $C^2$-discontinuity curves, since it leaves the parabola invariant and produces functions
with essential support in a rectangle of size {\em `width $\approx$ length$^2$'}.
We mention that in the actual construction of the classical tight frame of curvelets~\cite{CD04}, the translations and rotations are applied to a set of generators,
related to each other by a parabolic scaling law realised not by \eqref{eq:paramat} but by dilations with respect to polar coordinates.

A few years after curvelets in 2005, \emph{shearlets} were developed mainly by Kutyniok, Labate, Lim, and Weiss~\cite{KuLaLiWe}.
They also scale parabolically and feature the same celebrated approximation properties as curvelets for cartoon-like functions~\cite{GKL05,Kutyniok2010}.
The main difference is that shearings and not rotations are used for the change of direction. The choice of shears makes shearlets more adapted to a digital grid, since shearings given by the matrices
	\begin{align}\label{eq:shearmat1}
		S_{h} = \begin{pmatrix}
		1 & 0 \\ h & 1
		\end{pmatrix}
\quad\text{and}\quad
S^T_{h} = \begin{pmatrix}
		1 & h \\ 0 & 1
		\end{pmatrix},\quad h\in\R,
	\end{align}
leave the digital grid invariant. This is favorable in a discrete setting and
bears the advantage of a unified treatment of the continuum and digital realm. It should be noted that
some actual constructions of shearlet systems are not entirely faithful to the original idea of applying shears, translations, and parabolic scalings to a single generator:
%\begin{align*}
%\added{\psi_{j,\ell,k}(x) = 2^{3j/4} \psi(S_\ell A_{1/2, 2^j}x -k).}
%\end{align*}
The probably most notable and widely used adjustment is the idea of \emph{cone-adaption}, where several generators with different orientations are used in order to avoid large shear parameters. We will discuss this strategy in greater detail later in the article.

Nowadays, shearlets are a widely used directional representation system with
applications ranging from imaging science~\cite{EL2012}, simulations of inverse scattering problems \cite{KMP14}
to solvers for transport equations \cite{DHKSW14}. For more information we refer to the book \cite{KL12}.

\subsubsection{$\alpha$-Curvelets (and -Shearlets)}

The systems we have presented so far all utilize different versions of the scaling matrix
\begin{align}\label{eq:alphamat1}
A_{\alpha,s}=\begin{pmatrix}
s & 0 \\ 0 & s^\alpha
\end{pmatrix}, \qquad s>0,
\end{align}
where the parameter $\alpha\in[0,1]$ specifies the degree of anisotropy in the scaling:
$\alpha=1$ corresponds to wavelets, $\alpha=\frac{1}{2}$  to curvelets and shearlets, and $\alpha=0$ to ridgelets.
This observation was used in \cite{GKKScurve2014} to define \emph{$\alpha$-curvelets}, and associated bandlimited tight frames were constructed for every $\alpha\in[0,1]$.
Similar to $\alpha$-curvelets, the notion of a shearlet can be generalized to comprise $\alpha$-scaling. The resulting
\emph{$\alpha$-shearlets} have been defined and examined in \cite{Kei13,Kutyniok2012} (for the range $\alpha\in[\frac{1}{2},1)$).

\subsection{A common Framework}

The directional systems described above are all constructed using the same idea: start with a set of generators, and then perform scalings (with some degree of anisotropicity), changes of direction (e.g. rotations or shears) and translations. Further, in order to obtain systems with desirable properties, some regularity condition on the generators has to be posed. Having this in mind, it seems
possible to regard all such systems as certain instances of a more general concept.

\subsubsection{Parabolic Molecules}

In 2011, Grohs and Kutyniok introduced the concept of \emph{parabolic molecules}~\cite{Grohs2011}, which allows
to derive classical curvelets and shearlets as special instances of the same general construction process.
Starting from a set of generators, a system of parabolic molecules is obtained via parabolic dilations, rotations, and translations.
The essential novelty is that the generators can, apart from a certain time-frequency localization, be chosen freely and
each function may have its own generator. Together with the utilization of so-called parametrizations to allow generic indexing,
the `variability' of the generators provides the flexibility to cast rotation- and shear-based systems as instances of one unifying construction principle.
Moreover, it becomes possible to relax the
vanishing moment conditions -- important for high approximation rates -- imposed on the generators.
Rather to demand rigid conditions as in most classical constructions,
it suffices to require the moments of the variable generators to vanish asymptotically at high scales, without changing the asymptotic approximation behavior.

\subsubsection{$\alpha$-Molecules}

The scope of parabolic molecules is limited to parabolically scaled systems, wherefore
a major generalization was pursued in~\cite{GKKS2014}, namely the extension to \emph{$\alpha$-molecules}. These
incorporate more general $\alpha$-scaling~\eqref{eq:alphamat1}
and can thus bridge the gap between wavelets and ridgelets, as well as curvelets and shearlets in between.

However, like the framework of parabolic molecules, they are confined to a 2-dimensional setting.
Since nowadays higher dimensional data plays an ever increasing role, an extension of the theory to higher dimensions is appreciable.
A first step in this direction was taken by one of the authors \cite{Flinth13} with an extension of the parabolic molecules framework to 3D. In this paper, we aim to generalize the framework to arbitrary dimensions $d \in \N$, $d \geq 2$ and  general scaling parameters $\alpha \in [0,1]$.

\subsubsection{Why $\alpha$-Molecules?}

The concept of (multivariate) $\alpha$-molecules covers a great variety of directional multiscale systems
and unifies their treatment and analysis, e.g.\ with respect to approximation properties.
The foundational result behind this is Theorem~\ref{thm:almorth}, i.e.\ the fact that the localization of the cross-Gramian
of two systems of $\alpha$-molecules -- in the sense of a strong off-diagonal decay -- merely
depends on their respective parametrizations and orders. Hence, the parametrization and the order
of a system of $\alpha$-molecules alone is sufficient information to determine the corresponding approximation behavior.
This is illustrated by Theorem~\ref{thm:videoapprox}, where a large class of directional representation systems is specified
with almost optimal approximation performance with respect to cartoon video data.
Finally, we remark that apart from the analysis aspects the framework also promises new design approaches for novel constructions.

\subsection{Our Contributions}
As mentioned before, the goal of this paper is to generalize the concept of $\alpha$-molecules to arbitrary dimensions. The multivariate formulation extends the earlier results from \cite{Grohs2011,GKKS2014,Flinth13} and gives valuable insights, e.g.\ on how they scale with the dimension.
One should emphasize that the extension beyond dimension 2 comes with several delicacies such as to determine a suitable definition of the so-called $\alpha$-scale index
distance. The technical effort to prove the mentioned results is considerably higher when  dealing with more than 2 dimensions, mainly because many arguments take place on the unit sphere instead of the unit circle.  It also gets significantly harder to prove that shearlets, i.e. the main example of multidimensional directional systems in dimensions higher than 2, can be included in the framework. Hence already the core results of the theoretical framework themselves do not generalize straightforwardly.

\subsection{Outline}

The paper is organized as follows.
The core part of the theory, in particular Theorem~\ref{thm:almorth},
is presented in Section \ref{sec:molecules}.
Since the corresponding proof is quite involved it is outsourced to Section~\ref{sec:proofmain}.
The abstract theory is further developed in Section~\ref{sec:sparseapprox} with a focus on approximation theory.
Here Theorem~\ref{thm:sparseeqivmol} derives
sufficient conditions for two systems of $\alpha$-molecules to be sparsity equivalent.
In Section~\ref{sec:videodata} we then exemplarily apply the theory to video data and identify a large class of representation systems
providing almost optimal sparse approximation in Theorem~\ref{thm:videoapprox}.
Finally, Section~\ref{sec:shearletmols} is devoted to a large class of concrete systems of $\alpha$-molecules, namely multivariate $\alpha$-shearlet molecules.
As specific examples we present pyramid-adapted shearlet systems, in particular those generated by compactly supported functions
and the smooth Parseval frame of band-limited shearlets by Guo and Labate. This frame is presented in greater detail in Subsection~\ref{sssec:ssh}, thereby fixing some inaccuracies
of the original definition.

\subsection{Notation}

The (strictly) positive real numbers are denoted by $\R_{+}$.
The vector space $\R^d$, where $d\in\N$, is equipped with the usual Euclidean scalar product denoted by $\sprod{\cdot, \cdot}$. $\N$ hereby denotes the set of positive integers, while $\N_0 = \N \cup \set{0}$.
For the unit sphere in $\R^d$ the symbol $\mathbb{S}^{d-1}$ is used. The standard unit vectors are given by $e_1,\ldots,e_d$ and for a vector $x\in\R^d$ we use the notation
$[x]_i:=\langle x,e_i \rangle$, $i\in\{1,\ldots,d\}$, for the $i$:th component.
Its $p$-(quasi-)norm in the range $0<p\le\infty$ is denoted by $|x|_p$. In case of the Euclidean norm $|x|_2=\sqrt{\langle x,x \rangle}$,
we will usually omit the subindex. We further define $\otnorm{x} := |([x]_1,\ldots,[x]_{d-1},0)^T|_2$. For a matrix $M \in \R^{m,n}$, we denote its operator norm as a mapping from the Euclidean $\R^n$ to the Euclidean $\R^m$ by $\norm{M}_{2 \to 2}$, and its entries by $[M]_i^j$, $i = 1, \dots, m$, $j=1, \dots n$.

The usual Lebesgue spaces on $\R^d$ are denoted by $L^p(\R^d)$, where $0<p\le\infty$. The corresponding sequence spaces are given by $\ell^p(\Lambda)$, where $\Lambda$ is a countable index set.
In both cases we use the symbol $\|\cdot\|_p$ for the associated (quasi-)norms. Further, the symbol $\langle\cdot,\cdot\rangle$ will also be used
for the inner products on the Hilbert spaces $L^2(\R^d)$ and $\ell^2(\Lambda)$.
For the weak versions of the sequence spaces we use the notation $\omega\ell^p(\Lambda)$ with associated (quasi-)norms $\|\cdot\|_{\omega\ell^p}$. For their definition we
refer to Subsection~\ref{ssec:sparseapprox}.

In addition, we need the following function spaces on $\R^d$: the space of
continuous functions $C(\R^d)$, the space of $n$-times continuously differentiable functions $C^n(\R^d)$ for $n\in\N\cup\{\infty\}$, as well as their respective restrictions
$C_c(\R^d)$ and $C^n_c(\R^d)$ to functions with compact support.

The Fourier transform $\hat{f}$ of a function $f$ in the space of Schwarz functions $\mathscr{S}(\R^d)$ is given by
\begin{align*}
	\hat{f}(\xi) = \int_{\R^d} f(x) \exp(-2\pi i \sprod{\xi, x})dx.
\end{align*}
As usual, it extends to the space of tempered distributions $\mathscr{S}^\prime(\R^d)$.

For two entities $x,y$, usually dependent on a certain set of parameters, the notation
`$x\lesssim y$' shall mean that $x\le Cy$ for some fixed constant $C>0$, which is independent of the involved parameters.
If both $x \lesssim y$ and $y \lesssim x$ we write `$x\asymp y$'.
We further need the ceiling function on $\R$ given by $\lceil x \rceil:=\min\{ \ell\in\Z:\ell\ge x\}$.
A useful abbreviation is also the ubiquitous `analyst's bracket' defined by $\angles{x} := \sqrt{1+x^2}$ for $x\in\R$.

%%%%%%%%%%%%%%%%%%%%%%%%%%%%%%%%%%%%%%%%%%%%%%%%%%%%%%%%%%%%%%%%%%%%%%%%%%%%%%%%%%%%%%%%%%%%%%%%%%%%%%%%%%%%%%%%%%%%%%

\section{$\alpha$-Molecules in $d$ Dimensions} \label{sec:molecules}

Recalling the definition in \cite{GKKS2014}, a system of bivariate $\alpha$-molecules consists of functions in $L^2(\R^2)$ obtained by applying $\alpha$-scaling,
rotations, and translations to a set of generating functions, which need to be sufficiently localized in time and frequency.
Due to this construction, every $\alpha$-molecule is naturally associated with a certain scale, orientation and spatial position,
which -- in the $2$-dimensional case -- is conveniently represented by a point in the corresponding parameter space
$\mathbb{P}_2=\R_{+} \times \mathbb{S}^{1}  \times \R^2$.

Aiming for a multivariate generalization, we thus first need a $d$-dimensional version of this parameter space.
We let $\mathbb{S}^{d-1}$ denote the unit sphere in $\R^d$ and put
\begin{align*}
\mathbb{P}_d= \R_{+} \times \mathbb{S}^{d-1}  \times \R^d.
\end{align*}
Each function $m_{\lambda}\in L^2(\R^d)$
of a system of $d$-dimensional $\alpha$-molecules $(m_{\lambda})_{\lambda\in \Lambda}$ shall by definition
then be associated with a unique point $(s_\lambda,e_\lambda,x_\lambda)\in\mathbb{P}_d$,
where the variable $s_\lambda\in\R_+$ shall represent its scale, the vector $e_\lambda\in\mathbb{S}^{d-1}$
its orientation in $\R^d$, and $x_\lambda\in\R^d$ the spatial location.
The relation between the index $\lambda$ of a molecule $m_\lambda$ and its position $(s_\lambda,e_\lambda,x_\lambda)$
in $\mathbb{P}_d$ is described by a so-called parametrization, analogue to \cite{GKKS2014}.

\begin{defi}
    A \emph{parametrization} consists of a
    pair $(\Lambda,\Phi_\Lambda)$,
    where $\Lambda$ is a discrete index set and $\Phi_\Lambda$ is a mapping
    $$
        \Phi_\Lambda:\left\{\begin{array}{ccc}\Lambda &\to & \mathbb{P}_d,\\
        \lambda \in\Lambda & \mapsto & \left(s_\lambda , e_\lambda , x_\lambda\right).
        \end{array}\right.
    $$
    which associates with each $\lambda\in \Lambda$ a \emph{scale} $s_\lambda\in\R_+$,
    a \emph{direction} $e_\lambda\in\mathbb{S}^{d-1}$, and a \emph{location} $x_\lambda\in\R^d$.
\end{defi}

For practical purposes it is more convenient to represent an orientation $\eta\in\mathbb{S}^{d-1}$ by a set of angles.
Therefore we define the rotation matrix $R_{\theta}$ for $\theta=(\theta_1,\ldots,\theta_{d-2}) \in \R^{d-2}$ by
\begin{align*}
	R_{\theta} = \begin{pmatrix}
		\cos(\theta_1) &   & -\sin(\theta_1)  \\
		                           & I_{d-2} & \\
		                           \sin(\theta_1) &  & \cos(\theta_1)
	\end{pmatrix} \cdot {\ldots} \cdot \begin{pmatrix}
		\cos(\theta_{d-2}) &   &-\sin(\theta_{d-2})  &     \\
		                           & 1 &  &   \\
		                           \sin(\theta_{d-2}) &   & \cos(\theta_{d-2})  &   \\
		                           & & & I_{d-3} 	
	\end{pmatrix},
\end{align*}
where $I_{d}$ for $d\in\N$ denotes the $d$-dimensional identity matrix.
Furthermore, we introduce for $\varphi\in\R$ the matrix
\begin{align*}
	  R_{\varphi} = \begin{pmatrix}
	\cos(\varphi) & \sin( \varphi) &\\
	-\sin( \varphi) & \cos(\varphi) &  \\
	&  & I_{d-2}
	\end{pmatrix}.
\end{align*}
Note that these definitions pose an inconsistency in the notation, since they depend on the particular naming of the index.
However, since we always use these particular indices, this will not lead to any problems while improving the readability significantly.

Each orientation $\eta\in \mathbb{S}^{d-1}$ can now be uniquely represented by a set of angles $(\theta_1,\ldots,\theta_{d-2},\varphi)\in  [0, \pi] \times [-\frac{\pi}{2},\frac{\pi}{2}]^{d-3} \times [0, 2\pi]$
via the relation
\begin{align}\label{eq:angleparamet}
\eta=R^T_\varphi R^T_\theta  e_d,
\end{align}
where $e_d$ is the $d$th unit vector of $\R^d$.
Explicitly, it is given by
\begin{align*}
	\eta(\theta,\varphi) = \begin{pmatrix} \eta_1(\theta,\varphi)  \\ \vdots \\ \eta_d(\theta,\varphi)  \end{pmatrix}
     = \begin{pmatrix}
	  \phantom{+}  \cos(\varphi) \cos(\theta_{d-2}) \cdots \cdots\cos(\theta_{2})\sin(\theta_1)  \\
	  \phantom{+}  \sin(\varphi) \cos(\theta_{d-2}) \cdots \cdots\cos(\theta_{2})\sin(\theta_1)  \\
	  - \sin(\theta_{d-2}) \cos(\theta_{d-3}) \cdots \cos(\theta_{2})\sin(\theta_{1}) \\
			\vdots\\
      - \sin(\theta_3)\cos(\theta_2)\sin(\theta_1) \\
      - \sin(\theta_2)\sin(\theta_1) \\
		\cos(\theta_1)
	\end{pmatrix}.
	\end{align*}
Next, we adapt the $\alpha$-scaling matrix~\eqref{eq:alphamat1} to the multivariate setting. For $\alpha\in[0,1]$, we set
\begin{align}\label{eq:alphamat2}
A_{\alpha,s}= \begin{pmatrix}
	  s^\alpha I_{d-1} &  \\
	     &  s
	\end{pmatrix}.
\end{align}
In case $\alpha=1$ this matrix scales $\R^d$ isotropically, in the range $\alpha\in[0,1)$ it scales uniformly in all directions except for
the $e_d$-direction.
Note that here we choose $e_d$ as the distinguished direction in which the scaling is stronger -- in contrast to the 2-dimensional case \cite{GKKS2014},
where $e_1$ was chosen.

After this preparation we are ready to give the definition of $d$-variate $\alpha$-molecules, $d\in\N, d \geq 2$,
which essentially reduces to the original definition from \cite{GKKS2014} for $d=2$, except for the interchanged roles of the directions $e_1$ and $e_d$.

\begin{defi}
Let $\alpha\in[0,1]$, $d\in\N$, $d \geq 2$ and $L,M,N_1,N_2\in\N_0\cup\{\infty\}$. Further let $(\Lambda,\Phi_\Lambda)$ be a parametrization with $\Phi_\Lambda(\lambda)=(s_\lambda,e_\lambda,x_\lambda)\in\mathbb{P}_d$
for $\lambda\in\Lambda$. The corresponding angles \eqref{eq:angleparamet} for $e_\lambda$ shall be denoted by $(\theta_\lambda,\varphi_\lambda)$.
A family of functions $(m_{\lambda})_{ \lambda \in \Lambda} \subseteq L^2(\R^d)$
is called a \emph{system of $d$-dimensional $\alpha$-molecules} of order $(L,M,N_1,N_2)$ with respect to the parametrization $(\Lambda,\Phi_\Lambda)$, if each $m_\lambda$ is of the form
\begin{align} \label{eq:molgen}
	m_{\lambda} =s_\lambda^{\frac{1+\alpha (d-1)}{2}} g_{\lambda}\big( A_{\alpha,s_{\lambda}}R_{\theta_{\lambda}} R_{\varphi_{\lambda}}(x-x_{\lambda})\big)
\end{align}
with generators $g_{\lambda}\in L^2(\R^d)$ satisfying for every multi-index $\rho\in\N_0^d$ with $|\rho|_1\leq L$ the condition
\begin{align}\label{eq:molcon}
	\abs{\partial^{\rho}\hat{g}_{\lambda}(\xi)} \lesssim \min\left( 1, s_\lambda^{-1} + \abs{[\xi]_d} + s_\lambda^{-(1-\alpha)} \otnorm{\xi}\right)^M \angles{\abs{\xi}}^{-N_1}\angles{\otnorm{\xi}}^{-N_2}.
\end{align}
The implicit constant in \eqref{eq:molcon} is required to be uniform in $\Lambda$.
In case that a control parameter takes the value $\infty$, this shall mean that the condition \eqref{eq:molcon}
is fulfilled with the respective quantity arbitrarily large.
\end{defi}

A system of $\alpha$-molecules is thus obtained by applying rotations, translations, and $\alpha$-scaling
to a set of generating functions $(g_{\lambda})_\lambda$, which are required to obey a prescribed time-frequency localization.
Every molecule $m_\lambda$ is thereby allowed to have its own individual generator $g_{\lambda}$.

The definition only poses conditions on the Fourier transform of the generators $g_{\lambda}$.
The number $L$ describes the spatial localization, $M$ the number of directional (almost)
vanishing moments, and $N_1,N_2$ the smoothness of an element $m_\lambda$.
Also note that the weighting function on the right hand side of \eqref{eq:molcon}
is symmetric with respect to rotations around the $e_d$-axis, as well as reflections along this axis.

Applying $A_{\alpha,s}$ with $\alpha<1$ and $s>1$ to the unit ball $B=\{ x\in\R^d:|x|\le1 \}$ stretches $B$ in the $e_d$-direction. This
results in plate-like support of the characteristic function $\chi_B(A_{\alpha,s}\cdot)$ for large $s\in\R_+$, with the `plate' lying in the plane
spanned by the vectors $\{e_1,\ldots,e_{d-1} \}$. Thus, at high scales $\alpha$-molecules can be thought of as plate-like objects in the spatial domain.
The approximate frequency support on the other hand is concentrated in a pair of opposite
cones in the direction of the respective orientation.

\begin{bem}
It may seem more natural to choose a rotation $R_\eta$ from $e_d$ to $\eta\in\mathbb{S}^{d-1}$ in the $(e_d,\eta)$-plane to adjust the orientation
in \eqref{eq:molgen}. Due to the symmetries of the weighting function of the generators,
this choice is however not necessary. Since it is easier to use fixed rotation planes,
we stick to this more pragmatic choice of rotation parameters.
\end{bem}

Let us conclude this paragraph with some comments on the use of the term `molecule'. In the theory of function spaces, `atoms'
originally refer to the basic building blocks of a function space. In the widest sense, an atomic decomposition of
a function space is a countable subset containing functions called `atoms', which allow to represent every function of the space
as a countable linear combination.

In the theory of atomic decomposition of so called \emph{Hardy spaces} $H^p$, an atom is defined as a function $a$ supported in some cube $Q$ possessing vanishing moments and satisfying some norm bound, e.g. $\norm{a}_2 \leq \abs{Q}^{\frac{1}{2}-\frac{1}{p}}$\cite{wilson1985}. This definition has also been adapted in a slightly less rigid form in other parts of mathematics. Here, the term ''atom'' often simply refers to a function having compact support and possessing many vanishing moments. Typically, these atoms also possess additional properties, e.g.\ they may be bounded or fulfill smoothness conditions.

Furthermore, in many instances, the atoms are also related to each other.
Classical coorbit spaces e.g.\ possess atomic decompositions with atoms obtained from the action of a group on a single generator~\cite{FeGr89a}.
A primary example are the (homogeneous) Besov-Triebel-Lizorkin spaces with (compactly supported) wavelets as atoms.

%in the sense that they are obtained from each other by applying group transformations.
%Functions that have many vanishing moments and exhibits rapid decay are then called `molecules'.

A system of $\alpha$-molecules somewhat resembles this structure, whereas in a relaxed form.
The compact support condition is replaced with a less restrictive decay condition, and the moments are only asymptotically vanishing.
In addition, the molecules are related to each other via certain group transformations, namely translations, rotations, and dilations.
However, this relation shall not be understood in a strict sense, since the generators are allowed to vary to some extent. They just need to fulfill
a uniform localization condition.
For the functions of a system featuring these kind of soft conditions the term `molecules' was coined and used e.g.\ in \cite{CD02,GL08,Grohs2011,GKKS2014}.

\subsection{Index Distance in $d$ Dimensions}
\label{ssec:index}

A central ingredient of the theory of bivariate $\alpha$-molecules \cite{GKKS2014} is the fact that the parameter space $\mathbb{P}_2$ can be equipped with a natural
(pseudo-)metric with the property, that the distance between two points in $\mathbb{P}_2$ `anti-correlates' with the size of the scalar products of
$\alpha$-molecules associated with those points:
The greater the distance between two indices$\lambda, \lambda' \in \mathbb{P}_2$,
the smaller the scalar product of the corresponding $\alpha$-molecules $m_\lambda, m_{\lambda'} \in L^2(\R^d)$.

Our next aim is to find a suitable analogon of this (pseudo-)metric for the parameter space $\mathbb{P}_d$.
As for $\mathbb{P}_2$, the distance between two points $(s_\lambda,e_\lambda,x_\lambda),\,(s_\mu,e_\mu,x_\mu)\in\mathbb{P}_d$
must certainly take into account their spatial, scale, and orientational relation.
The spatial distance is measured by a rightly balanced combination of an isotropic term
$|x_\lambda-x_\mu|^2$ and a non-isotropic component $|\langle e_\lambda,x_\lambda-x_\mu \rangle|$, which depends on the orientation of the molecules.
For the distance between the orientations $e_\lambda$ and $e_\mu$, it seems natural to
consider the angle $d_{\sph}(e_\lambda, e_\mu)=\arccos( \sprod{e_\lambda, e_\mu})$ with $d_{\sph}(e_\lambda, e_\mu)\in [0,\pi]$.
Due to the symmetries of the weighting function in \eqref{eq:molcon} however,
the angle $d_{\sph}(e_\lambda, e_\mu)$ is projected onto the interval $[-\pi/2, \pi/2)$, with
the projected angle $\{d_\sph(e_\lambda, e_\mu)\}$ being the unique element of the set $\set{d_{\sph}(e_\lambda, e_\mu) + n \pi \ \vert \ n \in \Z}$ in the interval $[-\pi/2, \pi/2)$.
A suitable measure for the orientational distance is then $|\{d_{\sph}(e_\lambda, e_\mu)\}|^2$.
This definition is in fact consistent with the one in \cite{GKKS2014}, since in two dimensions we have $|\{d_{\sph}(e_\lambda, e_\mu)\}|= |\{\varphi_\lambda - \varphi_\mu\}|$.
Finally, the distance between different scales $s_\lambda,\,s_\mu>0$ is measured by the ratio $\max\left\{ s_\lambda/s_\mu, s_\mu/s_\lambda\right\}$.

Altogether, this leads to the following definition. It directly generalizes the
metric introduced in \cite{GV2014}, which is a simplified version of the original metric from \cite{GKKS2014}.

\begin{defi}
Let $\alpha\in[0,1]$. For given parametrizations $(\Lambda,\Phi_\Lambda)$ and $(\Delta,\Phi_\Delta)$ with
$(s_\lambda,e_\lambda,x_\lambda)=\Phi_\Lambda(\lambda)$ and $(s_\mu,e_\mu,x_\mu)=\Phi_\Delta(\mu)$, the
\emph{$\alpha$-scaled index distance} $\omega_\alpha:\Lambda\times\Delta\rightarrow[1,\infty)$ is defined by
\begin{align*}
	\omega_{\alpha}( \lambda, \mu ) = \max\left\{ \frac{s_\lambda}{s_\mu}, \frac{s_\mu}{s_\lambda}\right\} \left( 1 + d_\alpha(\lambda, \mu) \right),\qquad \lambda\in\Lambda,\,\mu\in\Delta,
\end{align*}
where with $s_0 = \min\{s_{\lambda},s_{\mu}\}$
\begin{align*}
	d_\alpha(\lambda, \mu) = s_0^{2\alpha} \abs{x_\lambda - x_\mu}^2 + s_0^{2(1-\alpha)} \abspi{d_\sph(e_\lambda, e_\mu)}^2  +  s_0 \abs{\sprod{e_{\lambda}, x_{\lambda} - x_{\mu}}}.
\end{align*}
\end{defi}

Let $(\Lambda,\Phi_\Lambda)$ be a parametrization. Then the induced index distance on $\Lambda$ is pseudosymmetric and satisfies a pseudo triangle inequality.
More precisely, it has the following properties:
\begin{itemize}
\item[(i)] $\omega_\alpha(\lambda,\lambda)=1$ for all $\lambda\in\Lambda$,
\item[(ii)] $\omega_\alpha(\lambda,\lambda^\prime)\asymp \omega_\alpha(\lambda^\prime,\lambda)$ for all $\lambda,\lambda^\prime\in\Lambda$,
\item[(iii)] $\omega_\alpha(\lambda,\lambda^\prime)\lesssim \omega_\alpha(\lambda,\lambda^{\prime\prime})\omega_\alpha(\lambda^{\prime\prime},\lambda^{\prime})$ for all $\lambda,\lambda^\prime,\lambda^{\prime\prime}\in\Lambda$.
\end{itemize}
Hence the function $\omega_\alpha$ can be viewed as a kind of multiplicative pseudo-metric.
A proof of these properties for the 2-dimensional case can be found in \cite{GV2014}, which translates very well to higher dimensions.

Now we are in a position to formulate the main theorem of this paper. It states that the index distance -- in a certain sense --
measures the size of the scalar products of $\alpha$-molecules.

\begin{theo} \label{thm:almorth}
Let $\alpha\in[0,1]$, $d\in\N$, $d\geq 2$, and let $(m_\lambda)_{\lambda\in\Lambda}$ and $(p_\mu)_{\mu\in\Delta}$ be two systems of $d$-dimensional $\alpha$-molecules of order $(L,M,N_1,N_2)$ with respect to parametrizations $(\Lambda,\Phi_\Lambda)$ and $(\Delta,\Phi_\Delta)$, respectively.
Further assume that there exists some constant $c>0$ such that
\[
s_\lambda\ge c \quad\text{and}\quad s_\mu\ge c \quad\text{for all }\lambda\in\Lambda,\mu\in\Delta \text{ with } (s_\lambda,e_\lambda,x_\lambda)=\Phi_\Lambda(\lambda),
(s_\mu,e_\mu,x_\mu)=\Phi_\Delta(\mu).
\]
If $N_1>\frac{d}{2}$ and if there exists some positive integer $N\in\N$ such that
\begin{align*}
 L \geq 2N, \quad  M >3N -d  + \frac{1+\alpha(d-1)}{2}, \quad  N_1 \geq N +\frac{1+\alpha(d-1)}{2}, \quad N_2 \geq 2N + d-2,
\end{align*}
then we have
\begin{align*}
	\abs{\sprod{m_\lambda, p_{\mu}}} \lesssim \omega_{\alpha}(\lambda, \mu)^{-N}.
\end{align*}
\end{theo}

The proof of Theorem~\ref{thm:almorth} is very long and technical and for this reason not presented here but in Section~\ref{sec:proofmain}.
Let us instead discuss the significance of this result. It states that -- with appropriate assumptions on the parametrizations -- the cross-Gramian of two systems of $\alpha$-molecules
is well-localized, in the sense of a fast off-diagonal decay
with respect to the index distance $\omega_\alpha$.
Put differently, the matrix is close to a diagonal matrix and the corresponding systems are almost orthogonal to each other.
This property has many implications, see for instance~\cite{GV2014,Grohs2012}. Its significance with respect to sparsity equivalence is elaborated in the next section. \newline

%%%%%%%%%%%%%%%%%%%%%%%%%%%%%%%%%%%%%%%%%%%%%%%%%%%%%%%%%%%%%%%%%%%%%%%%%%%%%%%%%%%%%%%%%%%%%%%%%

%------------------------------------------------------%
\section{Sparse Approximation with $\alpha$-Molecules} \label{sec:sparseapprox}
%------------------------------------------------------%

Based on Theorem~\ref{thm:almorth} it is possible to develop a general methodology to categorize frames of $\alpha$-molecules according to their sparse approximation behavior.
A central concept in this context is the notion of sparsity equivalence.

%Given a frame of $\alpha$-molecules, naturally the question of associated smoothness spaces arises. 
Another question that naturally arises in this context is if it is possible develop a theory of smoothness spaces associated with frames of $\alpha$-molecules. For such an investigation, coorbit space theory
provides an appropriate abstract framework. Usually however, due to the lack of group structure, this question
can not be handled within the classical theory developed in \cite{FeGr89a,FeGr86,FeGr89b}.
In subsequent contributions, among others \cite{dastte04,dastte04-1,daforastte08}, the classical theory has seen significant extensions beyond the group setting.
In fact, it is possible to base the theory solely on the notion of an abstract continuous frame, see \cite{fora05,RaUl10,Sch12,KemSchUl15}. In this general setup
coorbits associated to frames of $\alpha$-molecules can be defined and investigated.

Up to now, this has only been carried out to some extent and not yet systematically, i.e., for certain special instances of $\alpha$-molecules. %An exception
A particular example are cone-adapted bivariate shearlet systems \cite{labate2013shearlet}.
The authors believe that the development of a theory of $\alpha$-molecule smoothness spaces is beyond the scope of this paper, but certainly a very interesting possibility for future research. In this paper, we concentrate on sparse approximations.

%-----------------------------------------------------
\subsection{Sparse Approximation and Sparsity Equivalence}\label{ssec:sparseapprox}
%-----------------------------------------------------

Let us briefly recall some aspects of approximation theory in a general (separable) Hilbert space $(\mathcal{H},\langle\cdot,\cdot\rangle)$.
Utilizing a system $(m_\lambda)_{\lambda\in \Lambda}\subseteq \mathcal{H}$, a signal $f\in \mathcal{H}$ can be represented by the coefficients $c_\lambda\in\C$
 of the expansion
 \begin{align}\label{eq:expansion}
     f = \sum_{\lambda\in \Lambda} c_\lambda m_\lambda.
 \end{align}

 Suitable representation systems are provided e.g.\ by so-called frame systems \cite{Christensen2003a}, which
 ensure stable measurement of the coefficients and also stable reconstruction.
 A system $(m_\lambda)_{\lambda\in\Lambda}$ in $\mathcal{H}$ forms a \emph{frame} if there exist constants $A,B>0$,
 called the \emph{frame bounds}, such that
 \[
 A\|f\|^2 \le \sum_{\lambda\in\Lambda} |\langle f,m_\lambda \rangle|^2 \le B\|f\|^2 \text{ for all }f\in \mathcal{H}.
 \]
 If $A$ and $B$ can be chosen equal, the frame is called \emph{tight}. In case $A=B=1$, one speaks of a \emph{Parseval frame}.
 The associated frame operator $S:\mathcal{H}\rightarrow \mathcal{H}$ is given by $Sf=\sum_{\lambda\in\Lambda} \langle f,m_\lambda \rangle m_\lambda$.

 Since $S$ is always invertible, the system $(S^{-1}m_\lambda)_\lambda$ is also a frame, referred to as the \emph{canonical dual frame}.
 It can be used to compute a particular sequence of coefficients in the expansion \eqref{eq:expansion} via
 \[
 c_\lambda=\langle f,S^{-1}m_\lambda \rangle, \quad \lambda\in\Lambda.
 \]
 This sequence however is usually not the only one possible. Unlike the expansion in a basis, a representation with respect to a frame
 need not be unique. The canonical dual frame can also be used to express $f$ in terms
 of the {\em frame coefficients} $(\langle f,m_\lambda \rangle)_\lambda$ by
 \[
 f = \sum_{\lambda\in \Lambda} \langle f,m_\lambda \rangle S^{-1} m_\lambda.
 \]
 In general, any system $(\tilde{m}_\lambda)_{\lambda\in \Lambda}$ satisfying this reconstruction formula is called an associated {\em dual frame}.

Let us turn to the question of efficient encoding. In practice we have to restrict to finite expansions \eqref{eq:expansion}, which usually leads to an approximation error.
Given a positive integer $N$, the best $N$-term approximation $f_N$ of some element $f\in\mathcal{H}$
with respect to the system $(m_\lambda)_\lambda$ is defined by
$$
    f_N = \argmin \Big\Vert f -  \sum_{\lambda \in \Lambda_N}c_\lambda m_\lambda \Big\Vert^2 \quad\mbox{s.t.}\quad \#\Lambda_N \le N.
$$
For efficient approximation, it is desirable to find representation systems which provide good sparse approximation for the considered data, in the sense that
the approximation error $\|f-f_N\|$ decays quickly for $N\rightarrow\infty$.
Typically one wants to approximate signals in some subclass $\mathcal{C}\subseteq \mathcal{H}$. The approximation performance of a system with respect to such a class
is then usually judged by the worst-case scenario, i.e.\ the worst possible decay rate of the error $\|f-f_N\|$ for $f\in\mathcal{C}$.
In this sense a system $(m_\lambda)_\lambda$ provides {\em optimally
sparse approximations} with respect to $\mathcal{C}$, if its worst-case approximation rates are the best among all systems.

It is common to consider not the best $N$-term approximation but
the $N$-term approximation, obtained by keeping the $N$ largest coefficients.
This approximation is better understood and provides a bound for the best $N$-term approximation error.
We will also denote it by $f_N$, since the context will always make the meaning clear.

The $N$-term approximation rate achieved by a frame is closely related to the decay of the corresponding frame coefficients,
often measured by a strong or weak $\ell^p$-(quasi-)norm with $p>0$.
The weak $\ell^p$-(quasi-)norm is defined by
\[
\|(c_\lambda)_\lambda\|_{\omega\ell^p}:= \Big( \sup_{\varepsilon>0} \varepsilon^p \cdot \#\{\lambda: |c_\lambda|>\varepsilon \}\Big)^{1/p},
\]
and by definition a sequence $(c_\lambda)_\lambda \in \omega \ell^p$ if $\norm{(c_\lambda)_\lambda}_{\omega\ell^p}<\infty$. Since
$\norm{(c_\lambda)_\lambda}_{\omega\ell^p}\le \norm{(c_\lambda)_\lambda}_{p}$ for every sequence $(c_\lambda)_\lambda$, we have the embedding 
$\ell^p\hookrightarrow \omega\ell^p$. Note also that every non-increasing rearrangement $(c^*_n)_{n\in\N}$ of a sequence $(c_\lambda)_\lambda\in\omega\ell^p$ satisfies
\begin{align*}
\sup_{n>0} n^{1/p}|c^\ast_n| = \|(c_\lambda)_\lambda\|_{\omega\ell^p}.
\end{align*}

The well-known result below (see \cite{Devore1998}), whose proof can be found e.g.\ in \cite{GKKS2014}, shows that
membership of the expansion coefficients in an $\ell^p$ space for small $p>0$ implies good $N$-term approximation
rates.

 \begin{lemm}[{\cite[Lemma 5.1]{GKKS2014}}]\label{lem:decayapprox}
 Let $(m_\lambda)_{\lambda\in \Lambda}$ be a frame in $\mathcal{H}$ and $f=\sum c_\lambda m_\lambda$ an expansion of
 $f\in \mathcal{H}$ with respect to this frame. If $(c_\lambda)_\lambda\in \omega\ell^{2/(p+1)}(\Lambda)$ for some $p>0$,
 then the $N$-term approximation rate for $f$ achieved by keeping the $N$ largest coefficients is at least of order $N^{-p/2}$, i.e.
 \[
  \| f-f_N \|_2^2 \lesssim N^{-p}.
 \]
 In particular, the error of best $N$-term approximation decays at least with order $N^{-p/2}$.
 \end{lemm}

As illustrated by Lemma~\ref{lem:decayapprox} the decay rate of the frame coefficients determines the $N$-term approximation rate.
In particular, if the sequence $\left(\left\langle f, m_\lambda\right\rangle\right)_{\lambda\in \Lambda}$
of frame coefficients lies in $\ell^p$
for $p<2$,
the best approximation rate of the dual frame
$\left(\tilde{m}_\lambda\right)_{\lambda\in\Lambda}$
is at least of order $N^{-(1/p -1/2)}$.

Let us now assume that we have two frames $(m_\lambda)_{\lambda\in \Lambda}$ and $(p_\mu)_{\mu \in \Delta}$ in the Hilbert space $\mathcal{H}$ and expansion
coefficients for $f\in \mathcal{H}$ with respect to these two systems. Then these frames provide the same $N$-term approximation rate
for $f$, if the corresponding expansion coefficients have similar decay, e.g.\ if they belong to the same $\ell^p$-space.
We recall \cite[Proposition 5.2]{GKKS2014} and formulate this result in an abstract Hilbert space setting.

\begin{prop}[{\cite[Proposition 5.2]{GKKS2014}}]
Let $p\in(0,2)$, and let $(m_\lambda)_{\lambda\in \Lambda}$ and $(p_\mu)_{\mu \in \Delta}$ be frames in a Hilbert space $\mathcal{H}$ and
$(\tilde{m}_\lambda)_{\lambda\in \Lambda}$
a dual frame for $(m_\lambda)_{\lambda\in \Lambda}$ such that
\[
\left\|\left(\langle  m_{\lambda},p_\mu\rangle\right)_{\lambda\in \Lambda, \mu\in \Delta}\right\|_{\ell^p\to \ell^p}<\infty.
\]
Then $(\langle f,\tilde{m}_\lambda \rangle)_\lambda \in \ell^p(\Lambda)$
implies $(\langle f,p_\mu \rangle)_\mu \in \ell^p(\Delta)$.
In particular, $f\in\mathcal{H}$ can be encoded by the $N$ largest
frame coefficients from $(\langle f,p_\mu \rangle)_\mu$ up to accuracy $\lesssim N^{-(1/p-1/2)}$.
\end{prop}

This proposition motivates the following notion of sparsity equivalence initially introduced in~\cite{Grohs2011} for
parabolic molecules.

\begin{defi}[{\cite[Definition 4.2]{Grohs2011}}]
Let  $p\in(0,\infty]$, and let $(m_\lambda)_{\lambda\in \Lambda}$ and $(p_\mu)_{\mu \in \Delta}$ be
frames in a Hilbert space $\mathcal{H}$. Then $(m_\lambda)_{\lambda\in \Lambda}$ and $(p_\mu)_{\mu \in \Delta}$
are {\em sparsity equivalent in $\ell^p$}, if
\[
\left\|\left(\langle {m}_{\lambda},p_\mu\rangle\right)_{\lambda\in \Lambda, \mu\in \Delta}\right\|_{\ell^p\to \ell^p}<\infty.
\]
\end{defi}

Sparsity equivalence, as pointed out in \cite{GKKS2014}, is not an equivalence relation.
Nevertheless, it allows to transfer approximation properties from one anchor system to other systems.

%-----------------------------------------------------
\subsection{Consistency of Parametrizations}
%-----------------------------------------------------

Our aim in this section is to categorize frames of $\alpha$-molecules in $L^2(\R^d)$ with respect to their approximation
behavior, building upon the notion of sparsity equivalence.
We emphasize that a system of $\alpha$-molecules does not per se constitute a frame.
In fact, the question if a system of functions in $L^2(\R^d)$ is a frame is decoupled from the question if it forms a system of $\alpha$-molecules.

Theorem~\ref{thm:sparseeqivmol} will provide sufficient conditions for two frames of $\alpha$-molecules to be sparsity equivalent,
based upon the notion of ($\alpha,k$)-consistency originally introduced in \cite{GKKS2014}.
To motivate this concept, we recall a simple estimate from \cite{Grohs2011} for the
operator norm of a matrix on discrete $\ell^p$ spaces.

\begin{lemm}[{\cite[Lemma 4.4]{Grohs2011}}]\label{lem:matnorm}
Let $\Lambda,\Delta$ be two discrete index sets, and let ${\bf A} : \ell^p(\Lambda)\to \ell^p(\Delta)$, $p>0$ be a linear mapping
defined by its matrix representation ${\bf A} = \left(A_{\lambda,\mu}\right)_{\lambda\in \Lambda ,\, \mu\in \Delta}$.
Then we have the bound
\[
\|{\bf A}\|_{\ell^p(\Lambda)\to \ell^p(\Delta)}\le\max\left\{\sup_{\lambda}\sum_{\mu}|A_{\lambda,\mu}|^{\min\{1,p\}},\sup_{\mu}\sum_{\lambda}|A_{\lambda,\mu}|^{\min\{1,p\}}\right\}^{1/\min\{1,p\}}.
\]

\end{lemm}
%
 %
%     \begin{proof}
%     \deleted{The proof for $p<1$ follows easily from the fact that
%     %
%     $
%         |x + y|^p\le |x|^p + |y|^p
%     $
%     %
%     for $x,y\in\mathbb{R}$.
%     To show the case $p\geq 1$ one proves the assertion for $p=1$ and $p=\infty$. The claim then follows by interpolation.}
% \end{proof}
 %

We apply this lemma to the cross-Gramian ${\bf A}$ of two systems of $\alpha$-molecules,
and aim for sufficient conditions for the right hand side to be finite.
The following notion was introduced in \cite{GKKS2014}.

\begin{defi}[{\cite[Definition 5.5]{GKKS2014}}]
Let $\alpha \in [0,1]$ and $k > 0$. Two parametrizations $(\Lambda,\Phi_\Lambda)$ and $(\Delta,\Phi_\Delta)$
are called \emph{($\alpha,k$)-consistent}, if
\begin{align} \label{eq:consCond}
\sup_{\lambda\in \Lambda}\sum_{\mu\in\Delta}\omega_{\alpha}\left(\lambda,\mu\right)^{-k} <\infty \quad\text{and}\quad \sup_{\mu\in \Delta}\sum_{\lambda\in\Lambda}
\omega_{\alpha}\left(\lambda,\mu\right)^{-k} <\infty.
\end{align}
\end{defi}
\begin{bem} \label{rem:consCond}
 Note that due to symmetry of the distance function $\omega_\alpha$ it suffices to check only one of the two conditions in equation~\eqref{eq:consCond} to prove $(\alpha, k)$-consistency.
\end{bem}

In view of Theorem~\ref{thm:almorth}, the consistency of the parametrizations of two systems of $\alpha$-molecules provides a convenient sufficient condition for their sparsity equivalence.

\begin{theo}\label{thm:sparseeqivmol}
Let $\alpha \in [0,1]$, $d\in\N$, $d \geq 2$, $k > 0$, and $p\in(0,\infty]$.  Let $(m_\lambda)_{\lambda\in\Lambda}$ and $(p_\mu)_{\mu\in\Delta}$ be two frames of $d$-dimensional
$\alpha$-molecules of order $(L,M,N_1,N_2)$ with ($\alpha,k$)-consistent parametrizations $(\Lambda,\Phi_\Lambda)$ and $(\Delta,\Phi_\Delta)$ satisfying
\[
s_\lambda\ge c, s_\mu\ge c \text{ for all }\lambda\in\Lambda,\mu\in\Delta
\]
and with $q:=\min\{1,p\}$
\[
L\geq 2\frac{k}{q},\quad  M> 3\frac{k}{q} - d + \frac{1+\alpha(d-1)}{2} ,\quad N_1>\frac{d}{2},\: N_1\geq \frac{k}{q}+\frac{1+\alpha(d-1)}{2} , \quad N_2\geq 2\frac{k}{q}+d-2.
\]
Then $(m_\lambda)_{\lambda\in\Lambda}$ and $(p_\mu)_{\mu\in\Delta}$ are sparsity equivalent in $\ell^p$.
\end{theo}
\begin{proof}
Let $q:=\min\{1,p\}$. By Lemma \ref{lem:matnorm}, it suffices to prove that
\[
 \max \left\{ \sup_{\lambda\in\Lambda}\sum_{\mu\in \Delta} |\langle m_{\lambda},p_\mu\rangle|^q, \sup_{\mu\in\Delta}\sum_{\lambda\in \Lambda}
|\langle m_{\lambda},p_\mu\rangle|^q\right\}^{1/q} <\infty.
\]
Since, by Theorem \ref{thm:almorth}, we have
$
|\langle m_{\lambda},p_\mu\rangle|\lesssim \omega_{\alpha}(\lambda,\mu)^{-k/q},
$
we can conclude that
\[
\max \left\{ \sup_{\lambda\in\Lambda}\sum_{\mu\in \Delta}|\langle m_{\lambda},p_\mu\rangle|^q,\sup_{\mu\in\Delta}\sum_{\lambda\in \Lambda}|\langle m_{\lambda},p_\mu\rangle|^q\right\}%
\lesssim\max \left\{ \sup_{\lambda\in\Lambda}\sum_{\mu\in \Delta}\omega_{\alpha}(\lambda,\mu)^{-k},\sup_{\mu\in\Delta}\sum_{\lambda\in \Lambda}\omega_{\alpha}(\lambda,\mu)^{-k}\right\}
\]
with the expression on the right hand side being finite due to the $(\alpha,k)$-consistency of the parame\-tri\-zations $(\Lambda,\Phi_\Lambda)$ and $(\Delta,\Phi_\Delta)$.
The proof is completed.
\end{proof}

This theorem allows to categorize frames of $\alpha$-molecules according to their sparse approximation behavior.
The general strategy is as follows.  If an approximation result for a specific system of $\alpha$-molecules is known
and a class of $\alpha$-molecules satisfies the hypotheses of Theorem \ref{thm:sparseeqivmol}, i.e.\ they are all sparsity equivalent to this
specific system, they automatically inherit its known approximation behavior.
In this way, one in the end obtains a stand-alone result for frames of $\alpha$-molecules to exhibit sparse approximation, depending solely on the parametrization and the order.

%-----------------------------------------------------
\section{Sparse Approximation of Video Data}\label{sec:videodata}
%-----------------------------------------------------

In this section we demonstrate with
a specific example how the machinery of $\alpha$-molecules can be applied in practice.
In our exemplary  application, we are interested in the sparse representation of video signals, modelled by the class
of cartoon-like functions $\mathcal{E}^2(\R^3)$ introduced below.

Following the general methodology, we first need a suitable anchor system, for which a
sparse approximation result with respect to $\mathcal{E}^2(\R^3)$ is known.
Utilizing Theorem~\ref{thm:sparseeqivmol} the framework can then transfer the approximation rate from this reference system to other systems.
In this way we will be able to identify a large class of
representation systems, which provide almost optimal sparse approximation for $\mathcal{E}^2(\R^3)$.

%-----------------------------------------------------
\subsection{Cartoon-like Functions}
%-----------------------------------------------------

A suitable model for image and video data is provided by the class of cartoon-like functions, first introduced by Donoho \cite{Don01} and later
extended e.g.\ in \cite{Kutyniok2012}. We shall use the following simplified model in dimensions $d=2$ and $d=3$.

\begin{defi}[{\cite{Don01},\cite[Definition 2.1]{Kutyniok2012}}]\label{def:cartoon}
For fixed $\nu>0$ and $d\in\{2,3\}$ the \emph{class $\mathcal{E}^2(\R^d)$ of cartoon-like functions} consists of functions $f:\R^d\rightarrow\C$ of the form
\[
f=f_0+f_1\chi_B,
\]
where $B\subset[0,1]^d$ and $f_i\in C^2(\R^d)$ with $\supp f_i\subset [0,1]^d$ and $\|f_i\|_{C^2}\le 1$ for each $i=0,1$.
For dimension $d=2$, we assume that the boundary $\partial B$ is a closed $C^2$-curve with curvature bounded by $\nu$,
and, for $d=3$, the discontinuity $\partial B$ shall be a closed $C^2$-surface with principal curvatures bounded by $\nu$.
\end{defi}

This model is justified by the observation that real-life images and video data typically consist of
smooth regions, separated by piecewise smooth boundaries. Without loss of generality we restrict in Definition~\ref{def:cartoon}
to cartoon-like functions with smooth boundaries.

%------------------------
\subsection{Optimal Approximation for Cartoon-like Functions}
%------------------------

In \cite{Don01,Kutyniok2012} the optimal approximation rates for $\mathcal{E}^2(\R^2)$ and $\mathcal{E}^2(\R^3)$, achievable with algorithms satisfying a \emph{polynomial depth search} constraint, were derived.
An algorithm for sparse approximation is thereby restricted to polynomial depth search in a given (countable) dictionary,
if there exists a polynomial $\pi$ such that the algorithm only chooses from the first $\pi(N)$ vectors of the dictionary when forming the $N$:th sparse approximation~\cite{Don01}.
Such a constraint is very natural from a practical standpoint and has the nice side-effect that it ensures the existence of the best $N$-term approximation.
In general, this may not be the case, as the instance of a countable dense subset of $L^2(\R^d)$ as a dictionary shows. Here the best $1$-term approximation does not always exist.

We now cite the result of \cite{Don01,Kutyniok2012}.
In \cite{Kutyniok2012} it is conjectured that the result also generalizes to higher dimensions.

\begin{theo}[{\cite[Theorem 7.2]{Don01},\cite[Theorem 3.2]{Kutyniok2012}}]\label{thm:optapprox}
Let $d\in\{2,3\}$. The best $N$-term approximation rate for $\mathcal{E}^2(\R^d)$, achieved by an arbitrary dictionary under the restriction of polynomial depth search, cannot exceed
\[
\|f-f_N\|_2^2 \lesssim N^{ -\frac{2}{d-1}},
\]
where $f_N$ is the best $N$-term approximation of $f\in\mathcal{E}^2(\R^d)$.
\end{theo}

For $d=2$ it has been shown that this rate can indeed be achieved \cite{Don01} using so-called wedgelets, which are adaptive to the data. Moreover, there are several examples of non-adaptive frames in two and three dimensions which almost provide these optimal rates
\cite{CD04,GL07,Kutyniok2010,Guo2010,Kutyniok2012}, typically up to log-terms.
In particular, it was proven by Guo and Labate in \cite{Guo2010} that the smooth Parseval frame of band-limited 3D-shearlets $SH$ constructed by them
in \cite{Guo2012a} sparsely approximates the class $\mathcal{E}^2(\R^3)$ with an almost optimal approximation rate. We recall the definition of this particular system in Subsection~\ref{sssec:ssh}
and state the approximation result below.

\begin{theo}[{\cite[Theorem 3.1]{Guo2010}}]\label{thm:3Dshearapprox}
Let $SH=\{\psi_\lambda\}_{\lambda\in\Lambda_{SH}}$ be the smooth Parseval frame of 3D-shearlets defined in Subsection~\ref{sssec:ssh}.
Then the sequence of shearlet coefficients $\theta_\lambda(f):=\langle f, \psi_\lambda\rangle$, $\lambda\in\Lambda_{SH}$, associated with $f\in \mathcal{E}^2(\R^3)$ satisfies
\[
\sup_{f\in\mathcal{E}^2(\R^3)} |\theta_\lambda(f)|_{N} \lesssim N^{-1}\cdot \log(N),
\]
where $|\theta_\lambda(f)|_{N}$ denotes the $N$:th largest shearlet coefficient.
\end{theo}

Theorem~\ref{thm:3Dshearapprox} shows that the shearlet coefficients belong to $\omega\ell^{p}(\Lambda_{SH})$ for every $p>1$.
In view of Lemma~\ref{lem:decayapprox},  for every $f\in\mathcal{E}^2(\R^3)$, the frame $SH$ therefore provides at least the approximation rate
\begin{align}\label{eq:rateref}
\|f-f_N\|_2^2\lesssim N^{-1+\varepsilon} \quad ,\text{$\varepsilon>0$ arbitrary},
\end{align}
where $f_N$ denotes the $N$-term approximation obtained from the $N$ largest coefficients.
According to Theorem~\ref{thm:optapprox}, this is almost the optimal approximation rate achievable for cartoon-like functions $\mathcal{E}^2(\R^3)$.
For small $\varepsilon>0$, we get arbitrarily close to the optimal rate.

The idea to use 3D-shearlets for processing video data has been tested in practice in~\cite{negi2012}.
The authors of said article develop a clever discretisation procedure, allowing a fast computation of the shearlet coefficients.
They then test how their discrete shearlet transform can be used to denoise and enhance video sequences, with promising results.

%------------------------
\subsection{Transfer of the Approximation Rate  }
%------------------------

 Our final goal is to find a large class of representation systems which achieve the almost optimal rate \eqref{eq:rateref}.
 For this, we put the machinery of $\alpha$-molecules to work.
 Via Theorem~\ref{thm:sparseeqivmol} it is possible to transfer the approximation rate~\eqref{eq:rateref} established for the smooth Parseval frame of $3$D-shearlets $SH$
 to other systems of 3-dimensional $\frac{1}{2}$-molecules. In fact, the frame $SH$ is a suitable choice for the reference system, since by Proposition~\ref{prop:sshmol} it
 constitutes a system of 3-dimensional $\frac{1}{2}$-molecules of order $(\infty,\infty,\infty,\infty)$ with respect to the parametrization $(\Lambda_{SH},\Phi_{SH})$
 The following result is then a direct application of the general theory.

\begin{theo}\label{thm:videoapprox}
Assume that a frame $(m_\lambda)_{\lambda\in \Lambda}$ of $3$-dimensional
$\frac{1}{2}$-molecules satisfies, for some $k>0$, the following two conditions:
\begin{itemize}
\item[(i)] its parametrization $(\Lambda,\Phi_\Lambda)$ is ($\frac{1}{2},k$)-consistent with $(\Lambda_{SH},\Phi_{SH})$,
\item[(ii)] its order $(L,M,N_1,N_2)$ satisfies
\[
L\geq 2k,\quad  M \geq 3k-2,\quad N_1 \geq k + 1,\quad N_1>3/2,\quad N_2\geq 2k+1.
\]
\end{itemize}
Then each dual frame $(\tilde{m}_\lambda)_{\lambda\in \Lambda}$ possesses an almost optimal $N$-term approximation rate for the class
of cartoon-like functions $\cE^{2}(\R^3)$, i.e.\ for all $f\in\mathcal{E}^2(\R^3)$
\[
\|f-f_N\|_2^2 \lesssim N^{-1+\varepsilon}, \quad\varepsilon >0 \text{ arbitrary},
\]
where $f_N$ denotes the $N$-term approximation obtained from the $N$ largest frame coefficients.
\end{theo}
\begin{proof}
Let $SH=\{\psi_\lambda\}_{\lambda\in\Lambda_{SH}}$ be the Parseval frame of $3$D-shearlets from Subsection~\ref{sssec:ssh},
and take $f\in \mathcal{E}^{2}(\R^3)$. By Theorem~\ref{thm:3Dshearapprox} the sequence of shearlet
coefficients $(\theta_\lambda)_\lambda$ given by $\theta_\lambda=\langle f, \psi_\lambda \rangle$ belongs to $\omega\ell^{p}(\Lambda_{SH})$
for every $p>1$. Since $\omega\ell^{p}\hookrightarrow \ell^{p+\epsilon}$ for arbitrary $\varepsilon>0$,
this further implies $(\theta_\lambda)_\lambda \in \ell^{p}(\Lambda_{SH})$ for every $p>1$.
Let now
\[
f=\sum_{\mu\in\Lambda} c_\mu \tilde{m}_\mu
\]
be the canonical expansion of $f$ with respect to the dual frame $(\tilde{m}_\mu)_\mu$, with frame coefficients
$(c_\mu)_\mu$. Note that since $SH$ is a Parseval frame, the canonical dual frame of $SH$ is equal to $SH$ itself. Therefore the coefficents are given by
\[
c_\mu= \langle f, m_\mu\rangle = \Big\langle \sum_{\lambda} \theta_\lambda \psi_\lambda , m_\mu\Big\rangle=\sum_{\lambda} \langle  \psi_\lambda, m_\mu\rangle \theta_\lambda.
\]
 Thus, they are related to the shearlet coefficients $(\theta_\lambda)_\lambda$ by the cross-Gramian $(\langle \psi_\lambda,m_\mu \rangle)_{\lambda,\mu}$.
By Theorem~\ref{thm:sparseeqivmol}, conditions (i) and (ii) guarantee that the frame $(m_\mu)_{\mu\in\Lambda}$ is sparsity equivalent
to $(\psi_\lambda)_{\lambda\in\Lambda_{SH}}$ in $\ell^p$ for every $p>1$. This implies that the cross-Gramian is a bounded operator
$\ell^p(\Lambda_{SH})\rightarrow\ell^p(\Lambda)$, which maps $(\theta_\lambda)_\lambda$ to $(c_\mu)_\mu$. Hence, $(c_\mu)_\mu\in\ell^{p}(\Lambda)$
for every $p>1$. The embedding $\ell^{p}\hookrightarrow \omega\ell^{p}$ then proves $(c_\mu)_\mu\in\omega\ell^{p}(\Lambda)$
for every $p>1$. Finally, for arbitrary $\varepsilon>0$, the application of Lemma~\ref{lem:decayapprox} yields
\[
\|f-f_N\|_2^2 \lesssim N^{-1+\varepsilon},
\]
where $f_N$ denotes the $N$-term approximation with respect to the system $(\tilde{m}_\mu)_\mu$ obtained by
choosing the $N$ largest coefficients.
\end{proof}

Theorem~\ref{thm:videoapprox} specifies a large class of multiscale systems
with almost optimal approximation performance for video data $\mathcal{E}^{2}(\R^3)$.
According to Remark~\ref{rem:consist}
condition~(i) is in particular fulfilled by every $\frac{1}{2}$-shearlet parametrization (see Theorem~\ref{thm:shearletmol}) for $k>3$.
Hence, due to condition~(ii) all systems of 3-dimensional $\frac{1}{2}$-shearlet molecules of order
\[
L\ge 7,\quad  M \ge 8,\quad N_1 \ge 5,\quad N_2\geq 8,
\]
provide almost optimal approximation for $\mathcal{E}^{2}(\R^3)$.

Taking into account Proposition~\ref{prop:compactShearlets}, the statement of Theorem~\ref{thm:videoapprox} in particular includes
the following result for compactly supported shearlet frames.
\begin{cor}
Any dual frame of a shearlet frame of the form \eqref{eq:affineshear} generated by compactly supported functions $\phi,\psi^1,\psi^2,\psi^3\in L^2(\R^3)$, so that
$\phi\in C^{13}(\R^3)$ and so that for each $\varepsilon\in\{1,2,3\}$
\begin{enumerate}[(i)]
\item $\partial^{\gamma}\psi^\varepsilon$ exists and is continuous for every $\gamma=(\gamma_1,\gamma_2,\gamma_3)^T\in\N_0^3$ with $|\gamma|_\infty\leq 13$ and $\gamma_\varepsilon \leq 5$,
\item $\psi^\varepsilon$ has at least 15 vanishing directional moments in direction $e_\varepsilon$,
\end{enumerate}
 provides the almost optimal approximation rate~\eqref{eq:rateref} for the cartoon video class $\mathcal{E}^2(\R^3)$.
\end{cor}

This corollary is a new result for compactly supported shearlets on its own.
A similar result was proved in \cite{Kutyniok2012}. In comparison the most intriguing fact is the simplicity of its deduction:
The framework of $\alpha$-molecules enables a simple transport of the decay rates.

\begin{bem}
The 2-dimensional counterpart of Theorem~\ref{thm:videoapprox} is contained in \cite[Theorem 5.12]{GKKS2014} for the choice $\alpha=\frac{1}{2}$.
In \cite{GKKS2014} it is further shown that not only any system of bivariate $\frac{1}{2}$-shearlet molecules satisfies the conditions
of this theorem, provided that the order is sufficiently high, but also every system of sufficiently high order $\frac{1}{2}$-curvelet molecules.
This implies that also curvelet-like constructions yield almost optimal approximation for the class $\mathcal{E}^2(\R^2)$,
including the classical curvelet frame.

In contrast, a `true' curvelet construction in 3D is not known to the authors.
Note that, despite the misleading name, the construction of the `$3D$ discrete curvelet transform' in \cite{Ying2005} is actually shear-based.
Still, it can be expected that any curvelet-like system would fall into our framework, and thus Theorem~\ref{thm:videoapprox}
would immediately establish the almost optimal approximation rate.  
\end{bem}

%--------------------------------------------------------------------------------%

\section{Shearlet Systems in $d$ Dimensions} \label{sec:shearletmols}

We introduce a very general class of shear-based systems, namely
\emph{systems of $\alpha$-shearlet molecules}. The definition in $d$ dimensions is analogue
to the 2-dimensional case~\cite{GKKS2014}. Roughly speaking, they are shear-based systems obtained from variable generators, where similar to $\alpha$-molecules
the conditions on the generators have been relaxed to a mere time-frequency localization requirement.
The notion of $\alpha$-shearlet molecules comprises many specific shear-based constructions and simplifies the treatment of such systems within the general framework of $\alpha$-molecules.

\subsection{Multidimensional $\alpha$-Shearlet Molecules}

As explained in Section~\ref{sec:intro}, shearlet-like constructions are based on anisotropic scaling, shearings, and translations.
For the change of scale, we utilize $\alpha$-scaling as defined by \eqref{eq:alphamat2}.
The change of orientation is provided by shearings, in $d$ dimensions given by the shearing matrices
\begin{align*}
	S_{h}= \begin{pmatrix}
		I_{d-1} & 0 \\  h^T & 1
\end{pmatrix}
\quad\text{and}\quad
S^T_{h}= \begin{pmatrix}
		I_{d-1} & h  \\ 0 & 1
\end{pmatrix},\quad	h\in\R^{d-1},
\end{align*}
which are the natural generalizations of \eqref{eq:shearmat1}.
The matrix $S^T_{h}$ shears parallel to the $(e_1,\ldots,e_{d-1})$-plane and the
shear vector $h\in \R^{d-1}$ determines the direction of the shearing in this plane.
The transformations associated with shearings and $\alpha$-scalings naturally form a group~\cite{Dahlke_theuncertainty}.

To avoid directional bias, the frequency domain is divided into cone-like regions along the coordinate axes and a coarse-scale box for the low frequencies. Note that this comes at the cost
of losing the group properties mentioned above. This division procedure is however crucial for applications, and also, as the subsequent arguments will show, for including $\alpha$-shearlets in the concept of $\alpha$-molecules.
The cone-like regions along the $e_\varepsilon$-axes shall be called \emph{pyramids} and are explicitly given by
\begin{align*}
	\mathcal{P}_\varepsilon = \set{ (\xi_1, \dots, \xi_d)^T \in \R^d ~\vert~ \forall i\in\{1,\ldots,d\} : \abs{\xi_i} \leq \abs{\xi_\varepsilon} },
\end{align*}
where $\varepsilon \in \set{1,\ldots,d}$. $\varepsilon=0$ shall refer to a
coarse-scale box of the form $\mathcal{R}=\{ \xi\in\R^d : |\xi|_\infty\le C \}$, where $C>0$ is a suitably chosen constant. In the sequel we will always stay in this so-called cone-adapted setting. For an illustration of this specific setting in 3D,
we refer to Subsection~\ref{ssec:3Dshearlets}.

%\begin{figure}
%\centering
%\def\svgwidth{280pt}
%\input{drawing.pdf_tex}
%\caption{An illustration of the pyramid $\mathcal{P}_3$ together with the coarse-scale box $\mathcal{R}$ in the case $d=3$.}
%\end{figure}

In each cone we require different versions of the scaling and shearing operators.
The cyclic permutation matrix
\begin{align} \label{eq:cyclicPerm}
	Z=\begin{pmatrix}
		0 & 1 \\
		I_{d-1} & 0
	\end{pmatrix}
\end{align}
allows to elegantly define these operators associated with the respective cones by
$Z^\varepsilon S_{h} Z^{-\varepsilon}$ and $Z^\varepsilon A_{\alpha,s} Z^{-\varepsilon}$.

Before we come to the definition of $\alpha$-shearlet molecules, we need to introduce a set of characteristic parameters, associated with these systems.
The resolution of the underlying sampling grid is determined by the parameters $\sigma>1$, $\tau_1,\ldots,\tau_d>0$, and a sequence $\Theta=(\eta_j)_{j\in\N_0}\subset\R_+$.
The parameter $\sigma$ specifies the fineness of the scale sampling. The parameters $\tau_\varepsilon$, $\varepsilon\in\{1,\ldots,d\}$, determine
the spatial resolution in the $e_\varepsilon$-direction. For convenience they are summarized in the
diagonal matrix $\mathcal{T} :=\diag(\tau_1,\ldots,\tau_d)\in~\R^{d\times d}$. The angular resolution at each scale $j\in\N_0$ is given by the value $\eta_j$
of the sequence $\Theta$. Last but not least, in each cone $\varepsilon\in\{1,\ldots,d\}$ and at each scale $j\in\N_0$ the shearing parameter $\ell$ is restricted
to a set $\mathscr{L}_{\varepsilon,j}\subseteq \Z^{d-1}$. These sets are collected
in $\mathscr{L}:=\{ \mathscr{L}_{\varepsilon,j} : \varepsilon\in\{1,\ldots,d\},\, j\in\N_0 \}$.

After the introduction of this sampling data $\mathbb{D}:=\{\sigma,\Theta,\mathscr{L},\mathcal{T}\}$ we can now give the definition of a system of $\alpha$-shearlet molecules in $d$ dimensions, depending
on $\mathbb{D}$.
The scale-dependent step size $\eta_j$ of the directional sampling is assumed to satisfy $\eta_j \asymp \sigma^{-j(1-\alpha)}$ for $j\in\N_0$.
Further, we require the upper bounds
$\mathbb{L}_{j}:=\max \big\{ |\ell|_\infty : \ell\in\mathscr{L}_{\varepsilon,j}, \varepsilon\in\{1,\ldots,d\} \big\}$, $j\in\N_0$,
to fulfill the complementary condition $\mathbb{L}_{j} \lesssim \sigma^{j(1-\alpha)}$.
We remark that the translation parameters
$\tau_\varepsilon$ may also vary with the indices $(\varepsilon,j,\ell)$, as long as their values are restricted to some fixed interval ${[}\tau_{min},\tau_{max}{]}$
with $0<\tau_{min}\le\tau_{max}<\infty$. However, this is not indicated in the notation.

\begin{defi} \label{def:shearAlphaMol1}
Let $\alpha\in[0,1]$, $d\in\N$, $d \geq 2$, and $L,M,N_1,N_2\in\N_0\cup\{\infty\}$.
Further the sampling data $\mathbb{D}$ shall be given as above.
For $\varepsilon \in \set{1,\ldots,d}$, a system of functions
\[
\Sigma_\varepsilon:=\Big\{ m^\varepsilon_{j,\ell,k}\in L^2(\R^d) ~:~ (j,\ell,k)\in \Lambda^s_\varepsilon \Big\},
\]
indexed by the set
$
\Lambda^s_\varepsilon:=\big\{ (j,\ell,k) ~:~ j \in \N_0,\, \ell\in\mathscr{L}_{\varepsilon,j}\subseteq\Z^{d-1},\, k\in\Z^d \big\},
$
is called a \emph{system of $d$-dimensional $\alpha$-shearlet molecules of order $(L,M,N_1,N_2)$ associated with the orientation $\varepsilon$}, if it is of the form
\begin{align} \label{eq:sheargen}
	m_{j,\ell,k}^\varepsilon(x) = \sigma^{\frac{(1+\alpha(d-1))j}{2}} \gamma^\varepsilon_{j,\ell,k} \big( Z^{\varepsilon} A^j_{\alpha,\sigma}S_{\ell\eta_j} Z^{-\varepsilon} x - \mathcal{T} k \big)
\end{align}
with generating functions $\gamma^\varepsilon_{j,\ell,k}\in L^2(\R^d)$ satisfying for every $\rho\in\N_0^d$ with $|\rho|_1 \leq L$
\begin{align} \label{eq:shearcon}
	\abs{ \partial^{\rho} \hat{\gamma}^\varepsilon_{j,\ell,k}(\xi) } \lesssim \frac{\min\{1, \sigma^{-j} + \sigma^{-(1-\alpha)j}|Z^{-\varepsilon}\xi|_{[d-1]} + \abs{[Z^{-\varepsilon}\xi]_d} \}^M}{\angles{\abs{\xi}}^{N_1} \angles{\otnorm{Z^{-\varepsilon} \xi}}^{N_2}}.
\end{align}
The implicit constant is required to be uniform over $\Lambda^s_\varepsilon$.
If one of the parameters $L,M,N_1,N_2$ takes the value $\infty$, this shall mean that condition \eqref{eq:shearcon} is fulfilled with the respective quantity arbitrarily large.
\end{defi}

Combining systems of $\alpha$-shearlet molecules of order $(L,M,N_1,N_2)$ for each orientation $\varepsilon\in\set{1,\ldots,d}$ with a system of coarse-scale elements
\begin{align}\label{eq:coarsescalemol}
\Sigma_0:=\Big\{ m^0_{0,{\bf0},k}:= \gamma^0_{0,{\bf0},k}(\cdot - \mathcal{T} k) ~:~ k\in\Z^d \Big\},
\end{align}
where the generators $\gamma^0_{0,{\bf0},k}\in L^2(\R^d)$ fulfill
$| \partial^{\rho} \hat{\gamma}^0_{0,{\bf0},k}(\xi) | \lesssim \angles{\abs{\xi}}^{-N_1} \angles{\otnorm{\xi}}^{-N_2}$ for every $\rho\in\N_0^d$ with $|\rho|_1 \leq L$,
yields a \emph{system of $\alpha$-shearlet molecules} of order $(L,M,N_1,N_2)$. The associated index set is
\begin{align*}
\Lambda^s_0:=\Big\{ (0,{\bf0},k)  : k\in\Z^d \Big\} \subseteq  \N_0\times\Z^{d-1}\times\Z^d.
\end{align*}

\begin{defi}\label{def:shearAlphaMol2}
For each $\varepsilon\in\{1,\ldots,d\}$, let $\Sigma_\varepsilon$ be a system of $\alpha$-shearlet molecules of order $(L,M,N_1,N_2)$ associated with the respective orientation.
Further, let $\Sigma_0$ be a system of coarse-scale scale elements defined as in \eqref{eq:coarsescalemol}. Then the union
\[
\Sigma:= \bigcup_{\varepsilon=0}^d \Sigma_\varepsilon
\]
is called a \emph{system of $\alpha$-shearlet molecules of order $(L,M,N_1,N_2)$}.
The associated \emph{shearlet index set} is given by

$$ \Lambda^s=\{ (\varepsilon,j,\ell,k) :  \varepsilon\in\{0,\ldots,d\},\, (j,\ell,k)\in\Lambda^s_\varepsilon \}. $$

\end{defi}
\noindent
Next, we prove that the system $\Sigma$ is a system of $\alpha$-molecules.

\begin{theo} \label{thm:shearletmol}
Let $\alpha\in[0,1]$ and $\Sigma$ be a system of $\alpha$-shearlet molecules of order  $(L,M,N_1,N_2)$.
Then $\Sigma$ constitutes a system of $\alpha$-molecules of the same order.
The associated \emph{$\alpha$-shearlet parametrization}
$(\Lambda^s,\Phi^s)$ is given by the map $\Phi^s(\lambda)=(s_\lambda,e_\lambda,x_\lambda)\in\mathbb{P}_d$ for $\lambda=(\varepsilon,j,\ell,k)\in \Lambda^s$,
where
\begin{align}\label{eq:shearpara}
	s_\lambda= \sigma^j, \quad   e_{\lambda} = n_\lambda \cdot Z^{\varepsilon} \begin{pmatrix}
	\eta_j \ell \\ 1
	\end{pmatrix}, \quad x_{\lambda} = Z^{\varepsilon} S^{-1}_{\ell\eta_j}A^{-j}_{\alpha,\sigma}Z^{-\varepsilon} \mathcal{T} k,
\end{align}
and $n_\lambda = (1 + \eta_j^2\abs{\ell}_2^2)^{-1/2}$ is a normalization constant.
\end{theo}
In particular for $\varepsilon=0$ we have $s_\lambda=1$, $e_\lambda=e_d$ and $x_\lambda=\mathcal{T} k$ for every $\lambda=(0,0,{\bf0},k)\in\Lambda^s$.

\begin{proof}
Since a finite union of systems of $\alpha$-molecules is itself a system of $\alpha$-molecules,
we can prove this theorem separately for each system $\Sigma_\varepsilon$, $\varepsilon\in\{0,\ldots,d\}$.
For $\Sigma_0$ the statement is obvious.
For the other systems it suffices to give the proof for $\varepsilon=d$, since they are all related by a mere permutation of indices. We subsequently drop the index $\varepsilon$
to simplify the notation and note $Z^\varepsilon=I$ for $\varepsilon=d$.

    For the proof we introduce the index set $\Lambda^{s,d}:=\{(d,j,\ell,k)~:~ (j,\ell,k)\in \Lambda^s_d\}$.
	Let $\lambda=(d,j,\ell,k)\in\Lambda^{s,d}$ and $m_{\lambda}:=m^d_{j,\ell,k}$ the associated $\alpha$-shearlet molecule with corresponding generating function
    $\gamma_\lambda:=\gamma^d_{j,\ell,k}$.
    As usual we denote the angles representing the orientation $e_\lambda$ by $(\theta_\lambda,\varphi_\lambda)$, i.e.\ $e_\lambda=R^T_{\varphi_\lambda}R^T_{\theta_\lambda}e_d$.
    The molecule $m_\lambda$ can clearly be written in the form \eqref{eq:molgen} with respect to the generator
	\begin{align*}
		\gfunc (x) := \gamma_\lambda(  A^j_{\alpha,\sigma}S_{\ell\eta_j} R_{\varphi_\lambda}^TR_{\theta_\lambda}^T A^{-j}_{\alpha,\sigma} x ), \quad x\in\R^d.
	\end{align*}
    It remains to check condition \eqref{eq:molcon} for these functions. On the Fourier side we have
	\begin{align*}
		\gfuncfor( \xi) = \hat{\gamma}_\lambda( A^{-j}_{\alpha,\sigma}S_{\ell\eta_j}^{-T}  \rotmatrT{\lambda}A^j_{\alpha,\sigma} \xi),  \quad \xi\in\R^d.
	\end{align*}
	For $\lambda=(d,j,\ell,k)\in\Lambda^{s,d}$ let us first examine the matrix
    \begin{align}\label{eq:auxmatrix}
    M_\lambda:= S_{\ell\eta_j}^{-T} \rotmatrT{\lambda}.
    \end{align}
    A simple calculation shows
    $
    M_\lambda e_d=S_{\ell\eta_j}^{-T} \rotmatrT{\lambda} e_d = S_{\ell\eta_j}^{-T} e_\lambda = S_{\ell\eta_j}^{-T} n_\lambda (\eta_j\ell, 1)^T = n_\lambda e_d.
    $
    Hence, the entries of the last column of $M_\lambda$ vanish except for the last one.
    Next, we prove the uniform boundedness of the set of operators $\{ M_\lambda \}_{\lambda\in\Lambda^{s,d}}$. It holds uniformly for $\lambda\in\Lambda^{s,d}$
    \[
    \|M_\lambda\|_{2\rightarrow 2} = \| S_{\ell\eta_j}^{-T} \|_{2\rightarrow 2} =\sqrt{1 + \eta_j^2|\ell|^2_2} \lesssim  \sqrt{1 + \eta_j^2 \mathbb{L}_j^2} \lesssim 1.
    \]
    Note that this implies that each entry in $M_\lambda$ is bounded in modulus.
     Since similar considerations hold for the inverse
    $
    M^{-1}_\lambda= R_{\theta_\lambda}R_{\varphi_\lambda} S_{\ell\eta_j}^{T},
    $  we can conclude that both $\widetilde{M}_\lambda := A^{-j}_{\alpha,\sigma} M_\lambda A^j_{\alpha,\sigma}$ and its inverse $\widetilde{M}^{-1}_\lambda$ have the form
	\begin{align*}
		\begin{pmatrix}
		 * & \dots &* & 0\\
		 \vdots & \ddots & \vdots & \vdots \\
		 * & \dots & * &0 \\
		  \square & \dots & \square & *
		\end{pmatrix},
	\end{align*}
	where the entries $*$ are the same as in $M_\lambda$ (or $M^{-1}_\lambda$) and the entries $\square$ are of the form
    $\sigma^{-j(1-\alpha)}[M_{\lambda}e_i]_{d}$ (or $\sigma^{-j(1-\alpha)}[M^{-1}_{\lambda}e_i]_{d}$) for $i\in\{1,\ldots,d-1\}$. In particular, the
    entries of $\widetilde{M}_\lambda$ and $\widetilde{M}^{-1}_\lambda$ are uniformly bounded in modulus. This implies
    $\|\widetilde{M}_\lambda\|_{2\rightarrow 2}\lesssim 1 $ and  $\|\widetilde{M}^{-1}_\lambda\|_{2\rightarrow 2}\lesssim 1$.
    Altogether, we obtain
    \begin{align}\label{eq:shmolaux1}
    |\widetilde{M}_\lambda\xi| \asymp |\xi| \text{ uniformly for $\xi\in\R^d$ and }\lambda\in\Lambda^{s,d}.
	\end{align}
    Due to the structure of the last column of $\widetilde{M}_\lambda$ we further have for $\xi=(\xi_1,\ldots,\xi_d)^T\in\R^d$
    \begin{align*}
	|\widetilde{M}_\lambda \xi|_{[d-1]} = |\widetilde{M}_\lambda (\xi_1,\ldots,\xi_{d-1},0)^T|_{[d-1]}
    \le \| \widetilde{M}_\lambda \|_{2\rightarrow 2} |(\xi_1,\ldots,\xi_{d-1},0)^T | = \| \widetilde{M}_\lambda \|_{2\rightarrow 2} |\xi|_{[d-1]}.
	\end{align*}
	For the inverse $\widetilde{M}^{-1}$ it holds analogously $|\widetilde{M}^{-1}_\lambda\xi|_{[d-1]} \le \| \widetilde{M}^{-1}_\lambda \|_{2\rightarrow 2} \otnorm{\xi}$.
    We conclude
    \begin{align}\label{eq:shmolaux2}
    |\widetilde{M}_\lambda \xi|_{[d-1]} \asymp \otnorm{\xi} \text{ uniformly for $\xi\in\R^d$ and }\lambda\in\Lambda^{s,d}.
    \end{align}
	 Finally, the following estimate holds uniformly for $\xi=(\xi_1,\ldots,\xi_d)^T\in\R^d$ and $\lambda\in\Lambda^{s,d}$,
     \begin{align}\label{eq:shmolaux3}
     |[\widetilde{M}_\lambda\xi]_d| \le |[M_{\lambda}]^{d}_{d}| |\xi_d| + \sum_{i=1}^{d-1} \sigma^{-j(1-\alpha)} |[M_{\lambda}]^{i}_{d}| |\xi_i|  \lesssim \sigma^{-(1-\alpha)j}\otnorm{\xi} + \abs{\xi_d}.
     \end{align}
	Finally, we can prove \eqref{eq:molcon} for every $\rho\in\N_0^d$ with $|\rho|_1\le L$,
	\begin{align*}
		\abs{\partial^{\rho} \gfuncfor(\xi)}
     \lesssim \sup_{|\beta|_1\le L}  \abs{ \big( \partial^{\beta} \hat{\gamma}_\lambda \big) (\widetilde{M}_\lambda \xi)}
     &\lesssim \frac{\min\{1, \sigma^{-j} +  \sigma^{-(1-\alpha)j}|\widetilde{M}_\lambda\xi|_{[d-1]}
       + |[\widetilde{M}_\lambda\xi]_d| \}^M}{\angles{|\widetilde{M}_\lambda\xi|}^{N_1} \angles{|\widetilde{M}_\lambda\xi|_{[d-1]}}^{N_2}} \\
	 &\lesssim \frac{\min\{1, \sigma^{-j} + \sigma^{-(1-\alpha)j}\otnorm{\xi} +\abs{[\xi]_d}\}^M}{\angles{\abs{\xi}}^{N_1} \angles{\otnorm{\xi}}^{N_2}}.
	\end{align*}
	The first estimate holds true, since $\hat{g}_{\lambda}(\xi)=\hat{\gamma}_\lambda(\widetilde{M}_\lambda \xi)$ and the entries of $\widetilde{M}_\lambda$ are uniformly bounded in $\lambda$.
	The second estimate is due to \eqref{eq:shearcon}. For the last estimate we used \eqref{eq:shmolaux1}, \eqref{eq:shmolaux2}, and \eqref{eq:shmolaux3}. The observation $s_\lambda =\sigma^{j}$ finishes the proof.
\end{proof}

\subsection{Consistency of the Shearlet Parametrizations}

The properties of $\alpha$-molecules depend essentially on their parametrizations. In view of Theorem~\ref{thm:sparseeqivmol} the consistency is of particular interest
when investigating approximation properties. In this paragraph we shall prove, in Proposition~\ref{prop:shconsist}, that shearlet parametrizations are consistent with each other.
This allows to establish approximation rates for various shearlet-like constructions simultaneously,
as long as they fall under the umbrella of the shearlet-molecule concept.

\begin{lemm} \label{pushToPlane}
Consider the \emph{gnomonic projection} $\phi : \R^d \backslash \set{ x\in\R^d \ \vert \ [x]_d =0} \to \R^d, x \mapsto \frac{1}{[x]_d}x$,
and let $1 \geq c>0$ be fixed. For $v, w \in \mathbb{S}^{d-1} \cap \set{ x\in\R^d : [x]_d \geq c}$ we then have
$|\phi(v) - \phi(w)| \asymp |v-w|$ and $\otnorm{v-w} \asymp  |v - w|$.
\end{lemm}
\begin{proof}
First note that  $\otnorm{v-w}=|\pi(v)-\pi(w)|$, where $\pi$ is the orthogonal projection of $\R^d$
onto the $(e_1,\ldots,e_{d-1})$-plane. On the set $\mathbb{S}^{d-1} \cap \set{ x\in\R^d : [x]_d \geq c}$, the mappings $\phi$ and $\pi$ are diffeomorphisms with bounded derivatives in both directions.
This implies the statement.
\end{proof}

In order to apply the previous lemma it is useful to record the following observation.

\begin{lemm} \label{PointSchummel}
	Let $1\ge c>0$ be fixed and let $w=(w_1,\ldots,w_d)^T\in \sph^{d-1}$ be a vector such that $w_d < c$.
    Then there exists a point $\tilde{w}=(\tilde{w}_1,\ldots,\tilde{w}_d)^T \in \sph^{d-1}$ with $\tilde{w}_d \geq c$ such that
    \begin{align*}
     |\tilde{w} - v| \leq |w- v| \qquad \text{ for every }\quad v\in \mathbb{S}^{d-1} \cap \set{ x\in\R^d : [x]_d \geq c}.
	\end{align*}
\end{lemm}
\begin{proof}
	If $w_d \leq - c$ simply take $\tilde{w}$ to be the reflection of $w$ at the $(e_1,\ldots,e_{d-1})$-plane.
    Then we have $\tilde{w}_d \ge c >0$ and we can conclude for every $v=(v_1,\ldots,v_d)^T \in \sph^{d-1}$ with $v_d \geq c>0$
    \[
    |v-w|^2=\otnorm{v - w}^2 + |v_d-w_d|^2 = \otnorm{v-\tilde{w}}^2 + |v_d+ \tilde{w}_d|^2 \ge \otnorm{v-\tilde{w}}^2 + |v_d-\tilde{w}_d|^2 =  |v-\tilde{w}|^2.
    \]
	In the other case $|w_d| < c$ we argue as follows. Applying a rotation about the $e_d$-axis, we may assume
    that $w$ is of the form $[\sqrt{1-w_d^2}, 0,\ldots  0, w_d]^T$. The vector
    $\tilde{w} =[\sqrt{1-c^2}, 0, \ldots, 0,  c]^T$ then has the desired properties.
    To verify this it suffices to show $\langle \tilde{w},v \rangle \ge \langle w,v \rangle$, because then
	\begin{align*}
		|\tilde{w} - v |^2 = |\tilde{w}|^2 + |v|^2 - 2 \langle \tilde{w},v \rangle = 2 - 2 \langle \tilde{w},v \rangle
     \leq  |w|^2 + |v|^2 - 2 \langle w,v \rangle = |w-v|^2.
	\end{align*}

    In order to show  $0\le \langle \tilde{w} - w,v \rangle = v_1( \sqrt{1-c^2} - \sqrt{1-w_d^2} ) + v_d(c-w_d)$, we first observe that
    $\sqrt{1-c^2}< \sqrt{1-w_d^2}$ since $|w_d|<c$. Moreover, every $v\in \mathbb{S}^{d-1} \cap \set{ x\in\R^d : [x]_d \geq c}$
    satisfies $|v_1|\le \sqrt{1-v_d^2}$. It follows
    \begin{align*}
      v_1\Big( \sqrt{1-c^2} - \sqrt{1-w_d^2}\Big) + v_d(c-w_d) %\ge |v_1| ( \sqrt{1-c^2} - \sqrt{1-w_d^2} ) + v_d(c-w_d)
      \ge \sqrt{1-v_d^2} \Big( \sqrt{1-c^2} - \sqrt{1-w_d^2} \Big) + v_d(c-w_d).
    \end{align*}

    Hence, it just remains to prove the inequality $\sqrt{1-v_d^2} \big( \sqrt{1-c^2} - \sqrt{1-w_d^2} \big) + v_d(c-w_d) \ge 0$ under the condition $0\le|w_d|<c\le v_d\le1$.
    The associated angles $\theta_w,\theta_c,\theta_v\in[0,\pi]$ defined by $\cos \theta_w=w_d$, $\cos \theta_c = c$, and $\cos\theta_v=v_d$, satisfy
    $\pi \ge \theta_w \ge \theta_c \ge \theta_v \ge 0$, and the inequality reads as
    \[
    \cos(\theta_c -\theta_v)= \langle (\sin \theta_v,\cos \theta_v) , (\sin \theta_c,\cos \theta_c)  \rangle \ge \langle (\sin \theta_v,\cos \theta_v) , (\sin \theta_w,\cos \theta_w)  \rangle = \cos(\theta_w -\theta_v).
    \]
    This however is obviously true since $0\le \theta_c-\theta_v \le \theta_w-\theta_v \le \pi$.
\end{proof}

After this preparation we are in the position to prove the consistency.

\begin{prop}\label{prop:shconsist}
Let $\alpha\in[0,1]$ and assume that $(\Lambda,\Phi_\Lambda)$ and $(\Delta,\Phi_\Delta)$ are
$\alpha$-shearlet parametrizations (possibly with different parameters).
Then $(\Lambda,\Phi_\Lambda)$ and $(\Delta,\Phi_\Delta)$ are $(\alpha,k)$-consistent for every $k>d$.
\end{prop}
\begin{proof}
As already was noted in Remark~\ref{rem:consCond},
it suffices to prove that for $N>d$ it holds
\begin{align*}
	 \sup_{\mu \in \Delta}\sum_{\lambda \in \Lambda} \omega_{\alpha}(\lambda, \mu)^{-N} < \infty.
\end{align*}
For this task it is convenient to decompose the shearlet index set $\Lambda =  \Lambda_0 \cup \cdots \cup \Lambda_d$ into the sets $\Lambda_\varepsilon$ associated with
the respective pyramidal regions $\mathcal{P}_\varepsilon$ for $\varepsilon\in\{1,\ldots,d\}$ and the low-frequency box $\mathcal{R}$ for $\varepsilon=0$.
The sum then splits accordingly into $d+1$ parts, which we handle separately below.

\paragraph*{\underline{$\Lambda_0$:}}

Let $\mu\in\Delta$ and $\lambda=(0,0,{\bf0},k)\in\Lambda_0$ with $k\in\Z^d$. The shearlet parametrization \eqref{eq:shearpara}
yields $s_\lambda=1$, $e_\lambda =e_d$, and $x_\lambda =\mathcal{T} k$. Furthermore, $s_\mu\ge 1$ for all $\mu\in\Lambda$. Hence we have
\[
\omega_{\alpha}(\lambda, \mu)
= s_\mu( 1 +  \abs{ \mathcal{T} k- x_\mu}^2 + \abspi{d_\sph(e_d, e_\mu)}^2 + \vert \sprod{e_d, \mathcal{T} k-x_\mu}\vert) \geq s_\mu( 1 + \abs{ \mathcal{T} k-x_\mu}^2).
\]
We conclude
\begin{align*}
	\sum_{\lambda \in \Lambda_0} \omega_{\alpha}(\lambda, \mu)^{-N}\leq \sum_{k\in\Z^d}  s_\mu^{-N} (1 + \abs{\mathcal{T} k-x_\mu}^2)^{-N} \lesssim \sum_{k\in\Z^d}  (1 + \abs{k}^2)^{-N},
\end{align*}
where for $N>d/2$ the sum on the right converges.

\paragraph*{\underline{$\Lambda_\varepsilon$, $\varepsilon\in\{1,\ldots,d\}$:}}

We only deal with the special case $\varepsilon=d$, since the other cases can be transformed to this case via rotations.
Let $\mu\in\Delta$ and write $s_{\mu}=\sigma^{j'}$ with $j^\prime\in\R$. In view of $s_{\lambda} = \sigma^j$ for $\lambda=(d,j,\ell,k)\in\Lambda_d$ we then have
\begin{align*}
	\sum_{\lambda \in \Lambda_d} \omega_{\alpha}(\lambda, \mu)^{-N} = \sum_{j\in\N_0} \sigma^{-N\abs{j-j'}} \sum_{\stackrel{\lambda \in \Lambda_d}{s_{\lambda} = \sigma^j}} \left(1 + d_\alpha(\lambda,\mu)\right)^{-N}.
\end{align*}
If we can prove that
\begin{align}\label{eq:sumest}
\mathcal{S}:=\sum_{\stackrel{\lambda \in \Lambda_d}{s_{\lambda} = \sigma^j}} \left(1 + d_\alpha({\lambda}, \mu)\right)^{-N}
	\lesssim %
    \sigma^{d|j-j^\prime|},
\end{align}
independently of $j\in\N_0$ and $\mu\in\Delta$, we are done, since $\sigma>1$, $s_\mu=\sigma^{j^\prime}$, $\max\{ s_\lambda/s_\mu , s_\mu/s_\lambda \}=\sigma^{|j-j^\prime|}$ and thus if $N>d$
\[
\sum_{\lambda\in \Lambda_d} \omega_{\alpha}(\lambda, \mu)^{-N} \lesssim \sum_{j \in \N_0} \sigma^{(d-N)\abs{j-j'}}
\leq 2 \sum_{j\in\N_0} \sigma^{(d-N)j} = \frac{2}{1-\sigma^{d-N}} < \infty.
\]
Putting in the definition of $d_\alpha(\lambda,\mu)$ and abbreviating $j_0:=\min\{j,j^\prime\}$, the sum $\mathcal{S}$ becomes
\begin{align}\label{eq:sum}
\mathcal{S}=\sum_{  \stackrel{\lambda \in \Lambda_d}{s_{\lambda}= \sigma^j}}
\left(1 + \sigma^{2\alpha j_0}\abs{x_\lambda - x_\mu}^2 + \sigma^{2(1-\alpha)j_0}\abspi{d_\sph(e_\lambda, e_\mu)}^2 + \sigma^{j_0}\abs{\sprod{e_\lambda, x_\lambda - x_\mu}} \right)^{-N}.
\end{align}

In order to prove the estimate \eqref{eq:sumest} for $\mathcal{S}$, we first study the different terms of the summand independently.
Let $\lambda=(d,j,\ell,k)\in\Lambda_d$ and recall the matrix $M_\lambda$ from \eqref{eq:auxmatrix}. It holds
\[
M^T_\lambda=R_{\theta_\lambda}R_{\varphi_\lambda}S^{-1}_{\ell\eta_j},
\]
and -- according to the discussion of $M_\lambda$ in the proof of Theorem~\ref{thm:shearletmol} -- its last row is given by $(0,\ldots,0,n_\lambda)$
with $n_\lambda = (1 + \eta_j^2\abs{\ell}_2^2)^{-\frac{1}{2}}$. Since $\eta_j\asymp \sigma^{-j(1-\alpha)}$ and $|\ell|_2 \lesssim \mathbb{L}_j \lesssim  \sigma^{j(1-\alpha)}$
this implies $n_\lambda\asymp 1$ uniformly for all $\lambda\in\Lambda_d$.

As a direct consequence $[M^T_\lambda x]_d =n_\lambda [x]_d \asymp [x]_d$ uniformly for $\lambda\in\Lambda_d$ and $x\in\R^d$.
In addition, we have $|M^T_\lambda x| \asymp |x|$ uniformly for $\lambda\in\Lambda_d$ and $x\in\R^d$ since $\|M^T_\lambda\|_{2\rightarrow 2} = \|M_\lambda\|_{2\rightarrow 2} \lesssim 1 $
and also $\|M^{-T}_\lambda\|_{2\rightarrow 2} = \|M^{-1}_\lambda\|_{2\rightarrow 2}\lesssim 1$.

These observations allow the following estimate,
\begin{align}\label{eq:firstterm}
\abs{x_\lambda - x_\mu}&= \abs{ S^{-1}_{\ell\eta_j} A^{-j}_{\alpha,\sigma} \mathcal{T} k -x_\mu}  \notag
=\abs{R_{\theta_\lambda}R_{\varphi_\lambda}S^{-1}_{\ell\eta_j} A^{-j}_{\alpha,\sigma} \mathcal{T} k - R_{\theta_\lambda}R_{\varphi_\lambda}x_\mu} \\\notag
&=\abs{M^T_\lambda \big( \mathcal{T}  A^{-j}_{\alpha,\sigma} k - M^{-T}_\lambda R_{\theta_\lambda}R_{\varphi_\lambda}x_\mu \big) }
\asymp  \abs{  A^{-j}_{\alpha,\sigma} k -\mathcal{T}^{-1} M^{-T}_\lambda R_{\theta_\lambda}R_{\varphi_\lambda}x_\mu   } \\
&\gtrsim  \otnorm{   A^{-j}_{\alpha,\sigma} k-\mathcal{T}^{-1} S_{\ell\eta_j} x_\mu }
= \otnorm{  \sigma^{-j\alpha} k -\mathcal{T}^{-1} S_{\ell\eta_j} x_\mu    }.
\end{align}

In view of $e_\lambda=R^T_{\varphi_\lambda} R^T_{\theta_\lambda} e_d$ we also have the estimate
\begin{align}\label{eq:secondterm}
\abs{\sprod{e_\lambda, x_\lambda - x_\mu}} &= \abs{\sprod{e_\lambda,  S^{-1}_{\ell\eta_j} A^{-j}_{\alpha,\sigma} \mathcal{T} k - x_\mu }} \notag= \abs{\sprod{e_d, R_{\theta_\lambda} R_{\varphi_\lambda} S^{-1}_{\ell\eta_j} A^{-j}_{\alpha,\sigma} \mathcal{T} k - R_{\theta_\lambda} R_{\varphi_\lambda} x_\mu}} \\ \notag
&= \abs{\sprod{e_d, M^T_\lambda \big(\mathcal{T} A^{-j}_{\alpha,\sigma} k -M^{-T}_\lambda R_{\theta_\lambda} R_{\varphi_\lambda} x_\mu \big)}}
\asymp\abs{\sprod{e_d,    A^{-j}_{\alpha,\sigma} k - \mathcal{T}^{-1} S_{\ell\eta_j} x_\mu  }} \\
&=  \abs{\sprod{e_d,  \sigma^{-j} k - \mathcal{T}^{-1} S_{\ell\eta_j} x_\mu}}.
\end{align}

According to the shearlet parametrization \eqref{eq:shearpara} we have $e_\lambda=n_\lambda (\ell\eta_j , 1)^T$, where $n_\lambda \asymp 1$ as shown above.
Hence, there is a constant $c>0$ such that
$n_\lambda \ge c$ for all $\lambda\in\Lambda_d$. It follows $[e_\lambda]_d\ge c>0$ for all $\lambda\in\Lambda_d$.
Without loss of generality we can further assume that $[e_\mu]_d\ge 0$ since $|\{d_{\mathbb{S}}(e_\lambda, -e_\mu)\}|=|\{d_{\mathbb{S}}(e_\lambda, e_\mu)\}|$.

In this situation Lemma~\ref{lem:metrEquiv} applies and tells us that $\abspi{d_\sph(e_\lambda, e_\mu)} \asymp |e_\lambda - e_\mu|$.
Moreover, possibly after changing $e_\mu$ to $\tilde{e}_\mu$ as in
Lemma~\ref{PointSchummel} to enforce $[\tilde{e}_\mu]_d\ge c$,
we have $|e_\lambda - e_\mu| \ge | e_\lambda - \tilde{e}_\mu|$ for all $\lambda\in\Lambda_d$.

Let $\phi$ be the gnomonic projection from Lemma~\ref{pushToPlane}. Then $\phi(e_\lambda)=( \ell \eta_j , 1)^T\in\mathbb{R}^{d}$ and
Lemma~\ref{pushToPlane} together with the observation
$\otnorm{\phi(v)-\phi(w)}= \abs{\phi(v)- \phi(w)}$ implies the estimate,
\[
| e_\lambda -\tilde{e}_\mu |   \asymp | (\ell\eta_j , 1)^T - \phi(\tilde{e}_\mu) | = | (\ell\eta_j,1)^T -\phi(\tilde{e}_\mu) |_{[d-1]}.
\]
Subsequently, let $\nu_\mu\in\R^{d-1}$ be defined by $\phi(\tilde{e}_\mu)=(\nu_\mu,1)^T$. Altogether we arrive at
\begin{align}\label{eq:thirdterm}
\abspi{d_\sph(e_\lambda, e_\mu)}  \gtrsim | (\ell \eta_j , 1)^T -\phi(\tilde{e}_\mu) |_{[d-1]} = | \ell \eta_j - \nu_\mu|.
\end{align}
We now use \eqref{eq:firstterm}, \eqref{eq:secondterm} and \eqref{eq:thirdterm} to estimate
the sum $\mathcal{S}$ in \eqref{eq:sum}.
Introducing the quantities $q_1(\ell):=\sigma^{j\alpha}  \mathcal{T}^{-1}S_{\ell\eta_j} x_\mu$, $q_2(\ell):=\sigma^{j} \mathcal{T}^{-1}S_{\ell\eta_j} x_\mu$, and $q_3:=\eta^{-1}_j \nu_\mu$,
and taking into account $\eta_j\asymp\sigma^{-(1-\alpha)j}$ we obtain
\begin{align*}
\mathcal{S}\lesssim \sum_{k \in \Z^d} \sum_{\ell\in \mathscr{L}_{d,j}}
\left( 1 + \sigma^{2\alpha(j_0-j)}  \otnorm{   k -q_1(\ell)}^2
+ \sigma^{(j_0-j)} \abs{\sprod{e_d,   k-q_2(\ell) }}
	+ \sigma^{2(1-\alpha)(j_0-j)} |  \ell-q_3 |^2 \right)^{-N}.
\end{align*}
Next, we distinguish the cases $j> j^\prime$ and $j\le j^\prime$. For $j\le j^\prime$ we have $j_0=j$ and we obtain
\begin{align*}
\mathcal{S} &\lesssim  \sum_{k \in \Z^d} \sum_{\ell\in\Z^{d-1}}   \left( 1 +  \otnorm{  k-q_1(\ell) }^2
+ \abs{\sprod{e_d,   k-q_2(\ell) }} +  |  \ell- q_3 |^2 \right)^{-N} \\
&\lesssim \sum_{k\in\Z^d} \sum_{\ell\in\Z^{d-1}}   \left( 1 +  \otnorm{k}^2
+ \abs{\sprod{e_d,k}}
+  |\ell|_{[d-1]}^2 \right)^{-N}<\infty .
\end{align*}

In case $j>j^\prime$ it holds $j_0=j^\prime$ and $p:=j^\prime-j\le0$. The term $\sigma^{pd} \mathcal{S}$ is -- up to a multiplicative constant -- bounded by
\begin{align*}
\sum_{\ell\in\Z^{d-1}} \sigma^{p(1-\alpha)(d-1)} \sum_{k \in \Z^d} \sigma^{p\alpha(d-1)}\sigma^{p}
\big( 1 + \sigma^{2\alpha p}  \otnorm{   k-q_1(\ell)}^2
+ \sigma^{p} \abs{\sprod{e_d,  k- q_2(\ell) }}
	+ \sigma^{2(1-\alpha)p} |   \ell -q_3 |^2 \big)^{-N}.
\end{align*}

The last sum can be interpreted as a Riemann sum, which is bounded -- up to a multiplicative constant independent of $j$ and $j^\prime$ as long as $j>j^\prime$
-- by the corresponding integral
\begin{align*}
\int_{y \in \R^{d-1}}  \int_{x\in\R^d}   \Big( 1 +\otnorm{  x-\sigma^{\alpha p} q_1(y) }^2 +  \abs{\sprod{e_d,   x-\sigma^p q_2(y)}}
	+ | y - \sigma^{(1-\alpha)p} q_3 |^2 \Big)^{-N} \,dx\,dy.
\end{align*}
All in all we end up with
$%\begin{align*}
\mathcal{S}\lesssim  \sigma^{d(j-j^\prime)} \int_{y \in \R^{d-1}} \int_{x \in \R^{d}} \Big( 1 +  \otnorm{x}^2 +  \abs{\sprod{e_d, x}} + |y|^2 \Big)^{-N} \,dx\,dy.
$ %\end{align*}
To see that the integral converges for $N>d$, we carry out the integration over $x_d$, which yields up to a fixed constant
\[
\int_{y \in \R^{d-1}} \int_{\tilde{x}\in \R^{d-1}}  ( 1 + \otnorm{\tilde{x}}^2 + |y|^2 )^{-(N-1)} \,d\tilde{x}\,dy = \int_{z\in \R^{2(d-1)}} ( 1 + |z|^2)^{-(N-1)}.
\]
The integral on the right converges precisely for $N>d$. This observation concludes the proof.
\end{proof}

\subsection{Pyramid-adapted Shearlet Systems in $3$D }
\label{ssec:3Dshearlets}

In the sequel we focus on some concrete shearlet systems in $3$ dimensions, which are already on the market and for which we verify that they are instances of
$\alpha$-shearlet molecules for $\alpha=\frac{1}{2}$ and thus by Theorem~\ref{thm:shearletmol} also systems of 3-dimensional parabolic molecules.

Let us first recall the classic definition \cite{KL12} of a shearlet system in $L^2(\R^3)$.
A system of 3D-shearlets is defined as a collection of functions in $L^2(\R^3)$ of the form
\begin{align}\label{eq:biasedsystem}
\Big\{ \psi_{j,\ell,k} = 2^j \psi( S_\ell A_{\frac{1}{2},2}^j \cdot - k) ~:~ j\in\Z, \ell\in\Z^2, k\in\Z^3 \Big\},
\end{align}
where $\psi\in L^2(\R^3)$ is a suitable generating function.
The classical choice for the generator is a function $\psi$ defined on the frequency domain by

\begin{align}\label{eq:classicshearlet}
\hat{\psi}(\xi)=w(\xi_3) \textstyle{ v(\frac{\xi_1}{\xi_3})v(\frac{\xi_2}{\xi_3}) }, \quad \xi=(\xi_1,\xi_2,\xi_3)^T\in\R^3,
\end{align}
where $v\in C^\infty_c(\R)$ is a bump function and $w\in C^\infty_c(\R)$ the Fourier transform of a suitable univariate discrete wavelet.
It is possible to choose $v$ and $w$ so that \eqref{eq:biasedsystem} becomes a Parseval frame for $L^2(\R^3)$.

Unfortunately, the shearlet system \eqref{eq:biasedsystem} has a directional bias due to the fact, that for large shearings the frequency support of the elements
becomes more and more elongated along the $(e_1,e_2)$-plane.
This directional bias affects negatively the approximation properties of \eqref{eq:biasedsystem} and makes this system impractical in most applications.

To avoid this problem, the Fourier domain is usually partitioned into three pyramidal regions
\begin{align*}
\mathcal{P}_1&=\Big\{ (\xi_1,\xi_2,\xi_3)\in\R^3 ~:~ \textstyle{|\frac{\xi_2}{\xi_1}|\le 1,  |\frac{\xi_3}{\xi_1}|\le 1} \Big\}, \\
\mathcal{P}_2&=\Big\{ (\xi_1,\xi_2,\xi_3)\in\R^3 ~:~ \textstyle{|\frac{\xi_1}{\xi_2}|\le 1,  |\frac{\xi_3}{\xi_2}|\le 1} \Big\}, \\
\mathcal{P}_3&=\Big\{ (\xi_1,\xi_2,\xi_3)\in\R^3 ~:~ \textstyle{|\frac{\xi_1}{\xi_3}|\le 1,  |\frac{\xi_2}{\xi_3}|\le 1} \Big\},
\end{align*}
and for each pyramid a separate shearlet system is used.
Then, since each system only has to cover one pyramid, the shear parameters can be restricted avoiding large shears.
To take care of low frequencies, it is common to use distinguished coarse-scale elements with frequencies in a centered box.
Here this will be the cube
\[
\mathcal{R}=\Big\{ \xi\in\R^3 ~:~  |\xi|_\infty\le \textstyle{\frac{1}{8}} \Big\}.
\]
This cube together with the truncated pyramids $\tilde{\mathcal{P}}_1:=\mathcal{P}_1\backslash \mathcal{R}$,
$\tilde{\mathcal{P}}_2:=\mathcal{P}_2\backslash \mathcal{R}$, and $\tilde{\mathcal{P}}_3:=\mathcal{P}_3\backslash \mathcal{R}$
partitions the Fourier domain into 4 regions.

With each of these regions, different operators are associated. The coarse-scale
functions are only translated, in the other regions we also scale and shear.
The scaling and shearing operators associated with the respective regions $\varepsilon\in\{1,2,3\}$ are given by $A^{(\varepsilon)}_{\frac{1}{2},s}=Z^\varepsilon A_{\frac{1}{2},s} Z^{-\varepsilon}$
and $S^{(\varepsilon)}_{h}=Z^\varepsilon S_{h} Z^{-\varepsilon}$, and take the concrete form
\begin{align*}
A^{(1)}_{\frac{1}{2},s}= \begin{pmatrix}
		s  & 0 & 0 \\  0 & s^{\frac{1}{2}} & 0  \\ 0 & 0 & s^{\frac{1}{2}}
\end{pmatrix},
\quad
A^{(2)}_{\frac{1}{2},s}= \begin{pmatrix}
		s^{\frac{1}{2}} & 0 & 0 \\  0 & s & 0  \\ 0 & 0 & s^{\frac{1}{2}}
\end{pmatrix},
\quad	
A^{(3)}_{\frac{1}{2},s}=A_{\frac{1}{2},s}= \begin{pmatrix}
		s^{\frac{1}{2}}  & 0 & 0 \\  0 & s^{\frac{1}{2}} & 0  \\ 0 & 0 & s
\end{pmatrix},	
\end{align*}
for $s>0$,  and for $h\in\R^2$
\begin{align*}
S^{(1)}_{h}= \begin{pmatrix}
	1 & h_1 & h_2  \\ 0 & 1 & 0 \\ 0 & 0 & 1
\end{pmatrix},
\quad
S^{(2)}_{h}= \begin{pmatrix}
		1 & 0 & 0 \\ h_2  & 1 & h_1  \\ 0 & 0 & 1
\end{pmatrix},	
\quad
S^{(3)}_{h}=S_h= \begin{pmatrix}
		1 & 0 & 0 \\  0 & 1 & 0 \\ h_1 & h_2 & 1
\end{pmatrix}.	
\end{align*}

Now we are ready to define a modified shearlet system, which is adapted to our partition of the Fourier domain and therefore
called \emph{cone-adapted}. These systems do not exhibit the directional bias as \eqref{eq:biasedsystem}. In the 3D cone-adapted setting they are called
\emph{pyramid-adapted 3D shearlet systems}.

\begin{defi}[\cite{Kutyniok2012,KuLeLim12}]
For fixed $\tau_1,\tau_2>0$ let $\mathcal{T}=\diag(\tau_1,\tau_2,\tau_2)\in\R^{3\times3}$.
The (affine) \emph{pyramid-adapted 3D shearlet system} generated by the functions $\phi\in L^2(\R^3)$ and
$\psi^\varepsilon\in L^2(\R^3)$, $\varepsilon\in\{1,2,3\}$, is defined as the union
\begin{align}\label{eq:affineshear}
SH(\phi,\psi^1,\psi^2,\psi^3;\tau_1,\tau_2):=\Phi(\phi;\tau_1) \cup \Psi_1(\psi^1;\tau_1,\tau_2) \cup \Psi_2(\psi^2;\tau_1,\tau_2) \cup \Psi_3(\psi^3;\tau_1,\tau_2)
\end{align}
of the coarse-scale functions $\Phi(\phi;\tau_1):=\{ \phi_k=\phi(\cdot -  \tau_1 k) : k\in\Z^3\}$ and the functions
\begin{align*}%\label{eq:pyramidsys}
\Psi_\varepsilon(\psi^\varepsilon;\tau_1,\tau_2) :=
\Big\{ \psi^\varepsilon_{j,\ell,k} = 2^j \psi^\varepsilon( Z^\varepsilon S_{\ell} A_{\frac{1}{2},2}^j Z^{-\varepsilon} \cdot - Z^{\varepsilon}\mathcal{T}Z^{-\varepsilon} k) :~
j\in\N_0, \ell\in\Z^2, |\ell|_\infty\le \lceil 2^{j/2}\rceil, k\in\Z^3 \Big\}
\end{align*}
associated with the pyramids $\tilde{\mathcal{P}}_\varepsilon$ for $\varepsilon\in\{1,2,3\}$.
\end{defi}

These pyramid-adapted affine systems are the prime examples of $\frac{1}{2}$-shearlet-molecules.
In practice, one wants them to be frames, especially tight frames.
However, the construction of tight frames of pyramid-adapted shearlets is not trivial.

The simplest way to obtain a Parseval frame of pyramid-adapted shearlets starts with a Parseval shearlet frame of the type \eqref{eq:biasedsystem},
which is easier to construct. It yields a shearlet system associated with the pyramid $\tilde{\mathcal{P}}_3$
by removing all elements, whose frequency support does not intersect $\tilde{\mathcal{P}}_3$. Truncating
the remaining functions in the frequency domain outside of $\tilde{\mathcal{P}}_3$,
one obtains a Parseval frame for the space
\[
L^2(\tilde{\mathcal{P}}_{3})^\vee :=\{ f\in L^2(\R^3) ~:~ \supp\hat{f}\subset \tilde{\mathcal{P}}_{3} \}.
\]
A similar procedure yields Parseval frames associated with the the other parts of the Fourier domain, namely for
$L^2(\tilde{\mathcal{P}}_{\varepsilon})^\vee$, $\varepsilon\in\{1,2\}$, and $L^2(\mathcal{R})^\vee $.
The union of these frames then is a Parseval frame for the whole space $L^2(\R^3)$.

This approach has the drawback that it leads to bad spatial localization of the shearlets due to their lack of smoothness in the frequency domain,
which is a consequence of the truncation. A different approach was taken by Candès, Demanet, and Ying in \cite{Ying2005}.
They gave up on the affine structure of the system and found a shearlet-like construction of a Parseval frame.
Guo and Labate modified this approach \cite{Guo2010,Guo2012a} and found another shearlet-type construction,
which is even close to affine.

\subsubsection{Bandlimited Tight Frames of Pyramid-adapted Shearlets}
\label{sssec:ssh}

The construction we present here is the one due to Guo and Labate \cite{Guo2010,Guo2012a}.
It starts with a Meyer scaling function $\phi\in \mathscr{S}(\R)$ satisfying $0 \leq \hat{\phi} \leq 1$, $\supp\hat{\phi}\subseteq[-\frac{1}{8},\frac{1}{8}]$,
and $\hat{\phi}=1$ on $[-\frac{1}{16}, \frac{1}{16}]$, which is used to define $\Phi\in\mathscr{S}(\R^3)$ on the Fourier side by
\begin{align}
  \hat{\Phi}(\xi) := \hat{\phi}(\xi_1)\hat{\phi}(\xi_2)\hat{\phi}(\xi_3),\quad \xi=(\xi_1,\xi_2,\xi_3)^T\in\R^3.      \label{eq:phidef}
\end{align}
Then the function
\begin{align*}
	W(\xi) := \sqrt{ \hat{\Phi}^2(2^{-2}\xi)- \hat{\Phi}^2(\xi)}, \quad \xi\in\R^3,
\end{align*}
is defined, with
$
\supp W\subseteq {\textstyle[-{\frac{1}{2}},\frac{1}{2}]^3\backslash (-\frac{1}{16},\frac{1}{16})^3}
$
and $W=1$ on $[-{\scriptsize{\frac{1}{4}}},\scriptsize{\frac{1}{4}}]^3\backslash (-\frac{1}{8},\frac{1}{8})^3$,
such that $%\begin{align*}
\hat{\Phi}^2(\xi) + \sum_{j \geq 0} W^2(2^{-2j}\xi) = 1.
$%\end{align*}
for every $\xi\in\R^3$.
These functions
thus produce a smooth tiling of the frequency domain into cartesian coronae.
To take care of the directional scaling, a bump-like function $v\in C_c^\infty(\R)$ is used, which satisfies $\supp v  \subseteq [-1,1]$ and
\begin{align*}
	&\abs{v(t-1)}^2 + \abs{v(t)}^2 + \abs{v(t+1)}^2 =1 \quad \text{ for } t\in [-1,1],\\
    &v(0)=1 \text{ and } v^{(n)}(0)=0 \text{ for } n\ge1.
\end{align*}
For an explicit construction of such a function we refer to \cite{GL07}.
Then $V\in C^\infty(\R^3)$ is defined by
\[
V(\xi) = \textstyle{ v( \frac{\xi_1}{\xi_3}) v(\frac{\xi_2}{\xi_3})}, \quad\xi=(\xi_1,\xi_2,\xi_3)^T\in\R^3,
\]
with support fully contained in the pyramid $\mathcal{P}_3$.
After this preparation the
smooth Parseval frame of band-limited 3D-shearlets
introduced by Guo and Labate in \cite{Guo2012a} can be defined.

The \emph{coarse-scale functions}, which take care of the low frequencies in $\mathcal{R}$,
are translates of the function $\Phi$ from \eqref{eq:phidef},
\begin{align}\label{eq:coarsescale}
SH_0:=\big\{	\psi^{0}_{0,{\bf0},k}(x) = \Phi(\cdot - k) ~:~  k\in\Z^3 \big\}.
\end{align}

The shearlets, whose frequency support is fully contained in the respective pyramids, are called \emph{interior shearlets}.
They are defined as the collection of functions
\begin{align*}%\label{eq:intshear}
SH_{int}:=\big\{ \psi^{\varepsilon}_{j,\ell,k} ~:~ (\varepsilon,j,\ell,k)\in\Lambda_{int} \big\}
\end{align*}
indexed by $\Lambda_{int}:=\{ (\varepsilon,j,\ell,k)\in \{1,2,3\}\times\N_0\times \Z^2 \times\Z^3 : |\ell|_\infty<2^j \}$ and with Fourier transforms
\begin{align} \label{eq:psifordef}
\hat{\psi}_{j,\ell,k}^\varepsilon(\xi)= 2^{-2j} W(2^{-2j}\xi) V ( S^{-T}_{\ell} A^{-j}_{\frac{1}{2},4}  Z^{-\varepsilon} \xi)
\exp( - 2\pi i \sprod{ Z^{\varepsilon} S^{-T}_\ell A^{-j}_{\frac{1}{2},4} Z^{-\varepsilon}\xi, k}),
\quad \xi\in\R^3,
\end{align}
where $W$ and $V$ are the auxiliary functions from above. Observe that
\[
V ( S^{-T}_{\ell} A^{-j}_{\frac{1}{2},4} \xi)=v(2^j \textstyle{\frac{\xi_1}{\xi_3}} -\ell_1) v(2^j \textstyle{\frac{\xi_2}{\xi_3}} -\ell_2),
\]
and compare this to \eqref{eq:classicshearlet}.
Finally, the so-called \emph{boundary shearlets}, obtained by carefully glueing together shearlets from adjacent pyramidal regions, are added to obtain a smooth well-localized frame.
The glueing process is rather delicate, since it is important to select the right shearlets matching together.
In fact, even in the original construction~\cite{Guo2012a} there are some inaccuracies. We correct them here, which leads to a slight modification
of the original definition.

We make a distinction between boundary shearlets defined
between two pyramidal regions and those at the corners, where three pyramidal regions meet.
Following \cite{Guo2012a} we also distinguish between the scales $j\ge 1$ and $j=0$.

Let us begin with the scales $j\ge1$ and the functions at the boundary where only two pyramidal regions meet.
We define, for $j\ge1$, $\varepsilon_1\in\{-1,1\}$, $\ell_1=\varepsilon_1 2^j$, $|\ell_2|<2^j$, and $k\in\Z^3$,
\begin{align}\label{eq:bndshear1}
	\hat{\psi}_{j,\ell,k}^1(\xi) &= \begin{cases}\notag
         2^{-2j-3}W(2^{-2j}\xi) v(2^j\frac{\xi_2}{\xi_1} - \ell_1 ) v(2^j\frac{\xi_3}{\xi_1} - \ell_2) \exp(- 2\pi i \sprod{ 2^{-2} Z S^{-T}_\ell  A_{\frac{1}{2},2^{-2j}}   Z^{-1}\xi, k}) & , \xi \in \mathcal{P}_1; \\
	     2^{-2j-3}W(2^{-2j}\xi) v(2^j\frac{\xi_1}{\xi_2} - \ell_1 ) v(2^j\frac{\xi_3}{\xi_2} - \varepsilon_1\ell_2) \exp(- 2\pi i \sprod{ 2^{-2} Z S^{-T}_\ell A^{-j}_{\frac{1}{2},4}  Z^{-1}\xi, k}) & , \xi \in \mathcal{P}_2;
    \end{cases} \\\notag
	\hat{\psi}_{j,\ell,k}^2(\xi) &= \begin{cases}
         2^{-2j-3}W(2^{-2j}\xi) v(2^j\frac{\xi_3}{\xi_2} - \ell_1)  v(2^j\frac{\xi_1}{\xi_2} - \ell_2 ) \exp(- 2\pi i \sprod{ 2^{-2} Z^2 S^{-T}_\ell  A^{-j}_{\frac{1}{2},4}   Z^{-2} \xi, k}) &, \xi \in \mathcal{P}_2; \\
	     2^{-2j-3}W(2^{-2j}\xi) v(2^j\frac{\xi_2}{\xi_3} - \ell_1)  v(2^j\frac{\xi_1}{\xi_3} - \varepsilon_1\ell_2 ) \exp(- 2\pi i \sprod{ 2^{-2} Z^2 S^{-T}_\ell  A^{-j}_{\frac{1}{2},4}  Z^{-2} \xi, k}) &, \xi \in \mathcal{P}_3;
    \end{cases} \\
	\hat{\psi}_{j,\ell,k}^3(\xi) &= \begin{cases}
         2^{-2j-3}W(2^{-2j}\xi) v(2^j\frac{\xi_1}{\xi_3} - \ell_1 ) v(2^j\frac{\xi_2}{\xi_3} - \ell_2) \exp(- 2\pi i \sprod{ 2^{-2} S^{-T}_\ell  A^{-j}_{\frac{1}{2},4}   \xi, k}) &, \xi \in \mathcal{P}_3; \\
         2^{-2j-3}W(2^{-2j}\xi) v(2^j\frac{\xi_3}{\xi_1} - \ell_1 ) v(2^j\frac{\xi_2}{\xi_1} - \varepsilon_1\ell_2)  \exp(- 2\pi i \sprod{ 2^{-2} S^{-T}_\ell  A^{-j}_{\frac{1}{2},4}   \xi, k}) &, \xi \in \mathcal{P}_1.
    \end{cases}
\end{align}

Note that we only give the definition in the regions where the functions have non-trivial support. Outside they are supposed to be zero.

Next, we come to the corners where three pyramidal regions meet. Here we have to glue together three shearlet parts, each coming from a different pyramid.
For convenience the corner elements will be associated to the first pyramid $\mathcal{P}_1$.
For $j\geq1$, $\ell_1=\varepsilon_1 2^j$, $\ell_2=\varepsilon_2 2^j$ with $\varepsilon_1,\varepsilon_2\in\{-1,1\}$,
and $k\in\Z^3$, they are defined as follows:
\begin{align}\label{eq:bndshear2}
	\hat{\psi}_{j,\ell,k}^1(\xi) &= \begin{cases}
      2^{-2j-3}W(2^{-2j}\xi) v(2^j\frac{\xi_2}{\xi_1} - \ell_1 ) v(2^j\frac{\xi_3}{\xi_1} - \ell_2) \exp(- 2\pi i \sprod{ 2^{-2} Z S^{-T}_\ell A^{-j}_{\frac{1}{2},4}   Z^{-1} \xi, k}) &, \xi \in \mathcal{P}_1; \\
	  2^{-2j-3}W(2^{-2j}\xi) v(2^j\frac{\xi_1}{\xi_2} - \ell_1) v(2^j\frac{\xi_3}{\xi_2} -\varepsilon_1 \ell_2 )  \exp(- 2\pi i \sprod{ 2^{-2} Z S^{-T}_\ell  A^{-j}_{\frac{1}{2},4}  Z^{-1}\xi, k}) &, \xi \in \mathcal{P}_2; \\
	  2^{-2j-3}W(2^{-2j}\xi)  v(2^j\frac{\xi_2}{\xi_3} - \varepsilon_2 \ell_1) v(2^j\frac{\xi_1}{\xi_3} - \ell_2 ) \exp(- 2\pi i \sprod{ 2^{-2} Z S^{-T}_\ell A^{-j}_{\frac{1}{2},4} Z^{-1}\xi, k}) &, \xi \in \mathcal{P}_3.
	 \end{cases}
\end{align}

As in the original construction the definition of the boundary shearlets is slightly different at the lowest scale $j=0$.
For $j=0$, $\ell_1=\pm 1$, $\ell_2=0$, and $k\in\Z^3$, we set
\begin{align}\label{eq:bndshear3}
	\hat{\psi}_{j,\ell,k}^1(\xi) &= \begin{cases} \notag
         W(\xi) v(\frac{\xi_2}{\xi_1} - \ell_1 ) v(\frac{\xi_3}{\xi_1} ) \exp(- 2\pi i \sprod{\xi, k}) & , \xi \in \mathcal{P}_1; \\
	     W(\xi) v(\frac{\xi_1}{\xi_2} - \ell_1 ) v(\frac{\xi_3}{\xi_2} ) \exp(- 2\pi i \sprod{\xi, k}) & , \xi \in \mathcal{P}_2;
    \end{cases} \\
	\hat{\psi}_{j,\ell,k}^2(\xi) &= \begin{cases}
         W(\xi) v(\frac{\xi_3}{\xi_2} - \ell_1)  v(\frac{\xi_1}{\xi_2} ) \exp(- 2\pi i \sprod{  \xi, k}) &, \xi \in \mathcal{P}_2; \\
	     W(\xi) v(\frac{\xi_2}{\xi_3} - \ell_1)  v(\frac{\xi_1}{\xi_3} ) \exp(- 2\pi i \sprod{  \xi, k}) &, \xi \in \mathcal{P}_3;
    \end{cases} \\\notag
	\hat{\psi}_{j,\ell,k}^3(\xi) &= \begin{cases}
         W(\xi) v(\frac{\xi_1}{\xi_3} - \ell_1 ) v(\frac{\xi_2}{\xi_3} ) \exp(- 2\pi i \sprod{     \xi, k}) &, \xi \in \mathcal{P}_3; \\
         W(\xi) v(\frac{\xi_3}{\xi_1} - \ell_1 ) v(\frac{\xi_2}{\xi_1} ) \exp(- 2\pi i \sprod{    \xi, k}) &, \xi \in \mathcal{P}_1.
    \end{cases}
\end{align}

Finally, we come to the corner elements at the scale $j=0$. Again they are associated with the first pyramid.
Let $\varepsilon_1,\varepsilon_2\in\{-1,1\}$. Then we define them for $j=0$, $\ell=(\varepsilon_1,\varepsilon_2)$, and $k\in\Z^3$ by
\begin{align}\label{eq:bndshear4}
	\hat{\psi}_{j,\ell,k}^1(\xi) &= \begin{cases}
      W(\xi) v(\frac{\xi_2}{\xi_1} - \ell_1 ) v(\frac{\xi_3}{\xi_1} - \ell_2) \exp(- 2\pi i \sprod{  \xi, k}) &, \xi \in \mathcal{P}_1; \\
	  W(\xi) v(\frac{\xi_1}{\xi_2} - \ell_1) v(\frac{\xi_3}{\xi_2} -\varepsilon_1 \ell_2 )  \exp(- 2\pi i \sprod{\xi, k}) &, \xi \in \mathcal{P}_2; \\
	  W(\xi) v(\frac{\xi_2}{\xi_3} - \varepsilon_2 \ell_1) v(\frac{\xi_1}{\xi_3} - \ell_2 ) \exp(- 2\pi i \sprod{ \xi, k}) &, \xi \in \mathcal{P}_3.
	 \end{cases}
\end{align}

All boundary shearlets are collected in the family $SH_{bound}$.
Together with the coarse-scale functions $SH_{0}$ and the interior shearlets $SH_{int}$ they provide a Parseval frame for $L^2(\R^3)$.

\begin{theo}[\cite{Guo2012a}]
The system of $3D$-shearlets
\begin{align*}
SH := SH_0 \cup SH_{int} \cup SH_{bound} %\Big\{ \psi^\varepsilon_{j,\ell,k} ~:~ k\in\Z^3 \Big\},
\end{align*}
is a smooth (well-localized) Parseval frame for $L^2(\R^3)$, consisting of band-limited Schwarz functions.
\end{theo}

The corresponding index set $\Lambda_{SH}\subset \{0,1,2,3\}\times\N_0\times\Z^2\times\Z^3$ is given by
\begin{align}\label{eq:sshindex}
\Lambda_{SH}:=\Big\{ (0,0,{\bf0},k): k\in\Z^3 \Big\} \cup \Big\{ (\varepsilon,j,\ell,k) ~:~ \varepsilon\in \{1,2,3\},\, j\in\N_0,\, \ell\in\mathscr{L}_{\varepsilon,j}\subseteq\Z^2,\, k\in\Z^3  \Big\}
\end{align}
with the shear parameters $\mathscr{L}_{\varepsilon,j}=\{ (\ell_1,\ell_2)\in\Z^2 : |\ell_1|\le 2^j , |\ell_2|<2^j \}$ for $\varepsilon\in\{2,3\}$ and
$\mathscr{L}_{1,j}=\{ (\ell_1,\ell_2)\in\Z^2 : |\ell_1|\le 2^j , |\ell_2|<2^j \} \cup \{ (\pm 2^j, \pm 2^j) \}$ at each scale $j\in\N_0$.

The system $SH$ is another instance of a system of $\frac{1}{2}$-shearlet molecules (at least after an appropriate
re-indexing). In particular, it falls within the framework of $\alpha$-molecules for $\alpha=\frac{1}{2}$.

\begin{prop}\label{prop:sshmol}
Appropriately re-indexed, the smooth Parseval frame of band-limited shearlets $SH$ constitutes a system of
$3$-dimensional $\frac{1}{2}$-shearlet molecules of order $(\infty,\infty,\infty,\infty)$.
\end{prop}

\begin{proof}
We first re-index the coarse-scale functions \eqref{eq:coarsescale} as well as the boundary elements \eqref{eq:bndshear3} and \eqref{eq:bndshear4} at scale $j=0$.
For this we utilize the index set $\Gamma \subset  \{0,1,2,3\}\times\N_0\times\Z^2 $ given by
\[
\Gamma:=\big\{ (0,0,{\bf0}) \big\} %\in \N_0\times\N_0\times\N_0^2 \}
\cup \big\{ (\varepsilon,0,\ell) : \varepsilon\in\{1,2,3\}, \ell=(\pm 1,0)  \big\} \cup \big\{ (1,0,\ell):  \ell=(\pm 1,\pm1) \big\}.
\]
The functions we want to re-index are then precisely those functions $\psi^\varepsilon_{j,\ell,k}\in SH$ with $(\varepsilon,j,\ell,k) \in \Delta:=\Gamma\times\Z^3$.
It holds $\#\Gamma=11$, and we enumerate the set $\Gamma$ in lexicographic order from $0$ to $10$,
described by a bijective function $\mathcal{N}:\Gamma\rightarrow\{0,\ldots,10\}$.
For $(\varepsilon,j,\ell,k) \in \Delta$ we then re-index as follows.
Writing $\tilde{k}(\varepsilon,j,\ell,k):=(k_1,k_2,11 \cdot k_3+\mathcal{N}(\varepsilon,j,\ell))\in\Z^3$
the re-indexed elements are given by
\[
\widetilde{\psi}^0_{0,{\bf0},\tilde{k}(\varepsilon,j,\ell,k)}:=\psi^\varepsilon_{j,{\bf\ell},k} \quad\text{for }(\varepsilon,j,\ell,k) \in \Delta.
\]
For $(\varepsilon,j,\ell,k) \in \Lambda_{SH}\backslash \Delta$ the functions $\tilde{\psi}^\varepsilon_{j,{\bf\ell},k}:=\psi^\varepsilon_{j,{\bf\ell},k}$ remain the same.
The relabelling is thus given by
\begin{align}\label{eq:relabel}
F_{label}:\Lambda_{SH}\rightarrow \widetilde{\Lambda}_{SH},\:
(\varepsilon,j,\ell,k)\mapsto \begin{cases} (0,0,{\bf0},\tilde{k}(\varepsilon,j,\ell,k)) \,&,(\varepsilon,j,\ell,k)\in\Delta; \\ (\varepsilon,j,\ell,k) &,(\varepsilon,j,\ell,k)\notin\Delta. \end{cases}
\end{align}
The newly obtained system
\begin{align}\label{eq:reindexed}
\widetilde{SH}:=\big\{\widetilde{\psi}^\varepsilon_{j,\ell,k} \big\}_{(\varepsilon,j,\ell,k)\in \widetilde{\Lambda}_{SH}}
\end{align}
is equipped with a shearlet index set $\tilde{\Lambda}_{SH}\subset \{0,1,2,3\}\times\N_0\times\Z^2\times\Z^3$ similar to \eqref{eq:sshindex},
however with the modified shear parameters $\mathscr{L}_{1,0}=\mathscr{L}_{2,0}=\mathscr{L}_{3,0}=\{ (0,0)\}$ at scale $j=0$.

In the remainder we show that \eqref{eq:reindexed} is a system of $\frac{1}{2}$-shearlet molecules of
order $(\infty,\infty,\infty,\infty)$.
As parameters we fix $\sigma=4$ and $\eta_j=2^{-j}$ for $j\in\N_0$.
The translation parameters $\mathcal{T}=\diag(\tau_1,\tau_2,\tau_3)$ vary with the indices and are chosen suitably later.
Clearly, we have $S_{\ell} A^j_{\frac{1}{2},4}=A^j_{\frac{1}{2},4} S_{\ell\eta_j}$.
Thus, defining
\begin{align}\label{eq:generators}
\gamma^\varepsilon_{j,\ell,k}(x):= 2^{-2j} \tilde{\psi}^\varepsilon_{j,\ell,k}\big(Z^{\varepsilon}   A^{-j}_{\frac{1}{2},4} S^{-1}_{\ell} Z^{-\varepsilon} (x + \mathcal{T} k) \big), \quad x\in\R^3,
\end{align}
for every $(\varepsilon,j,\ell,k)\in\widetilde{\Lambda}_{SH}$ we get the desired representation \eqref{eq:sheargen}, i.e.\
\begin{align*}
	\tilde{\psi}_{j,\ell,k}^\varepsilon(x) = 2^{2j} \gamma_{j,\ell,k}^\varepsilon ( Z^{\varepsilon} S_{\ell} A^j_{\frac{1}{2},4} Z^{-\varepsilon} x- \mathcal{T} k)
    = 2^{2j} \gamma_{j,\ell,k}^\varepsilon ( Z^{\varepsilon} A^j_{\frac{1}{2},4} S_{\ell\eta_j}Z^{-\varepsilon} x- \mathcal{T} k), \quad x\in\R^3.
\end{align*}
On the Fourier side the generators \eqref{eq:generators} take the form
\begin{align}\label{Ffffff}
	\sigmafuncfor^\varepsilon(\xi)= 2^{2j}\hat{\tilde{\psi}}^\varepsilon_{j,\ell,k}\big( Z^{\varepsilon}  A^j_{\frac{1}{2}, 4} S^T_{\ell} Z^{-\varepsilon} \xi\big) \exp(2\pi i \sprod{\xi, \mathcal{T} k}), \quad \xi\in\R^3,
\end{align}
and it remains to show that these functions satisfy \eqref{eq:shearcon}.
For this task we distinguish between the coarse-scale elements, the interior shearlets and the boundary shearlets of \eqref{eq:reindexed}.

\paragraph*{\underline{Coarse-Scale:}}

\noindent
The coarse-scale elements $\{\tilde{\psi}^0_{0,{\bf0},k}\}_{k\in\Z^3}$ of \eqref{eq:reindexed} are precisely those functions
$\psi^\varepsilon_{j,{\bf\ell},k}\in SH$ where $(\varepsilon,j,\ell,k)\in \Delta$. For them \eqref{eq:generators} simplifies to
$
\gamma^0_{0,{\bf0},k}={\tilde{\psi}^0_{0,{\bf0},k}}(\cdot + \mathcal{T} k)%, \quad x\in\R^3,
$,
and we just need to show \eqref{eq:shearcon} for $\hat{\gamma}^0_{0,{\bf0},k} $ with $k\in\Z^3$.
For this $\mathcal{T}=\diag(1,1,\frac{1}{11})$ is the right choice.
A calculation yields for every $(\varepsilon,j,\ell,k)\in \Delta$ and $\tilde{k}=\tilde{k}(\varepsilon,j,\ell,k)\in\Z^3$

\begin{align*}
\hat{\gamma}^0_{0,{\bf0},\tilde{k}(\varepsilon,j,\ell,k)}(\xi)
%&={\hat{\tilde{\psi}}^0_{0,{\bf0},\tilde{k}}(\varepsilon,j,\ell,k)}(\xi) \exp( 2\pi i \langle \xi,\mathcal{T}\tilde{k}(\varepsilon,j,\ell,k) \rangle) \\
%&={\hat{\psi}^\varepsilon_{j,{\bf\ell},k}}(\xi)  \exp( 2\pi i \langle \xi,k \rangle) \exp( \textstyle{\frac{ 2\pi i}{11}} \xi_3 \cdot \mathcal{N}(\varepsilon,j,\ell))   \\
=\hat{\psi}^\varepsilon_{j,{\bf\ell},k}(\xi)   \exp( {\textstyle\frac{ 2\pi i}{11}} \xi_3 \cdot \mathcal{N}(\varepsilon,j,\ell)), \quad \xi=(\xi_1,\xi_2,\xi_3)^T\in\R^3.
\end{align*}
Hence, looking at the definitions \eqref{eq:coarsescale}, \eqref{eq:bndshear3}, \eqref{eq:bndshear4}, it follows
$\hat{\gamma}^0_{0,{\bf0},k}\in C^\infty_c(\R^3)$ with
support in $[-\frac{1}{2},\frac{1}{2}]^3$.

\paragraph*{\underline{Interior Shearlets:}}

Choosing $\mathcal{T}=\diag(1,1,1)$, Equation \eqref{Ffffff} together with \eqref{eq:psifordef}
yields for $\xi\in\R^3$
\begin{align*}
	\hat{\gamma}^{\varepsilon}_{j,\ell,k}(\xi) %&=  W(2^{-2j} Z^{\varepsilon}  A^j_{\frac{1}{2},4} S^T_{\ell} Z^{-\varepsilon} \xi ) V(Z^{-\varepsilon} \xi)  \exp(-2\pi i \sprod{\xi,  k}) \exp(2\pi i \sprod{\xi,  k}), \\
= W( Z^{\varepsilon} M_{j,\ell} Z^{-\varepsilon} \xi ) V(Z^{-\varepsilon} \xi),
\end{align*}
where the matrix $M_{j,\ell}:=2^{-2j} A^j_{\frac{1}{2},4} S^T_{\ell} $ has the form
\begin{align}\label{eq:Mmmmmmm}
	M_{j,\ell} = \begin{pmatrix}
	2^{-j} & 0 & \ell_1 2^{-j} \\
	0 & 2^{-j} & \ell_2 2^{-j}\\
	0 & 0 & 1
	\end{pmatrix}.
\end{align}
We check \eqref{eq:shearcon} exemplarily for the case $\varepsilon=3$ where $Z^3=I$. A calculation yields
\begin{align*}%\label{Mmmmmmm}
M_{j,\ell}\xi=\begin{pmatrix} 2^{-j}(\xi_1 + \ell_1\xi_3) \\ 2^{-j}(\xi_2 + \ell_2\xi_3) \\ \xi_3 \end{pmatrix}, \quad \xi=(\xi_1,\xi_2,\xi_3)^T\in\R^3.
\end{align*}
Subsequently we prove that for $j\in\N_0$ and $|\ell|_\infty \le 2^j$
\begin{align}\label{eq:support1}
\supp (\hat{\gamma}^{3}_{j,\ell,k}) \subseteq \mathcal{P}_3 \cap \Big\{ \xi\in\R^3 ~:~ \textstyle{ \frac{1}{32} \le |\xi_3| \le \frac{1}{2} } \Big\}.
\end{align}
First observe that clearly $\supp (\hat{\gamma}^{3}_{j,\ell,k}) \subseteq \mathcal{P}_3 $.
Now take $\xi\in \mathcal{P}_3$ and $\xi\notin \big\{ \xi\in\R^3:\frac{1}{32} \le |\xi_3| \le \frac{1}{2} \big\}$.
With $|\xi_3|>\frac{1}{2}$ also $|[M_{j,\ell}\xi]_3|>\frac{1}{2}$ and thus $M_{j,\ell}\xi\notin \supp W$.
If on the other hand $|\xi_3|<\frac{1}{32}$ then $|[M_{j,\ell}\xi]_3|<\frac{1}{32}$ and, since $j\ge 0$, $|\xi_1|\le|\xi_3|$ and $|\ell|_\infty\le 2^j$,
\begin{align*}
|[M_{j,\ell}\xi]_1|=|2^{-j}(\xi_1 + \ell_1\xi_3)| \le 2^{-j}(1 + |\ell_1|) |\xi_3| \le 2^{-j}(1 + 2^{j}) |\xi_3| \le (1 +  2^{-j}) |\xi_3| \le 2 |\xi_3| < {\textstyle\frac{1}{16}}.
\end{align*}
Analogously, one obtains $|[M_{j,\ell}\xi]_2|<\frac{1}{16}$. Altogether, this yields $|M_{j,\ell}\xi|_\infty < \frac{1}{16}$ if $|\xi_3|<\frac{1}{32}$,
which implies $M_{j,\ell}\xi\notin \supp W$.
Consequently, $\hat{\gamma}^{3}_{j,\ell,k}(\xi)=0$ outside $\mathcal{P}_3 \cap \{ \xi\in\R^3 : \textstyle{ \frac{1}{32} \le |\xi_3| \le \frac{1}{2} } \}$, which
proves \eqref{eq:support1}.

Using analogous estimates for $\varepsilon\in\{2,3\}$ it follows that
$
\supp \hat{\gamma}^{\varepsilon}_{j,\ell,k} \subseteq \textstyle{ [-\frac{1}{2}, \frac{1}{2} ]^3\backslash(-\frac{1}{32},\frac{1}{32})^3 }
$
for every $(\varepsilon,j,\ell,k)\in\Lambda_{int}=\{ (\varepsilon,j,\ell,k)\in \{1,2,3\}\times\N_0\times \Z^2 \times\Z^3 : |\ell|_\infty<2^j \}$.

The derivatives of $\hat{\gamma}^{\varepsilon}_{j,\ell,k}$ are linear combinations of functions
$\partial^{\beta}W(Z^\varepsilon M_{j,\ell} Z^{-\varepsilon}\xi)\partial^{\delta} V(Z^{-\varepsilon}\xi)$ with coefficients uniformly bounded, since the entries of $M_{j,\ell}$ are.
This fact together with the support condition implies that \eqref{eq:shearcon} is fulfilled for
arbitrary order.

\paragraph*{\underline{Boundary Shearlets:}}

\noindent
The boundary shearlets of \eqref{eq:reindexed} satisfy $j\ge1$ and are precisely the functions in \eqref{eq:bndshear1} and \eqref{eq:bndshear2}.
We exemplarily handle the functions \eqref{eq:bndshear1}, the argumentation for the functions \eqref{eq:bndshear2} is similar.
Let us first look at the boundary between $\mathcal{P}_1$ and $\mathcal{P}_3$, i.e.\ the functions $\tilde{\psi}^\varepsilon_{j,\ell,k}$,
where $\varepsilon=3$, $j\ge1$, $\ell_1=\pm 2^j$, $|\ell_2|<2^j$, and $k\in\Z^3$.
Plugging \eqref{eq:bndshear1} into \eqref{Ffffff} and choosing $\mathcal{T}=\frac{1}{4}\diag(1,1,1)$ we obtain
\begin{align*}
	\hat{\gamma}^{3}_{j,\ell,k}(\xi) = 2^{-3} \begin{cases}
      W(M_{j,\ell}  \xi ) V (\xi) &, \xi \in M^{-1}_{j,\ell}\mathcal{P}_3; \\
      W(M_{j,\ell}  \xi ) V (N_{j,\ell} \xi)    &, \xi \in M_{j,\ell}^{-1}\mathcal{P}_1;
      \end{cases}
\end{align*}
for the generators, where $M_{j,\ell}$ is the matrix \eqref{eq:Mmmmmmm} and where
\begin{align*}
	N_{j,\ell} := \begin{pmatrix}
	- \ell_1 2^{-j} & 0 & 2^{j}- \ell^2_1 2^{-j} \\
	-\sgn(\ell_1) \ell_2 2^{-j} & 1 & \ell_2 - 2^{-j} |\ell_1| \ell_2 \\
	2^{-j} & 0 & 2^{-j} \ell_1
	\end{pmatrix}.
\end{align*}
Similar calculations as for the interior shearlets then prove
\[
\supp(\hat{\gamma}^{3}_{j,\ell,k})\subseteq\Big\{ \xi\in\R^3 ~:~ \textstyle{|\frac{\xi_1}{\xi_3}|\le 2, |\frac{\xi_2}{\xi_3}| \le 4} \Big\} \cap \Big\{ \xi\in\R^3 ~:~ {\textstyle \frac{1}{64} \le |\xi_3| \le\frac{1}{2}}\Big\}.
\]
For $\varepsilon\in\{1,2\}$ complementary results hold true.
Altogether, it follows
$
\supp \hat{\gamma}^{\varepsilon}_{j,\ell,k} \subseteq \textstyle{[-2,2]^3\backslash(-\frac{1}{64},\frac{1}{64})^3}
$
for the generators of the functions \eqref{eq:bndshear1}. Hence, condition \eqref{eq:shearcon} is fulfilled for arbitrary order.
\end{proof}

Proposition~\ref{prop:sshmol} in particular shows that the system $SH$ is a system of $\frac{1}{2}$-molecules, however not with respect to a shearlet parametrization
because of the necessary relabelling $F_{label}$ in \eqref{eq:relabel}. The actual parametrization is given by
\[
\Phi_{SH}:=\widetilde{\Phi}_{SH} \circ F_{label},
\]
where $\widetilde{\Phi}_{SH}$ denotes the $\frac{1}{2}$-shearlet parametrization of the relabeled system.

\begin{cor}
The smooth shearlet frame $SH$ is a system of $3$-dimensional $\frac{1}{2}$-molecules of order $(\infty,\infty,\infty,\infty)$
with respect to the parametrization $(\Lambda_{SH},\Phi_{SH})$.
\end{cor}

\begin{bem} \label{rem:consist}
Although $(\Lambda_{SH},\Phi_{SH})$ is not a shearlet parametrization, it clearly is $(\frac{1}{2},k)$-consistent with every $\frac{1}{2}$-shearlet parametrization for $k>3$.
This follows from Proposition~\ref{prop:shconsist} and the observation that relabelling of elements does not make any difference here.
\end{bem}

\subsubsection{Compactly Supported Pyramid-adapted Shearlets}

There also exist shearlet frames for $L^2(\R^2)$ and $L^2(\R^3)$ consisting of compactly supported functions.
Compactly supported frames of the form \eqref{eq:affineshear} have been constructed in \cite{Kutyniok2012,Kittipoom2010}.
They are also instances of $\frac{1}{2}$-shearlet molecules and their order can be controlled by the regularity of the generators.

\begin{prop} \label{prop:compactShearlets}
	Let $\phi,\psi^1,\psi^2,\psi^3\in L^2(\R^3)$ be compactly supported and $L,M,N_1,N_2\in\N_0\cup\{\infty\}$. If
    $\phi\in C^{N_1+N_2}(\R^3)$ and if, for every $\varepsilon\in\{1,2,3\}$,
	\begin{enumerate}[(i)]
	\item the derivatives $\partial^{\gamma}\psi^\varepsilon$ exist and are continuous for every $\gamma\in \N_0^3$ with $[Z^\varepsilon\gamma]_1 , [Z^\varepsilon\gamma]_2 \leq N_1+N_2$
          and $[Z^\varepsilon\gamma]_3 \leq N_1$, where $Z$ is the cyclic permutation matrix \eqref{eq:cyclicPerm},
	\item the generator $\psi^\varepsilon$ has $M+L$ directional vanishing moments in $e_\varepsilon$-direction, i.e.
    \begin{align*}%\label{eq:directmom}
    \forall (x_1,x_2)\in\R^2:  \int_\R \psi^\varepsilon(Z^\varepsilon x)x_3^N \,dx_3 = 0 \quad\text{for every $N\in\{0,\ldots,M+L-1\}$},
    \end{align*}
	\end{enumerate}
	 then the system \eqref{eq:affineshear} obtained from these generators is a system of {$\frac{1}{2}$-shearlet} molecules of order $(L,M,N_1,N_2)$.
\end{prop}
\begin{proof}
    It is obvious that -- rightly indexed -- a system of the form \eqref{eq:affineshear} constitutes a system of \linebreak $\frac{1}{2}$-shearlet molecules.
    Hence we just need to verify the order of the system. For this, little more is needed than utilizing the facts that
    spatial decay implies smoothness in Fourier domain (and vice versa), and that vanishing moments in spatial domain implies estimates of the form $\abs{\hat{g}(\xi)}\lesssim \min(1, \abs{\xi})^M$ in Fourier domain. We refer to \cite[Proposition 3.11]{Grohs2011} for details, where a similar two-dimensional version of the theorem is proven.
\end{proof}

%%%%%%%%%%%%%%%%%%%%%%%%%%%%%%%%%%%%%%%%%%%%%%%%%%%%%%%%%%%%%%%%%%%%%%%%%%%%%%%%%%%%%%%%%%%%%%%%%%%%%%%%%%%%%%%
%%%%%%%%%%%%%%%%%%%%%%%%%%%%%%%%%%%%%%%%%%%%%%%%%%%%%%%%%%%%%%%%%%%%%%%%%%%%%%%%%%%%%%%%%%%%%%%%%%%%%%%%%%%%%%%
%%%%%%%%%%%%%%%%%%%%%%%%%%%%%%%%%%%%%%%%%%%%%%%%%%%%%%%%%%%%%%%%%%%%%%%%%%%%%%%%%%%%%%%%%%%%%%%%%%%%%%%%%%%%%%%
%%%%%%%%%%%%%%%%%%%%%%%%%%%%%%%%%%%%%%%%%%%%%%%%%%%%%%%%%%%%%%%%%%%%%%%%%%%%%%%%%%%%%%%%%%%%%%%%%%%%%%%%%%%%%%%
%%%%%%%%%%%%%%%%%%%%%%%%%%%%%%%%%%%%%%%%%%%%%%%%%%%%%%%%%%%%%%%%%%%%%%%%%%%%%%%%%%%%%%%%%%%%%%%%%%%%%%%%%%%%%%%

%---------------------------------------------------------------%

\section{Proof of Theorem~\ref{thm:almorth}} \label{sec:proofmain}

%---------------------------------------------------------------%

This final section is devoted to the technical proof of Theorem~\ref{thm:almorth}. It is split up into several pieces and has the same general structure as the proof of the corresponding 2-dimensional result \cite[Theorem~4.2]{GKKS2014}.
In $d$ dimensions however it takes more effort and the arguments are more involved.

%---------------------------------------------------------------%

\subsection{Auxiliary Lemmas}

 Let us first collect some simple elementary facts, which turn out to be useful.
Subsequently $\mathcal{O}(d,\R)$ shall denote the orthogonal group of $\R^d$. Recall also the notation $\{\theta\}$ for the `projection' of
$\theta\in\R$ onto the interval $[-\frac{\pi}{2},\frac{\pi}{2})$ in the sense of Subsection~\ref{ssec:index}.

\begin{lemm} \label{lem:sinest}
For $\theta \in \R$ let $\{\theta\}$ denote its `projection' onto the interval $[-\frac{\pi}{2},\frac{\pi}{2})$ as introduced in Subsection~\ref{ssec:index}. It then holds
$\abspi{\theta} \asymp \abs{\sin(\theta)}$.
\end{lemm}
\begin{proof}
Due to $\pi$-periodicity it suffices to verify the relation for $\theta\in[-\frac{\pi}{2},\frac{\pi}{2})$. In this range we have
$
\frac{2}{\pi} |\theta| \le |\sin(\theta)| \le |\theta|.
$
\end{proof}

\noindent
An immediate corollary is the following result. Recall the notation $d_\sph(v,w)$ for
the angle $\arccos(\langle v,w\rangle)\in[0,\pi]$ between two vectors $v,w\in\mathbb{S}^{d-1}$.

\begin{lemm} \label{lem:raygeo}
Let $e_d\in\R^d$ be the $d$:th unit vector. For $\eta\in\mathbb{S}^{d-1}$ we have $\abspi{d_\sph(\eta,e_d)} \asymp \otnorm{\eta}$.
\end{lemm}
\begin{proof}
    Using a suitable rotation $R\in\mathcal{O}(d,\R)$ of the form
    \[
    R=\begin{pmatrix}
    R_{d-1} & 0 \\ 0 & 1
    \end{pmatrix},
    \]
    where $R_{d-1}\in \mathcal{O}(d-1,\R)$, we can achieve
	$
		R\eta=(\sin(\theta), 0,  \ldots, 0, \cos(\theta))^T
	$
    with $\theta=d_\sph(\eta, e_d)$.
    Since $|\eta|_{[d-1]}=|\eta - e_d|_{[d-1]}=|R(\eta - e_d)|_{[d-1]}=|R\eta -e_d|_{[d-1]}=|\sin(\theta)|$, it just remains to prove
	$\abs{\sin(\theta)} \asymp \abspi{\theta}$, which is true by Lemma~\ref{lem:sinest}.
\end{proof}

\begin{lemm} \label{lem:metrEquiv}
Let $c>0$ be a constant. Then we have for all $v,w\in\mathbb{S}^{d-1}$ with
$[v]_d\ge c$ and $[w]_d\ge 0$
\begin{align*}
	|\{ d_{\mathbb{S}}(v,w) \}| \asymp |v-w|.
\end{align*}
\end{lemm}
\begin{proof}
Under the assumptions there exists $\varepsilon>0$ dependent on $c$, such that
$
0\le d_{\mathbb{S}}(v,w) \le \pi-\varepsilon.
$
It follows $ \frac{\varepsilon}{\pi-\varepsilon} | d_{\mathbb{S}}(v,w) |   \le |\{ d_{\mathbb{S}}(v,w) \}| \le | d_{\mathbb{S}}(v,w) | $.
The observation
$
d_{\mathbb{S}}(v,w) \asymp |v-w|
$
finishes the proof.
\end{proof}

\begin{lemm}
\label{lem:rotest}
Let $R\in\mathcal{O}(d,\R)$ be a rotation and $\theta_0=d_{\mathbb{S}}(e_d, R e_d)\in[0,\pi]$ the angle between the d:th unit vector $e_d\in\R^d$ and its image $Re_d$ under $R$.
Then it holds for all $\eta\in\mathbb{S}^{d-1}$
\begin{align*}
|R \eta|_{[d-1]} = \sin(d_{\mathbb{S}}(R \eta, e_d)) \geq \min\{\abs{\sin(d_{\mathbb{S}}(\eta,e_d) + \theta_0)}, \abs{\sin(d_{\mathbb{S}}(\eta,e_d) - \theta_0)}\}.
\end{align*}
\end{lemm}
Note $d_{\mathbb{S}}(\eta,e_d)=d_{\mathbb{S}}(R\eta,Re_d)$.
\begin{proof}
    Let $\eta=(\eta_1,\ldots,\eta_d)^T\in\mathbb{S}^{d-1}$ and put $\theta_1:=d_{\mathbb{S}}(\eta, e_d)= \arccos(\langle\eta,e_d\rangle)\in[0,\pi]$.
    The rotation $R\in\mathcal{O}(d,\R)$ can be decomposed in the form
    $
    R=\tilde{R} R_{\theta_0}
    $
    with $\tilde{R},R_{\theta_0}\in\mathcal{O}(d,\R)$ such that
     \[
     \tilde{R}=\begin{pmatrix} R_{d-1} & \\   & 1  \end{pmatrix} \quad\text{and}\quad  R_{\theta_0}=\begin{pmatrix} \cos(\theta_0) & & -\sin(\theta_0) \\  & I_{d-2} & \\ \sin(\theta_0) & & \cos(\theta_0) \end{pmatrix},
     \]
     where $R_{d-1}\in\mathcal{O}(d-1,\R)$ is some $(d-1)$-dimensional rotation matrix and $I_{d-2}$ is the $(d-2)$-dimensional identity matrix. The rotation $\tilde{R}$ leaves $|\cdot|_{[d-1]}$ invariant, whence
     \[
     |R \eta|_{[d-1]} = |\tilde{R} R_{\theta_0} \eta|_{[d-1]} =  |R_{\theta_0} \eta|_{[d-1]}.
     \]
    Using $\eta_d=\cos(\theta_1)$ and $|\eta|^2_{[d-1]}=\eta_1^2 +  \eta_2^2 + \ldots + \eta_{d-1}^2 =1- \eta_d^2$, it further follows
	\begin{align*}
		\otnorm{R_{\theta_0}\eta }^2&= ( \cos(\theta_0) \eta_1 - \sin(\theta_0)\eta_d)^2 + \eta_2^2 + \dots + \eta_{d-1}^2 \\
                                                     &=  \cos^2(\theta_0)\eta_1^2 + \sin^2(\theta_0)\cos^2(\theta_1) - 2 \cos(\theta_0)\sin(\theta_0)\eta_1\cos(\theta_1) + (1 - \eta_1^2 - \cos^2(\theta_1)) \\
                                                     &= 1 - ( \eta_1\sin(\theta_0) + \cos(\theta_1)\cos(\theta_0))^2.
	\end{align*}
	The last expression is a second-degree polynomial in the variable $\eta_1$ with a negative leading coefficient. Since
    $\eta_1^2 \le 1 -\eta_d^2 = 1- \cos^2(\theta_1) = \sin^2(\theta_1)$, the variable $\eta_1$ can take values only in
	$  [- \sin(\theta_1), \sin(\theta_1)]$. The polynomial attains its minimum on this interval at the endpoints.
    Hence, we can conclude
	\begin{align*}
	 \otnorm{R_{\theta_0}\eta }^2 &\ge \min_{\epsilon\in\{-1,1\}}  \big\{1 - ( \epsilon\sin(\theta_1)\sin(\theta_0) + \cos(\theta_1)\cos(\theta_0))^2 \big\} \\
      &=  \min_{\epsilon\in\{-1,1\}} \big\{ 1- \cos^2(\theta_1 - \epsilon \theta_0) \big\} = \min_{\epsilon\in\{-1,1\}} \big\{ \sin^2(\theta_1 - \epsilon \theta_0) \big\},
	\end{align*}
	which proves the claim.
\end{proof}

%---------------------------------------------------------------%

\subsection{Integral Estimates}

We start with an estimate for the generators in \eqref{eq:molgen}, which will allow us to work in polar coordinates.

\begin{lemm} \label{alem:forParaEst}
Let the family of functions $\{g_{\lambda}\}_{\lambda\in\Lambda}$ satisfy \eqref{eq:molcon} uniformly for a multi-index $\rho\in\N^d_0$,
and assume that there is a constant $c>0$ such that $s_\lambda \geq c$ for all $\lambda\in\Lambda$.
Then the following estimate holds true uniformly for $\lambda\in\Lambda$ and $\xi\in\R^d$
\begin{align} \label{paraproprad}
	\abs{(\partial^{\rho}\hat{g}_{\lambda})( \matricol{\lambda} \xi)} \lesssim \wfuncrad{\lambda}{M}{N_1}{N_2}.
\end{align}
\end{lemm}
\begin{proof}
We have $|A^{-1}_{\alpha,s_\lambda}\xi| \geq \min\{s_\lambda^{-1},s_\lambda^{-\alpha}\}\abs{\xi} \gtrsim s_\lambda^{-1}\abs{\xi}$
uniformly for $\xi \in \R^d$ and $\lambda\in\Lambda$, since $s_\lambda \geq c >0$ for every $\lambda\in\Lambda$.
It follows $|\matricol{\lambda} \xi| \gtrsim s_\lambda^{-1} \abs{\rotmatr{\lambda}\xi} = s_\lambda^{-1} |\xi| $.  Further,
we observe
$|A^{-1}_{\alpha,s_\lambda}\xi|_{[d-1]} = s_\lambda^{-\alpha }\otnorm{\xi}$ and $|[A^{-1}_{\alpha,s_\lambda}\xi]_d| = s_\lambda^{-1}\abs{[\xi]_d}$.
Finally, it holds $\angles{|\xi|} \asymp 1 + |\xi|$ and $\abs{[\xi]_d} + \otnorm{\xi} \asymp |\xi|$. Collecting all of these estimates, one obtains
\begin{flushright}
\begin{align*}
\big\vert \big(\partial^{\rho} \hat{g}_{\lambda}\big)( \matricol{\lambda} \xi) \big\vert
&\lesssim \frac{\min\Big\{ 1, s_{\lambda}^{-1} + |[\matricol{\lambda}\xi]_d|
+ s_\lambda^{-(1-\alpha)}|\matricol{\lambda}\xi|_{[d-1]}\Big\}^{M}}{\angles{|\matricol{\lambda}\xi|}^{N_1}\angles{|\matricol{\lambda}\xi|_{[d-1]}}^{N_2}}  \\
 &\lesssim \wfuncrad{\lambda}{M}{N_1}{N_2}.
\end{align*}
\end{flushright}
\end{proof}

\noindent
The expression on the right hand side of \eqref{paraproprad} can further be estimated by the function
\begin{align}\label{eq:sepfunc}
	S_{\lambda, M,N_1,N_2}(\xi) := \Sfunc{\lambda}{M}{N_1}{N_2}, \quad \xi\in\R^d.
\end{align}
As already discussed in \cite{GKKS2014}, this function can be separated into angular and radial components, allowing us to treat these parts independently in the integration later.
Since the notation in this article differs slightly from the one in \cite{GKKS2014}, we choose to state the lemma once more here, but refer to said article for a proof.

\begin{lemm} \label{zerlegungslemma} (\cite[Lemma 6.4]{GKKS2014})
Assume that $s_\lambda \geq c >0$ for all $\lambda\in\Lambda$. For every $M,N_1,N_2,K\in\N_0$ such that $K \leq N_2$ we have
with respect to $\lambda\in\Lambda$ and $\xi\in\R^d$ the uniform estimate
\begin{align*} %\label{zerlegung}
	\wfuncrad{\lambda}{M}{N_1}{N_2} \lesssim S_{\lambda, M-K, N_1 ,K}(\xi).
\end{align*}
\end{lemm}

Next, we want to estimate the scalar product of two functions of the form \eqref{eq:sepfunc}.
Before the actual result, Lemma~\ref{intWholeEstimate}, we need some preparation.
This is the part of the proof of Theorem~\ref{thm:almorth} which differs the most from the situation in two dimensions.

\begin{lemm} \label{intLinSphereEstimate} Let $a\ge a'>0$, $d\in\N$, $d \geq 2$, and $N>1$. Then we have uniformly for $y \in \R$
\begin{align*}
	\int_{\R} \frac{\abs{x}^{d-2}dx}{(1+a \abs{x})^{N+d-2}(1+a'\abs{x-y})^{N+d-2}} \lesssim a^{-(d-1)} (1+ a'\abs{y})^{-N}.
\end{align*}
\end{lemm}
\begin{proof}
Utilizing the following result from Grafakos \cite{Grafakos2008}[Appendix K.1]
\[
\int_{\R} \frac{dx}{(1+a\abs{x})^{N}(1+a'\abs{x-y})^{N}}\lesssim \max\{a,a'\}^{-1} (1+ \min\{a,a'\}\abs{y})^{-N}
\]
we can estimate
\begin{align*}
\int_{\R} &\frac{\abs{x}^{d-2}dx}{(1+a \abs{x})^{N+d-2}(1+a'\abs{x-y})^{N+d-2}}
= a^{-(d-2)} \int_{\R} \frac{\abs{ax}^{d-2}dx}{(1+a \abs{x})^{N+d-2}(1+a'\abs{x-y})^{N+d-2}} \\
& \le a^{-(d-2)} \int_{\R} \frac{(1+\abs{ax})^{d-2}dx}{(1+ a\abs{x})^{N+d-2}(1+a'\abs{x-y})^{N+d-2}}
= a^{-(d-2)} \int_{\R} \frac{dx}{(1+ a\abs{x})^{N}(1+a'\abs{x-y})^{N+d-2}} \\
& \lesssim a^{-(d-2)} \max\{a,a'\}^{-1} (1+\min\{a,a'\}\abs{y})^{-N}
= a^{-(d-1)}  (1+ a'\abs{y})^{-N}.
\end{align*}
\end{proof}

We can immediately deduce the following corollary.

\begin{cor}\label{cor:intSphereEstimate}
Let $a\ge a'>0$, $d\in\dimension$, and $N>1$. Then we have uniformly for $\theta_0\in\R$
\begin{align*}
	\int_0^\pi \frac{|\sin^{d-2}(\theta)|\,d\theta}{(1+a\abs{\sin(\theta)})^{N+d-2}(1+a'\abs{\sin(\theta - \theta_0)})^{N+d-2}} \lesssim a^{-(d-1)}\big(1+a'\abspi{\theta_0}\big)^{-N}.
\end{align*}
\end{cor}
\begin{proof}
Let us call the integral to be estimated $\mathcal{S}$. Since the integrand on the left hand side is $\pi$-periodic,  we may change the domain of integration to $[-\pi/2, \pi/2]$. Applying Lemma~\ref{lem:sinest}, we can further conclude
\begin{align*}
\mathcal{S}\asymp \int\limits_{-\pi/2}^{\pi/2} \frac{\abs{\theta}^{d-2}d \theta}{(1 + a\abs{\theta})^{N+d-2}(1+ a'\abspi{\theta-\theta_0})^{N+d-2}},
\end{align*}

Since $|\{\theta_0\}|\le\frac{\pi}{2}$ we can estimate
\begin{align*}
\mathcal{S}\lesssim \sum_{\vartheta\in\{-\pi,0,\pi\}}	 \int_{\R} \frac{\abs{\theta}^{d-2} d\theta}{(1+a\abs{\theta})^{N+d-2}(1+a'\abs{\theta - (\{\theta_0\} + \vartheta)})^{N+d-2}},
\end{align*}
We now use Lemma~\ref{intLinSphereEstimate} to estimate this by
\begin{align*}
\mathcal{S}\lesssim \sum_{\vartheta\in\{-\pi,0,\pi\}}	 a^{-(d-1)} (1+a'\abs{\{\theta_0\} + \vartheta})^{-N}\lesssim  a^{-(d-1)}(1+ a'\abspi{\theta_0})^{-N}.
\end{align*}
\end{proof}

This result is used to estimate the integral of the angular parts of \eqref{eq:sepfunc} over the sphere $\sph^{d-1}$.

\begin{lemm} \label{realIntSphereEstimate}
Let $a, a'>0$, {$d\in\N$, $d \geq 2$}, $\theta_{\lambda}, \theta_{\mu} \in [0,\pi]\times[-\frac{\pi}{2},\frac{\pi}{2}]^{d-3}$, $\varphi_{\lambda} , \varphi_{\mu} \in [0, 2\pi]$ and $N>1$.
Further, let $d\sigma$ denote the standard surface measure on the sphere $\sph^{d-1}$.
We then have the estimate
\begin{align*}
	\int_{\sph^{d-1}} \frac{d \sigma( \eta)}{(1+ a |\rotmatr{\mu}\eta|_{[d-1]})^{N+d-2}(1+a'\otnorm{\rotmatr{\lambda}\eta })^{N+d-2}} \\
	\lesssim \max\{a,a'\}^{-(d-1)} \big(1+\min\{a,a'\}|\{d_\sph(e_\lambda, e_\mu)\}| \big)^{-N},
\end{align*}
where  $e_\lambda=R^T_{\varphi_\lambda} R^T_{\theta_\lambda}e_d $ and $e_\mu=R^T_{\varphi_\mu} R^T_{\theta_\mu}e_d $.
\end{lemm}
\begin{proof}
    Note the symmetry of the statement with respect to interchanging the entities $a, a'$ and $\lambda, \mu$. Without loss of generality we can therefore restrict to the case
    $a\ge a'>0$.

    Since the mapping $\rotmatr{\mu}$ is an isometry, the integral is equal to
	\begin{align*}%\label{eq:sphereintegral}
    \mathcal{S}:=\int_{\sph^{d-1}} \frac{d \sigma(\eta)}{(1+ a \otnorm{\eta})^{N+d-2}(1+a'\vert R_{\theta_\lambda}R_{\varphi_\lambda}R^T_{\varphi_\mu}R^T_{\theta_\mu} \eta \vert_{[d-1]})^{N+d-2}}.
	\end{align*}
    For the integration we parameterize the sphere $\mathbb{S}^{d-1}$ by standard spherical coordinates, i.e.\ coordinates $(\theta_1, \dots \theta_{d-2}, \varphi)\in [0,\pi]^{d-2}\times [0,2\pi)$ such that
    for $\eta\in\mathbb{S}^{d-1}$
	\begin{align*}
	\eta(\theta,\varphi) = \begin{pmatrix}
	\sin(\theta_1) \cdots \cdots\sin(\theta_{d-2})\cos(\varphi) \\
		\sin(\theta_1) \cdots \cdots\sin(\theta_{d-2})\sin(\varphi) \\
			\sin(\theta_1) \cdots \sin(\theta_{d-3})\cos(\theta_{d-2}) \\
			\vdots\\
			\cos(\theta_1)
	\end{pmatrix}.
	\end{align*}
    Observe that $\langle \eta, e_d\rangle=\cos(\theta_1)$ and thus $\theta_1=d_{\mathbb{S}}(\eta,e_d)$.
    Also note $|\eta|_{[d-1]}=|\sin(\theta_1)|$.
    Letting $\theta_0:=d_{\mathbb{S}}(e_\lambda,e_\mu )\in[0,\pi]$ denote the angle between $e_\lambda$ and $e_\mu$ we have
    $\theta_0 =d_{\sph}( e_d, R_{\theta_\lambda}R_{\varphi_\lambda}R^T_{\varphi_\mu}R^T_{\theta_\mu}  e_d)$.
    Since $R_{\theta_\lambda}R_{\varphi_\lambda}R^T_{\varphi_\mu}R^T_{\theta_\mu} \in\mathcal{O}(d,\R)$ we can apply Lemma~\ref{lem:rotest} to estimate
    $| R_{\theta_\lambda}R_{\varphi_\lambda}R^T_{\varphi_\mu}R^T_{\theta_\mu} \eta|_{[d-1]} $.
    We obtain
		\begin{align*}
			%\int_{\sph^{d-1}} &\frac{d \sigma(\eta)}{(1+ a \otnorm{\eta})^{N+d-2}(1+a'\vert T_{\lambda,\mu}\eta \vert_{[d-1]})^{N+d-2}}  \\
			\mathcal{S} &\leq  \int_0^{2\pi} \int_0^\pi\dots \int_0^\pi \frac{\sin^{d-2}(\theta_1)\sin^{d-3}(\theta_2)\dots\sin(\theta_{d-2}) d\theta_1 d\theta_2 \dots d\theta_{d-2} d\varphi}{(1+ a
          |\sin(\theta_1)|)^{N+d-2}(1+a'\min\{\abs{\sin(\theta_1 + \theta_0)}, \abs{\sin(\theta_1 - \theta_0)}\})^{N+d-2}} \\
          &\lesssim \int_0^\pi \frac{\sin^{d-2}(\theta_1)\, d\theta_1 }{(1+ a
          |\sin(\theta_1)|)^{N+d-2}(1+a'\min\{\abs{\sin(\theta_1 + \theta_0)}, \abs{\sin(\theta_1 - \theta_0)}\})^{N+d-2}} \\
          &\le \sum_{\epsilon\in\{-1,1\}} \int_0^\pi \frac{|\sin(\theta_1)|^{d-2}\, d\theta_1 }{(1+ a
          |\sin(\theta_1)|)^{N+d-2}(1+a'\abs{\sin(\theta_1 - \epsilon \theta_0)})^{N+d-2}}.
		\end{align*}
		Using Corollary~\ref{cor:intSphereEstimate} we finally arrive at $\mathcal{S} \lesssim \max\{a,a'\}^{-(d-1)} \big(1+\min\{a,a'\}\abspi{\theta_0}\big)^{-N}$.
\end{proof}

With this estimate for the angular components in our toolbox, we proceed to prove the main result concerning the correlation of
functions of the form \eqref{eq:sepfunc}.

\begin{lemm} \label{intWholeEstimate}
Let $\alpha\in[0,1]$, $d\in\N$, $d \geq 2$, and $M,N_1,N_2\in\N_0$. Further, let $(\Lambda,\Phi_\Lambda)$ and $(\Delta,\Phi_\Delta)$ be parametrizations
with $(s_\lambda,e_\lambda,x_\lambda)=\Phi_\Lambda(\lambda)$ and $(s_\mu,e_\mu,x_\mu)=\Phi_\Delta(\mu)$ for $\lambda\in\Lambda$, $\mu\in\Delta$,
such that $c\le s_\lambda$ and $c\le s_\mu$ for a fixed constant $c>0$.
Then for $A>0$ and $B>1$ satisfying
\begin{align*}
	N_1> \frac{d}{2}, \quad M +d > N_1 \geq  A + \frac{1+\alpha(d-1)}{2}, \quad\text{and}\quad  N_2 \geq B+d-2
\end{align*}
the following estimate holds true with an implicit constant independent of $\lambda\in\Lambda$ and $\mu\in\Delta$,
\begin{align*}
(s_\lambda s_\mu)^{-\frac{1+\alpha(d-1)}{2}}\int\limits_{\R^d} S_{\lambda,M,N_1,N_2}(x) S_{\mu,M,N_1,N_2}(x) \,dx
\lesssim \max\left\{ \frac{s_\lambda}{s_\mu} , \frac{s_\mu}{s_\lambda} \right\}^{- A}\left( 1 + \min\{s_\lambda,s_\mu\}^{1-\alpha}|\{d_\sph(e_\lambda, e_\mu)\}| \right)^{-B}.
\end{align*}
\end{lemm}
\begin{proof}
Without loss of generality we subsequently assume $s_\lambda \leq s_{\mu}$. The strategy is to separate the integration into an angular and a radial part and estimate these
independently. For the estimate of the angular part we can use Lemma~\ref{realIntSphereEstimate}, which yields
\begin{align*}	
(s_\lambda s_\mu)^{-\frac{(1+\alpha(d-1)}{2}} \int_{0}^\infty \int_{\sph^{d-1}}  S_{\lambda, M, N_1, N_2}(\eta,r) S_{\mu, M, N_1, N_2}(\eta,r) r^{d-1} \,d\sigma(\eta) dr \\
	\lesssim    (s_\lambda s_\mu)^{-\frac{1+\alpha(d-1)}{2}}    s_\mu^{-(1-\alpha)(d-1)} s_\mu^d  \big(1+ s_\lambda^{1-\alpha} |\{d_\sph(e_\lambda, e_\mu)\}|\big)^{-B}
	\cdot\, \mathcal{S}
\end{align*}
with a remaining radial integral
\[
\mathcal{S}:= s_\mu^{-d} \int_{0}^\infty \SfuncRad{\lambda}{M}{N_1}\SfuncRad{\mu}{M}{N_1} r^{d-1} \,dr.
\]
Note that for the estimate we used the assumptions $s_\lambda \leq s_\mu$, $B>1$ and $N_2 \ge B+d-2$.
	It remains to verify the relation
	$(s_\mu s_\lambda)^{-(1+\alpha(d-1))/2} s_\mu^{-(1-\alpha)(d-1)} s_\mu^d \cdot \mathcal{S}  \lesssim ( s_\mu/s_\lambda)^{-A}$,
    or equivalently
    \begin{align*}
	  \mathcal{S} \lesssim \Big(\frac{s_\mu}{s_\lambda}\Big)^{-A-\frac{1+\alpha(d-1)}{2}}.
	\end{align*}
	To prove this, we split the integration of $\mathcal{S}$ into three parts $\mathcal{S}_1,\mathcal{S}_2,\mathcal{S}_3$
    corresponding to the integration ranges $0\leq r \leq 1$, $1 \leq r \leq s_\mu$, and $s_\mu \leq r$ respectively.
	
	\paragraph*{\underline{$0 \leq r \leq 1$}:}
    Here we estimate
    $\min\left\{1, s_{\lambda}^{-1}(1+r)\right\}^{M} \leq s_{\lambda}^{- M}(1+r)^{M} \leq 2^M s_\lambda^{- M}$
	 and $(1 +s_{\lambda}^{-1}r)^{N_1} \geq 1$, and similarly for the index $\mu$. Hence, the integral over this part can be estimated by
    \begin{align*}
		\mathcal{S}_1 \lesssim s_\mu^{-d} s_\lambda^{- M} s_\mu^{- M} \int_0^1 r^{d-1} \,dr \asymp  s_\mu^{- (M+d)} s_\lambda^{- M} \lesssim \left( \frac{s_\mu}{s_\lambda}\right)^{-(M+d)},
	\end{align*}
    where the last inequality holds because of the uniform lower bound $0<c\le s_\lambda$ for $\lambda\in\Lambda$.
	Finally observe that the assumed inequalities imply $M + d >A + \frac{1+\alpha(d-1)}{2}$.
	
	\paragraph*{\underline{$1 \leq r \leq s_\mu$}}
    We estimate the terms involving $\mu$ as follows: $(1 +s_{\mu}^{-1}r)^{N_1}\ge 1$ and $(r+1)\leq 2r$. Hence
	\begin{align*}
	\min\left\{1, s_\mu^{-1}(1+r)\right\}^{M} \leq s_{\mu}^{- M}(1+r)^{M}  \leq s_{\mu}^{- M}(r+r)^{M}  \leq 2^M s_{\mu}^{- M}r^{M}.
	\end{align*}
	For the terms with $\lambda$'s, we have $( 1+ s_{\lambda}^{-1}r)^{N_1} \geq s_\lambda^{- N_1} r^{N_1}$ and $\min\left\{1, s_{\lambda}^{-1}(1+r)\right\}^{M} \leq 1$. The integral $\mathcal{S}_2$
    hence satisfies
    \begin{align*}
    \mathcal{S}_2\lesssim s_\mu^{-d} s_\lambda^{  N_1} s_\mu^{- M} \int_1^{s_\mu} r^{M - N_1 +d-1} \,dr \lesssim  s_\mu^{-(M+d)}  s_\lambda^{  N_1} s_\mu^{M+d-N_1} = \Big( \frac{s_\mu}{s_\lambda} \Big)^{-N_1},
    \end{align*}
    where it was used that $M+d>N_1$, which implies $M+d-N_1-1>-1$, for the integration.
    By assumption $N_1 \geq  A + \frac{1+\alpha(d-1)}{2}$, giving the desired result.

\paragraph*{\underline{$s_\mu \leq r $}}
We estimate both terms like the $\lambda$-terms above to obtain
\begin{align*}
    \mathcal{S}_3\lesssim s_\mu^{-d} s_\lambda^{N_1} s_\mu^{N_1} \int_{s_\mu}^\infty r^{ d-1- 2N_1 } \, dr \lesssim  s_\mu^{N_1-d} s_\lambda^{N_1} s_\mu^{d-2N_1} \lesssim \Big( \frac{s_\mu}{s_\lambda} \Big)^{-N_1}.
    \end{align*}

The integral converges since $N_1>\frac{d}{2}$.
Since $N_1 \geq  A + \frac{1+\alpha(d-1)}{2}$ the proof is finished.
\end{proof}

\subsection{Cancellation Estimate}

Theorem~\ref{thm:almorth} provides estimates for the scalar products of $\alpha$-molecules.  To derive them
we evaluate these scalar products on the Fourier side, where
we can take advantage of cancellation phenomena.
Technically, the method is based on a clever integration by parts involving the following differential operator, depending on $\lambda\in\Lambda$, $\mu\in\Delta$,
\begin{align}\label{eq:diffop}
	\mathscr{L}_{\lambda,\mu} := \mathcal{I} - s_0^{2\alpha} \Delta -  \frac{s_0^{2}}{1 + s_0^{2(1-\alpha)}\abspi{d_{\sph}(e_\lambda, e_\mu)}^2}\sprod{e_\lambda, \nabla}^2,
\end{align}
where $s_0=\min\{s_\lambda,s_\mu \}$, $\mathcal{I}$ is the identity operator, $\nabla$ the gradient and $\Delta$ the standard Laplacian.

Lemma~\ref{Lwirkung} shows how $\mathscr{L}_{\lambda,\mu}$ acts on products of functions $\afunc$, $\bfunc$ which satisfy \eqref{eq:molcon}.

\begin{lemm} \label{Lwirkung}
Let $a_{\lambda}$ and $b_{\mu}$ satisfy \eqref{eq:molcon} for every multi-index $\rho\in\N_0^d$ with $|\rho|_1 \leq L$ and assume $s_\lambda,s_\mu\ge c>0$.
Then we can write the expression
\begin{align*}
	\mathscr{L}_{\lambda,\mu}\left(  a_{\lambda}(\matricol{\lambda} \xi)b_{\mu}(\matricol{\mu}\xi) \right)
\end{align*}
as a finite linear combination of terms of the form
\begin{align*}
	p_{\lambda}(\matricol{\lambda}\xi)q_{\mu}(\matricol{\mu}\xi)
\end{align*}
with functions $p_{\lambda},q_{\mu}$, which satisfy \eqref{eq:molcon} for all multi indices $\rho\in\N_0^d$ with $|\rho|_1\leq L-2$.
\end{lemm}

\begin{proof}
For convenience we introduce the operators $O_\lambda:=\matricol{\lambda}$ and $O_\mu:=\matricol{\mu}$.
Further, we define the functions $\widetilde{a}_{\lambda}(\xi):=a_{\lambda}(O_\lambda \xi)$ and $\widetilde{b}_{\mu}(\xi):=b_{\mu}(O_\mu \xi)$.
We also abbreviate $\xi_\lambda:=O_\lambda \xi$ and $\xi_\mu:=O_\mu \xi$.
Taking into account $s_\lambda\gtrsim 1$, we observe
$
\|O_\lambda \|_{2\rightarrow 2}=\| A^{-1}_{\alpha,s_\lambda} \|_{2\rightarrow 2} = \max\{s_\lambda^{-\alpha}, s_\lambda^{-1}\} \lesssim s_\lambda^{-\alpha}.
$
Analogously, it holds $\|O_\mu \|_{2\rightarrow 2}=\| A^{-1}_{\alpha,s_\mu} \|_{2\rightarrow 2} \lesssim s_\mu^{-\alpha}$. 	Finally, we introduce the `transfer' matrix
\begin{align}\label{eq:transfermat}
T_{\lambda,\mu}:= R_{\theta_\mu}R_{\varphi_\mu}R^T_{\varphi_\lambda}R^T_{\theta_\lambda} \in \mathcal{O}(d,\R).
\end{align}

After these remarks we turn to the proof, where we treat the components of $\mathscr{L}_{\lambda,\mu}$ separately.

\paragraph{\bf\underline{$\mathcal{I}$}} This term causes no pain.

\paragraph{\bf\underline{$s_0^{2\alpha} \Delta$}}

By the product rule we have
\begin{align*}
	\Delta (\widetilde{a}_{\lambda} \widetilde{b}_{\mu})
= \underbrace{ 2 \sprod{ \nabla \widetilde{a}_{\lambda}, \nabla  \widetilde{b}_{\mu}}}_{\bf A}
+ \underbrace{\widetilde{a}_{\lambda} \Delta  \widetilde{b}_{\mu}+  \widetilde{b}_{\mu} \Delta \widetilde{a}_{\lambda}}_{\bf B}  .
\end{align*}
In the following we first treat part {\bf A} and then part {\bf B}.

\paragraph*{\bf A}

The chain rule yields $\nabla \widetilde{a}_\lambda (\xi) = O^T_{\lambda}\nabla a_{\lambda}(\xi_\lambda)$ for every $\xi\in\R^d$ and an analogous formula for $\widetilde{b}_{\mu}$.
Thus we obtain
\[
\big\langle \nabla\widetilde{a}_{\lambda}(\xi) , \nabla\widetilde{b}_{\mu}(\xi) \big\rangle =	\big\langle O^T_{\lambda}\nabla a_{\lambda}(\xi_\lambda),  O^T_{\mu}\nabla b_{\mu}(\xi_\mu) \big\rangle
= \big\langle O_{\mu}O^T_{\lambda}\nabla a_{\lambda}(\xi_\lambda),  \nabla b_{\mu}(\xi_\mu) \big\rangle.
\]
The expression $\big\langle O_{\mu}O^T_{\lambda}\nabla a_{\lambda},  \nabla b_{\mu} \big\rangle$ is a linear combination of the products $\partial_i a_{\lambda} \partial_j b_{\mu}$,
where $i,j\in\{1,\ldots,d\}$, with the entries of the matrix $O_\mu O^T_\lambda$ as coefficients.
The functions $\partial_i a_{\lambda}$ and $\partial_j b_{\mu}$ clearly satisfy \eqref{eq:molcon} for every $\rho\in\N_0^d$ with $|\rho|_1 \leq L-1$.
Moreover, the entries of the matrix $O_\mu O^T_\lambda$ are bounded in modulus by $\|O_\mu O^T_\lambda \|_{2\rightarrow 2}$, which in turn obeys the estimate
\[
  \|O_\mu O^T_\lambda \|_{2\rightarrow 2}
	=\|A^{-1}_{\alpha,s_\mu}T_{\lambda,\mu} A^{-1}_{\alpha,s_\lambda} \|_{2\rightarrow 2}
    \leq \|A^{-1}_{\alpha,s_\mu}\|_{2\rightarrow 2} \|A^{-1}_{\alpha,s_\lambda}\|_{2\rightarrow 2} \lesssim (s_\mu s_\lambda)^{-\alpha} \leq s_0^{-2\alpha},
    \]
where $s_0=\min\{s_\lambda,s_\mu \}$. This shows that the function $s_0^{2\alpha}{\bf A}$ can be written as claimed.

  \paragraph*{\bf B}

  Due to symmetry it suffices to treat the term $\widetilde{b}_{\mu} \Delta \widetilde{a}_{\lambda}$.
  Since $\widetilde{b}_{\mu}(\xi)=b_{\mu}(\xi_\mu)$ for $\xi\in\R^d$ and since $b_{\mu}$ fulfills condition \eqref{eq:molcon} for every $\rho\in\N_0^d$ with
  $|\rho|_1\le L$, the function $b_{\mu}$ is a suitable first factor with the required properties. Let us investigate the second factor $\Delta \widetilde{a}_{\lambda}$.

The second derivative of $\widetilde{a}_\lambda$ is at each $\xi\in\R^d$ a bilinear mapping $\R^d \times \R^d \to \R$, which by the chain rule satisfies for $v,w\in\R^d$
\begin{align*}
	\widetilde{a}^{\:\prime\prime}_{\lambda}(\xi)[v,w] =
 a^{\prime\prime}_{\lambda}(\xi_\lambda) [O_{\lambda}v,O_{\lambda} w].
\end{align*}

\noindent
Thus, we have the expansion
\begin{align*}
	\Delta \widetilde{a}_{\lambda}(\xi) = \sum_{i=1}^d 	\widetilde{a}^{\prime\prime}_{\lambda}(\xi) [e_i,e_i] = \sum_{i=1}^d 	a^{\prime\prime}_{\lambda}(\xi_\lambda) [O_\lambda e_i,O_\lambda e_i].
\end{align*}

Let $\rho\in\N_0^d$ be a multi-index with $|\rho|_1 \leq L-2$. Then the partial derivative with respect to $\rho$ of the function
$\xi\mapsto \sum_{i=1}^d s_0^{2\alpha} a^{\prime\prime}_{\lambda}(\xi) [O_\lambda e_i,O_\lambda e_i]$ clearly exists.
It remains to prove the frequency localization \eqref{eq:molcon}.

In view of $\partial^\rho (a_\lambda^{\prime\prime}) = (\partial^\rho a_\lambda)^{\prime\prime}$ we can estimate for every $i\in\{1,\ldots,d\}$ and every $\xi\in\R^d$
\begin{align*}
s^{2\alpha}_0	| \partial^\rho a^{\prime\prime}_{\lambda}(\xi)[O_\lambda e_i, O_\lambda e_i] | \leq s^{2\alpha}_0 \bimapnorm{  \partial^\rho a_{\lambda}''(\xi)} \| O_{\lambda}\|_{2\rightarrow 2}^2
\lesssim  \bimapnorm{  \partial^\rho a_{\lambda}''(\xi) }.
\end{align*}
The norm of the bilinear mapping is given by $\bimapnorm{  \partial^\rho a_{\lambda}''(\xi)}= \sup_{|v|,|w|=1} | \partial^\rho a''_\lambda(\xi)[v,w] |$.
This is equal to the spectral norm of the corresponding Hesse matrix.
Therefore we can deduce
$\bimapnorm{  \partial^\rho a_{\lambda}^{\prime\prime}(\xi) }\lesssim \sup_{|\beta|_1=2} \abs{\partial^{\beta} \partial^\rho a_{\lambda}(\xi)}$.
The functions $\partial^{\beta}\partial^\rho a_{\lambda}$ satisfy \eqref{eq:molcon} for every $\beta\in\N_0^d$ with $|\beta|_1 = 2$ due to the assumption on $a_{\lambda}$.
The required frequency localization follows.

\paragraph*{\bf\underline{$s_0^{2}(1+ s_0^{2(1-\alpha)}\abspi{d_\sph(e_\lambda, e_\mu)}^2)^{-1}\sprod{e_\lambda, \nabla}^2$}}

First we put $w_1:=s_0^2$, $w_2:=s_0^{2\alpha} \abspi{d_\sph(e_\lambda, e_\mu)}^{-2} $, and $w_3:=s_0^{1+\alpha}  \abspi{d_\sph(e_\lambda, e_\mu)}^{-1}$
and notice that the pre-factor satisfies
\begin{align}\label{eq:prefactor}
s_0^{2}(1+ s_0^{2(1-\alpha)}\abspi{d_\sph(e_\lambda, e_\mu)}^2)^{-1} \le \min\{ w_1, w_2, w_3  \}.
\end{align}
The first two estimates are obvious. For the third, recall that $1+t^2\geq 2t$ for all $t\in\R$. Hence,
\begin{align*}
	s_0^{2} (1 + s_0^{2(1-\alpha)} \abspi{d_\sph(e_\lambda, e_\mu)}^2 )^{-1}
\leq  \textstyle{\frac{1}{2}} s_0^{2} (s_0^{1-\alpha}  \abspi{d_\sph(e_\lambda, e_\mu)})^{-1}  \le  s_0^{1+\alpha} \abspi{d_\sph(e_\lambda, e_\mu)}^{-1}.
\end{align*}

We begin with the product rule, which yields
\begin{align}\label{eq:prodrule}
\sprod{e_\lambda, \nabla}^2 \big( \widetilde{a}_{\lambda} \widetilde{b}_{\mu} \big) = \widetilde{b}_{\mu} \langle e_\lambda,\nabla \rangle^2 \widetilde{a}_{\lambda}
      + 2 (\langle e_\lambda,\nabla \rangle \widetilde{a}_{\lambda}) (\langle e_\lambda,\nabla \rangle \widetilde{b}_{\mu})
      + \widetilde{a}_{\lambda} \langle e_\lambda,\nabla \rangle^2  \widetilde{b}_{\mu}.
\end{align}
Recall that $e_\lambda=R^T_{\varphi_\lambda}R^T_{\theta_\lambda} e_d$.
We calculate with the chain rule for $\xi\in\R^d$
\begin{align*}
	\big\langle e_\lambda, \nabla\widetilde{a}_{\lambda} (\xi) \big\rangle =\big\langle O_{\lambda} e_\lambda, \nabla a_{\lambda} (\xi_\lambda) \big\rangle
	= \big\langle A^{-1}_{\alpha,s_\lambda}e_d, \nabla a_{\lambda} ( \xi_\lambda ) \big\rangle = s_\lambda^{-1}\partial_d a_{\lambda}(\xi_\lambda),
\end{align*}
where we used $O_\lambda e_\lambda= A^{-1}_{\alpha,s_\lambda}e_d$. Using the `transfer' matrix $T_{\lambda,\mu}$ from \eqref{eq:transfermat},
we similarly obtain
\begin{align*}
	\big\langle e_\lambda , \nabla \widetilde{b}_{\mu} (\xi)  \big\rangle  =\big\langle O_{\mu} e_\lambda, \nabla b_{\mu}  ( \xi_\mu ) \big\rangle
	= \big\langle A^{-1}_{\alpha,s_\mu} T_{\lambda,\mu} e_d, \nabla b_{\mu} ( \xi_\mu ) \big\rangle.
\end{align*}

Next, we note that $\langle e_\lambda,\nabla \rangle^2 \widetilde{a}_\lambda(\xi)= \widetilde{a}_\lambda^{\prime\prime}(\xi)[e_\lambda,e_\lambda]$. Together with the chain rule, this implies
\begin{align*}
\langle e_\lambda,\nabla \rangle^2 \widetilde{a}_\lambda (\xi) = a_\lambda^{\prime\prime} (\xi_\lambda)[O_\lambda e_\lambda, O_\lambda e_\lambda]=a^{\prime\prime}_\lambda (\xi_\lambda)[A^{-1}_{\alpha,s_\lambda} e_d, A^{-1}_{\alpha,s_\lambda} e_d ]= s^{-2}_\lambda \partial^2_d a_\lambda (\xi_\lambda).
\end{align*}
We also obtain
\begin{align*}
%\langle e_\lambda,\nabla\rangle^2 \widetilde{a}_{\lambda}(\xi)&= a^{\prime\prime}_\mu (\xi_\lambda)[A^{-1}_{\alpha,s_\lambda} e_d, A^{-1}_{\alpha,s_\lambda} e_d ]= s^{-2}_\lambda \partial^2_d a_\lambda (\xi_\lambda), \\
\langle e_\lambda,\nabla \rangle^2 \widetilde{b}_{\mu}(\xi) = b^{\prime\prime}_\mu (\xi_\mu) [ O_\mu e_\lambda,  O_\mu e_\lambda ]
=  \big(\sprod{A^{-1}_{\alpha,s_\mu} T_{\lambda,\mu} e_d, \nabla}^2  b_{\mu} \big)(\xi_\mu).
\end{align*}

Let us henceforth use the abbreviation $\eta:=T_{\lambda,\mu} e_d\in \mathbb{S}^{d-1}$.
Plugging the above calculations into \eqref{eq:prodrule} leads to the following expression for $\sprod{e_\lambda, \nabla}^2 \big( \widetilde{a}_{\lambda} \widetilde{b}_{\mu} \big)(\xi)$ at $\xi\in\R^d$
\begin{align}\label{eq:lincomb}
s_\lambda^{-2}b_{\mu}(\xi_\mu) \cdot \partial_d^2a_{\lambda}(\xi_\lambda)
+  2 s_\lambda^{-1} \partial_d a_{\lambda}(\xi_\lambda) \cdot \sprod{A^{-1}_{\alpha,s_\mu} \eta, \nabla  b_{\mu}(\xi_\mu) }
+ a_{\lambda}(\xi_\lambda) \cdot \big(\sprod{A^{-1}_{\alpha,s_\mu} \eta, \nabla}^2  b_{\mu} \big)(\xi_\mu).
\end{align}

For the first summand of \eqref{eq:lincomb} we consider the product of the functions $s^{-2}_\lambda\partial^2_d a_\lambda$ and $b_\mu$.
Since $s_\lambda^{-2}\le s_0^{-2}$ and in view of \eqref{eq:prefactor} the pre-factor $w_1$ is compensated. Due to the assumptions on $\afunc$ and $\bfunc$ the product is thus of the desired form.

Let us put $\eta_{[d-1]}:=(\eta_1,\ldots,\eta_{d-1},0)^T\in\R^d$ and
$\eta_{[d]}:=(0,\ldots,0,\eta_{d})^T\in\R^d$ and observe that
\[
A^{-1}_{\alpha,s_\mu} \eta = A^{-1}_{\alpha,s_\mu} ( \eta_{[d-1]} + \eta_{[d]} ) = s_\mu^{-\alpha }  \eta_{[d-1]} + s_\mu^{-1}\eta_{[d]}.
\]
The second summand of \eqref{eq:lincomb} then becomes -- up to the factor $2$ --
\[
\partial_d a_{\lambda}(\xi_\lambda) \cdot \big( s_\lambda^{-1}s_\mu^{-\alpha} \sprod{  \eta_{[d-1]} , \nabla  b_{\mu}(\xi_\mu) }
+ s_\lambda^{-1}s_\mu^{-1} \eta_d \partial_d b_{\mu}(\xi_\mu) \big).
\]
We choose the function $\partial_d a_\lambda$ as the first factor, which clearly has the required properties, and the function
\begin{align*}
\xi\mapsto	s_\lambda^{-1}s_\mu^{-\alpha} \langle \eta_{[d-1]}, \nabla b_\mu(\xi)  \rangle  +  s^{-1}_\lambda s_\mu^{-1} \eta_{d} \partial_d b_\mu(\xi) .
\end{align*}
as the second factor.
The second component of this function causes no problems because $|\eta_d|\le 1$ and the pre-factor $w_1$ is compensated due to $(s_\lambda s_\mu)^{-1}\le s^{-2}_0$.
To deal with the other term, notice that by Lemma~\ref{lem:raygeo}
$ |\eta_{[d-1]}|=|\eta|_{[d-1]} \asymp \abspi{d_\sph(e_\lambda, e_\mu)} $. Thus
\[
s^{-1}_\lambda s^{-\alpha}_\mu  |\sprod{  \eta_{[d-1]} , \nabla  b_{\mu} }| {\lesssim} s^{-1}_\lambda s^{-\alpha}_\mu \abspi{d_\sph(e_\lambda, e_\mu)} |\nabla  b_{\mu}|.
\]
The fact that $\partial_i \bfunc$, $i\in\{1,\ldots,d\}$, satisfy \eqref{eq:molcon} by assumption, and
that $s^{-1}_\lambda s^{-\alpha}_\mu \abspi{d_\sph(e_\lambda, e_\mu)} $ compensates $w_3$, implies that also the first component satisfies the required properties.

Let us turn to the last summand of \eqref{eq:lincomb}. The first factor $a_\lambda$ is of the desired form. For the second factor
we expand the function $\sprod{A^{-1}_{\alpha,s_\mu} \eta, \nabla}^2  b_{\mu}$ in the form
\[
s^{-2\alpha}_\mu \langle \eta_{[d-1]}, \nabla \rangle^2 b_{\mu}  + 2 s^{-1-\alpha}_\mu \eta_d \langle\eta_{[d-1]}, \nabla \rangle \partial_d b_{\mu} + s^{-2}_\mu \eta^2_d \partial^2_d b_{\mu}.
\]
Its partial derivatives of order $\rho\in\N_0^d$ with $|\rho|_1\le L-2$ clearly exist, and we get the estimate
\[
|\sprod{A^{-1}_{\alpha,s_\mu} \eta, \nabla}^2 \partial^\rho b_{\mu}| \lesssim s_0^{-2\alpha} |d_{\mathbb{S}}(e_\lambda,e_\mu)|^2 \sum_{i,j=1}^{d-1} |\partial_i\partial_j \partial^\rho b_{\mu}|  +
2 s_0^{-1-\alpha} |d_{\mathbb{S}}(e_\lambda,e_\mu)|  | \nabla \partial_d \partial^\rho b_{\mu}|  + s_0^{-2} |\partial^2_d \partial^\rho b_{\mu}|.
\]
Here we again used that $|\eta_{[d-1]}|\asymp \abspi{d_\sph(e_\lambda, e_\mu)}$ according to Lemma~\ref{lem:raygeo}. This estimate completes the proof, taking into account the estimate~\eqref{eq:prefactor} of the pre-factor and the fact that the partial derivatives of $b_{\mu}$ up to order $L$ satisfy \eqref{eq:molcon}.
\end{proof}

\subsection{Proof of Theorem~\ref{thm:almorth}}

At last we have all the tools available to prove Theorem~\ref{thm:almorth}.
Write $\Delta x = x_\mu - x_\lambda$. An application of the Plancherel identity yields
\begin{align*}
\sprod{m_\lambda, p_\mu} &= \sprod{\hat{m}_\lambda, \hat{p}_\mu} \\
&=  (s_\lambda s_\mu)^{-\frac{\alpha (d-1)+1}{2}} \int_{\R^d} \afuncfor(\matricol{\lambda}\xi) \overline{\bfuncfor(\matricol{\mu}\xi)} \exp(2 \pi i \sprod{\xi, \Delta x}) \, d\xi
\end{align*}
for two $\alpha$-molecules $m_\lambda$ and $p_\mu$ with respective generators $a_{\lambda}$ and $b_{\mu}$.
According to Lemma~\ref{alem:forParaEst} the Fourier transforms of the generators therefore satisfy \eqref{paraproprad} for every $\rho\in\N_0^d$ with $|\rho|_1 \leq L$

Next, we want to exploit cancellation. For this we utilize the
differential operator $\mathscr{L}_{\lambda,\mu}$ from \eqref{eq:diffop}. First, we observe that partial integration yields
\begin{align*}
	\Big\langle \mathscr{L}_{\lambda,\mu}^N & \exp(2 \pi i \sprod{\xi, \Delta x}),\afuncfor(\matricol{\lambda}\xi) \overline{\bfuncfor(\matricol{\mu}\xi)} \Big\rangle \\
	&= \Big\langle\exp(2 \pi i \sprod{\xi, \Delta x}),\mathscr{L}_{\lambda,\mu}^N\big(\afuncfor(\matricol{\lambda}\xi) \overline{\bfuncfor(\matricol{\mu}\xi)}\big) \Big\rangle,
\end{align*}
since the boundary terms vanish due to the decay properties~\eqref{eq:molcon} of the generators and its derivatives. Note that we assume $N_1>d/2$ and $L\ge2N$.
Second, we calculate for $\xi\in\R^d$
\begin{align*}
 \mathscr{L}_{\lambda,\mu}^N \big(\exp( 2\pi i \sprod{\xi, \Delta x})\big)
 = \left( 1 + 4\pi^2 s_0^{2\alpha} \abs{\Delta x}^2 + \frac{4\pi^2s_0^{2}\sprod{e_\lambda, \Delta x}^2}{1 + s_0^{2(1-\alpha)}\abspi{d_\sph(e_\lambda, e_\mu)}^2}\right)^N\exp(2 \pi i\sprod{\xi, \Delta x}).
 \end{align*}
Consequently, we have
 \begin{align*}
 \sprod{ m_\lambda, p_\mu}
 =  \left( 1 +4\pi^2 s_0^{2\alpha} \abs{\Delta x}^2
 + \frac{4\pi^2s_0^{2}\sprod{e_\lambda, \Delta x}^2}{1 + s_0^{2(1-\alpha)}\abspi{d_\sph(e_\lambda, e_\mu)}^2}\right)^{-N} \cdot\, \mathcal{S}_{\lambda,\mu},
 \end{align*}
 with
 \begin{align*}
 \mathcal{S}_{\lambda,\mu}:=  (s_\lambda s_\mu)^{-\frac{\alpha (d-1)+1}{2}}
 \int\limits_{\R^d}  \mathscr{L}_{\lambda,\mu}^N\big(\afuncfor(\matricol{\lambda}\xi) \overline{\bfuncfor(\matricol{\mu}\xi)}\big) \exp(2 \pi i \sprod{\xi, \Delta x}) \,d\xi.
 \end{align*}

 Since $L \geq 2N$ by assumption, Lemma~\ref{Lwirkung} can iteratively be applied $N$ times, and we conclude that
 \begin{align*}
 	\mathscr{L}_{\lambda,\mu}^N\big( \afuncfor(\matricol{\lambda}\xi)\overline{\bfuncfor(\matricol{\mu}\xi)}\big)
 \end{align*}
 can be written as a finite linear combination of terms of the form
 \[
 p_{\lambda}(\matricol{\lambda}\xi)q_{\mu}(\matricol{\mu}\xi)),
 \]
 where $p_{\lambda}$ and $q_{\mu}$ satisfy \eqref{eq:molcon} (for the multi-index just containing zeros).

 Using Lemma~\ref{alem:forParaEst} and putting $K=2N+d-2\le N_2$ in Lemma~\ref{zerlegungslemma} then yields
 \begin{align*}
 	\big|\mathscr{L}_{\lambda,\mu}^N\big( \afuncfor(\matricol{\lambda}\xi)&\overline{\bfuncfor(\matricol{\mu}\xi)}\big)\big|\\
 	&\lesssim S_{\lambda, M-(2N+d-2), N_1, 2N+d-2}(\xi)S_{\mu, M - (2N+d-2), N_1, 2N+d-2}(\xi).
 \end{align*}
 Due to the assumptions, we can further choose a number $\widetilde{N}\le N_1$ which satisfies
 	\begin{align}\label{eq:auxnumber}
 		(M - (2N +d-2))+d > \widetilde{N} \geq N + \frac{1+\alpha(d-1)}{2}.
 	\end{align}
 Since $\widetilde{N}\le N_1$ we have the estimate $ S_{\eta, M -(2N+d-2), N_1, 2N+d-2} \leq S_{\eta, M -(2N+d-2), \widetilde{N}, 2N+d-2}$ for $\eta=\lambda$, $\mu$.
 Hence, we obtain
 \begin{align*}
 |\mathcal{S}_{\lambda,\mu}|&\lesssim  (s_\lambda s_\mu)^{-\frac{\alpha (d-1)+1}{2}} \int_{\R^d}  S_{\lambda, M-(2N+d-2), N_1, 2N+d-2)}(\xi)S_{\mu, M - (2N+d-2), N_1, 2N+d-2}(\xi)  \,d\xi \\
    &\lesssim  (s_\lambda s_\mu)^{-\frac{\alpha (d-1)+1}{2}} \int_{\R^d}  S_{\lambda, M -(2N+d-2), \widetilde{N}, 2N+d-2}(\xi)S_{\mu, M -(2N+d-2), \widetilde{N}, 2N+d-2}(\xi)  \,d\xi \\
    &\lesssim  \max\left\{ \frac{s_\lambda}{s_\mu}, \frac{s_\mu}{s_\lambda} \right\}^{-N}(1 + s_0^{(1-\alpha)}\abspi{d_\sph(e_\lambda, e_\mu)})^{-2N}.
 \end{align*}
 Here we used \eqref{eq:auxnumber} and Lemma~\ref{intWholeEstimate} in the last line (using this $S$ and
 	setting $\tilde{M}= M-(2N+d-2)$,  $A = N$ and $B=2N$  ($B>1$, $A>0$ since $N>1$)).

  Altogether, we arrive at the desired estimate
 \begin{align*}
 |\sprod{m_\lambda, p_\mu}|&\lesssim \max\left\{ \frac{s_\lambda}{s_\mu}, \frac{s_\mu}{s_\lambda} \right\}^{-N} \left( 1 + s_0^{2\alpha} \abs{\Delta x}^2 + \frac{s_0^{2}\sprod{e_\lambda, \Delta x}^2}{1 + s_0^{2(1-\alpha)}\abspi{d_\sph(e_\lambda, e_\mu)}^2}\right)^{-N} \big(1 + s_0^{(1-\alpha)}\abspi{d_\sph(e_\lambda, e_\mu)}\big)^{-2N} \\
 	&\lesssim
 	 \max \left\{ \frac{s_\lambda}{s_\mu},  \frac{s_\mu}{s_\lambda} \right\}^{-N} \left( 1 + s_0^{2(1-\alpha)} \abspi{d_\sph(e_\lambda, e_\mu)}^2
 	+ s_0^{2\alpha} \abs{\Delta x}^2 + \frac{s_0^{2} \sprod{e_\lambda, \Delta x}^2}{1 + s_0^{2(1-\alpha}\abspi{d_\sph(e_\lambda, e_\mu)}^2} \right)^{-N} \\
  & \lesssim \omega_{\alpha}(\lambda, \mu)^{-N}.
 \end{align*}
 For the last estimate observe that the inequality between the arithmetic and the geometric mean
 \begin{align*}
 	\big( 1 + s_0^{2(1-\alpha)} \abspi{d_\sph(e_\lambda, e_\mu)}^2 \big) + \frac{s_0^{2}\sprod{e_\lambda, \Delta x}^2}{1 + s_0^{2(1-\alpha)}\abspi{d_\sph(e_\lambda, e_\mu)}^2}
 	\ge 2s_0|\langle e_\lambda,\Delta x\rangle|
 \end{align*}
 implies
 \begin{align*}
 	1 &+ s_0^{2(1-\alpha)} \abspi{d_\sph(e_\lambda, e_\mu)}^2 + s_0^{2\alpha} \abs{\Delta x}^2 + \frac{s_0^{2} \sprod{e_\lambda, \Delta x}^2}{1 + s_0^{2(1-\alpha)}\abspi{d_\sph(e_\lambda, e_\mu)}^2}
 	 \geq \\
 	  &\frac{1}{2} \left( 1 + s_0^{2(1-\alpha)} \abspi{d_\sph(e_\lambda, e_\mu)}^2 + s_0^{2\alpha}\abs{\Delta x}^2 \right)
 	  + \frac{1}{2} \left( 1 + s_0^{2(1-\alpha)} \abspi{d_\sph(e_\lambda, e_\mu)}^2 + \frac{s_0^{2}\sprod{e_\lambda, \Delta x}^2}{1 + s_0^{2(1-\alpha)}\abspi{d_\sph(e_\lambda, e_\mu)}^2} \right) \\
 	  & \quad \quad \quad \quad \quad \quad \gtrsim  1 +  s_0^{2(1-\alpha)}\abspi{d_\sph(e_\lambda, e_\mu)}^2 + s_0^{2\alpha}\abs{\Delta x}^2 + s_0\abs{\sprod{e_\lambda, \Delta x}}
     = 1 + d_\alpha(\lambda,\mu).
 \end{align*}
 This concludes the proof. \qed

\section*{Acknowledgements}

The first author acknowledges support by Deutsche
Forschungsgemeinschaft (DFG) Grant KU 1446/18~-~1 and by the Deutscher Akademischer Austausch Dienst (DAAD).
The second author would like to thank Anton Kolleck for enlightening discussions on this and related topics.
Moreover, both authors would like to thank Gitta Kutyniok for carefully proofreading the article and providing many suggestions for enhancing the readability of the paper.

\bibliographystyle{abbrv}
\bibliography{molecules1}

\end{document}